\documentclass{article}

\usepackage[english]{babel}
\usepackage{amsthm,amssymb}
\usepackage[letterpaper, top=2cm,bottom=2cm,left=3cm,right=3cm,marginparwidth=1.75cm]{geometry}
\usepackage{amsmath}
\usepackage{graphicx}
\usepackage[colorlinks=true, allcolors=blue]{hyperref}
\usepackage{authblk}
\usepackage{caption}
\usepackage{subcaption}
\usepackage{tikz}
\usepackage{comment}
\usepackage{ulem}

\newtheorem{theorem}{Theorem}
\newtheorem{observation}{Observation}

\newtheorem{claim}{Claim}

\newtheorem{lemma}{Lemma}

\newcommand{\qedclaim}{\hfill $\diamond$ \medskip}

\newcommand{\x}{\square}

\newcommand{\red}[1]{{\textcolor{red}{#1}}}
\newcommand{\blue}[1]{{\textcolor{blue}{#1}}}

\title{The Graph Coloring Game on $4\times n$-Grids\thanks{Research supported by the CAPES-Cofecub project Ma 1004/23, by the Inria Associated Team CANOE, by the research grants
    ANR P-GASE, ANR DiGraphs and by the French government, through the EUR DS4H Investments in the Future project managed by the National Research Agency (ANR) with the reference number ANR-17-EURE-0004.}}
\author[1,4]{Caroline Brosse}
\author[2]{Nicolas Martins}
\author[1]{Nicolas Nisse}
\author[3]{Rudini Sampaio}

\affil[1]{Universit\'{e} C\^{o}te d’Azur, Inria, CNRS, I3S, France.}
\affil[2]{Inst. de Eng. e Desenvolvimento Sustent\'{a}vel, UNILAB, Brazil.}
\affil[3]{Dept. Computação, Universidade Federal do Cear\'{a},  Brazil.}
\affil[4]{Université d’Orléans, INSA Centre Val de Loire, LIFO EA 4022, Orléans, France}

\date{}

\begin{document}

\maketitle

\begin{abstract}
The graph coloring game is a famous two-player game (re)introduced by Bodlaender in $1991$. Given a graph $G$ and $k \in \mathbb{N}$, Alice and Bob alternately (starting with Alice) color an uncolored vertex with some color in $\{1,\cdots,k\}$ such that no two adjacent vertices receive a same color. If eventually all vertices are colored, then Alice wins and Bob wins otherwise. The game chromatic number $\chi_g(G)$ is the smallest integer $k$ such that Alice has a winning strategy with $k$ colors in $G$. It has been recently (2020) shown that, given a graph $G$ and $k\in \mathbb{N}$,  deciding whether $\chi_g(G)\leq k$ is PSPACE-complete. Surprisingly, this parameter is not well understood even in ``simple" graph classes. Let $P_n$ denote the path with $n\geq 1$ vertices. For instance, in the case of Cartesian grids, it is easy to show that $\chi_g(P_m \x P_n) \leq 5$ since $\chi_g(G)\leq \Delta+1$ for any graph $G$ with maximum degree $\Delta$. However, the exact value is only known for small values of $m$, namely $\chi_g(P_1\x P_n)=3$, $\chi_g(P_2\x P_n)=4$ and $\chi_g(P_3\x P_n) =4$ for $n\geq 4$ [Raspaud, Wu, 2009]. Here, we prove that, for every $n\geq 18$,  $\chi_g(P_4\x P_n) =4$.
\end{abstract}

\section{Introduction}

Given a graph $G=(V,E)$ and $k \in \mathbb{N}$, the {\it $k$-coloring game} involves two players, Alice and Bob, who alternately (starting with Alice) choose an uncolored vertex $v \in V$ and assign to it a color $c(v) \in \{1,\cdots,k\}$ such that, for every colored neighbor $u$ of $v$, $c(u)\neq c(v)$ ({\it i.e.}, at every turn, the partial coloring must be {\it proper}). Alice wins if, eventually, all vertices of $G$ are colored; Bob wins otherwise, {\it i.e.}, if at some turn, there exists an uncolored vertex that has, for each $1 \leq i \leq k$, a neighbor already colored with color $i$. Let $\chi_g(G)$ be the smallest integer $k$ such that Alice has a winning strategy in the $k$-coloring game, {\it i.e.}, Alice can win whatever be the strategy of Bob. Note that, for any graph $G$ with maximum degree $\Delta(G)$, $\chi_g(G) \leq \Delta(G)+1$ since, for every vertex $v \in V$, Alice can always choose a color $c(v)\leq \Delta(G)+1$ for $v$ that does not appear in its neighborhood.

The coloring game was introduced by Brams in the context of map coloring and was described by Gardner in 1981 in his ``Mathematical Games'' column of Scientific American~\cite{gardner81}. It remained unnoticed until Bodlaender \cite{bodlaender91} reinvented it in 1991 as the ``Coloring Construction Game''. It was only recently proved that deciding whether $\chi_g(G)\leq k$ is PSPACE-complete~\cite{costa20}. Thereafter, many games based on classical variants of the coloring problem were proved PSPACE-complete, such as the greedy coloring game \cite{lima22} and the connected coloring game \cite{lima23}.

From the above complexity results, it is natural to investigate the coloring game in more restricted graph classes. This has been done in trees~\cite{faigle1993game,FRS23,DBLP:journals/entcs/FurtadoDFG19,DBLP:journals/rairo/FurtadoPDF23}. In particular, it has been proved that $\chi_g(T)\leq 4$ for every tree $T$~\cite{faigle1993game} and sufficient conditions for trees $T$ ensuring that $\chi_g(T) = 4$ are provided in~\cite{FRS23}. 
In other graph classes, it was proved that $\chi_g(G)\leq 7$ in outerplanar graphs \cite{kierstad94}, $\chi_g(G)\leq 3k+2$ in partial $k$-trees~\cite{zhu00} and $\chi_g(G)\leq 5$ in cacti~\cite{game-cactus07}.
Moreover, $\chi_g(G)\leq 17$~\cite{zhu08} in planar graphs, $\chi_g(G)\leq 13$ in planar graphs with girth at least 4~\cite{sekiguchi14} and $\chi_g(G)\leq 5$ in planar graphs with girth at least 7~\cite{nakprasit18}.

Finally, tori and grids with at most $3$ rows have been considered in~\cite{DBLP:journals/ipl/RaspaudW09}.
Here, we consider the grids with four rows and prove that $\chi_g(P_4 \x P_n)\leq 4$, for every $n\in \mathbb{N}$.

\section{Preliminaries}

Given two graphs $G$ and $H$, $G'=G \x H$ denotes the Cartesian product of $G$ and $H$, that is the graph with vertex set $V(G) \times V(H)$ and, for every $a,b \in V(G)$ and $x,y \in V(H)$, there is an edge $\{(a,x),(b,y)\}$ in $G'$ if and only if $a=b$ and $\{x,y\} \in E(H)$ or if $x=y$ and $\{a,b\} \in E(G)$.
For any $m$ and $n$, the grid of dimension $m \times n$ is isomorphic to the Cartesian product of the two paths $P_m$ and $P_n$.
In what follows, we will then use the notation $P_m \x P_n$ to denote the $m\times n$ grid.
The same holds for the cylinder $P_m\x C_n$ which is the Cartesian product of a path $P_m$ and a cycle $C_n$.

\begin{claim}
    $\chi_g(P_4\x P_n)\leq 5$ and  $\chi_g(P_4\x C_n) \leq 5$
\end{claim}
\begin{proof}
    This is direct since the maximum degree is at most $4$: if five colors can be used, then at least one color is available for any uncolored vertex.
\end{proof}

The notation $v_{i,j}$ is used for the vertex on the row $i$ and column $j$ of a grid.
We consider that the rows are numbered from bottom to top and the columns are numbered from left to right. 
Moreover, in this paper we consider only proper colorings with $4$ colors $\{1,2,3,4\}$.
For any vertex $v$, we denote the color of the vertex $v$ by $c(v) \in \{0,1,2,3,4\}$, the value $c(v) = 0$ standing for an uncolored vertex.

\begin{claim}\label{claim:lower}
Let $n \geq 18$. Then, $\chi_g(P_4\x P_n) \geq 4$ and $ \chi_g(P_4\x C_n) \geq 4$.
\end{claim}
\begin{proof}
Let us consider a game with three available colors.
    Consider the first turn of Bob and let $j$ be a column such no vertex has been colored yet in columns $k \in \{j-4,\cdots,j+3,j+4\}$ (such $j$ exists since $n\geq 18$). Bob colors $v_{2,j}$ with color $1$. W.l.o.g., next move of Alice does not color any vertex in columns $k \in \{j+1,j+2,j+3,j+4\}$. Then, Bob colors $v_{2,j+2}$ with color $c=2$ if Alice has not colored $v_{1,j}$ nor $v_{3,j}$, and Bob colors $v_{2,j+2}$ with color $c \in \{2,3\}$ if Alice has colored $v_{1,j}$ or $v_{3,j}$ with color $c\neq 1$. If, during her next turn, Alice does not color $v_{2,j+1}$, then at least one of $v_{1,j+1}$ or $v_{3,j+1}$ can be colored with the color in $\{2,3\} \setminus \{c\}$ and, now, $v_{2,j+1}$ has three neighbors with distinct colors. Hence, Alice must have colored $v_{2,j+1}$ with the color $c'$ in $\{2,3\} \setminus \{c\}$. If $v_{1,j}$ (resp. $v_{3,j}$) is colored with $c$, then Bob colors $v_{1,j+2}$ (resp. $v_{4,j+1}$) with color $1$ and then $v_{1,j+1}$ (resp. $v_{3,j+1}$) cannot be colored anymore. Otherwise, $v_{1,j}$ and $v_{3,j}$ are not colored and $c=2$ and $c'=3$. Bob colors $v_{4,j+2}$ with color $3$. If Alice does not color $v_{3,j+2}$ then Bob colors  $v_{3,j+1}$ or $v_{3,j+3}$ with $1$ and $v_{3,j+2}$ cannot be colored anymore. Otherwise, Alice colors $v_{3,j+2}$ with $1$ and Bob colors $v_{4,j+1}$ with $2$ and $v_{3,j+1}$ cannot be colored anymore.
\end{proof}

Let $\chi^{\cal B}_g(G)$ be the smallest number of colors ensuring that Alice wins when Bob starts.

\begin{lemma}\label{lem:easy}
Let $n \geq 9$. Then, $\chi^{\cal B}_g(P_4\x P_n) =  \chi^{\cal B}_g(P_4\x C_n) = 4$
\end{lemma}
\begin{proof}
The lower bound is similar to the previous one (except that we only need $n \geq 9$ since Bob starts). 
A winning strategy for Alice proceeds as follows. Each time Bob colors a vertex $v_{a,j}$ with color $c$, Alice colors the (unique) vertex $v_{b,j}$ on column $j$ at distance $2$ from $v_{a,j}$ (i.e., $|a-b|=2$) with the color $c$. It is clear that Alice can always follow this strategy since $c(v_{a,j-1})=c(v_{b,j-1})$,  $c(v_{a,j+1})=c(v_{b,j+1})$ and both neighbors of $v_{b,j}$ in column $j$ have the same color. Hence, whenever a vertex $v$ is being colored, either its two neighbors on the same column are uncolored, or they are colored with the same color, or $v$ has at most three neighbors. In all cases, at most four colors are sufficient.
\end{proof}

Unfortunately, the simple strategy described above does not work when Alice starts. Indeed, let $n\geq 6$ and assume that, after her first move, Alice follows the strategy above. let $i\leq n$ be the column of the first vertex colored by Alice and let $j\leq n$, such that $i \notin \{j-1,j,j+1\}$. Bob colors $v_{1,j-1}$ and $v_{2,j+1}$ with color $1$, $v_{2,j-1}$ and $v_{1,j+1}$ with color $2$. Then, since Alice has started, Bob can ensure that she is the first to color some vertex of column $j$ (unless the game has stopped before).
W.l.o.g., she colors the vertex $v_{a,j}$ with color $3$, then Bob colors the vertex $v_{b,j}$ such that $|a-b|=2$ with color $4$. The common neighbor of $v_{a,j}$ and $v_{b,j}$ will require a fifth color.

\section{Graph coloring game in \texorpdfstring{$P_4 \x P_n$}{P4xPn}}

Let us consider a partial proper $4$-coloring of the grid $G=P_4 \x P_n$. The \emph{columns} of $G$ are the copies of $P_4$ in $P_4 \x P_n$. A column is \emph{empty} if all its vertices are uncolored. 
A \emph{block} is any maximal subgrid with at least one colored vertex in each of its columns. As an exception\footnote{This exception is for technical reasons, just to avoid extra cases in the proof below.}, if the last column is empty and the penultimate column is not, we include the  last column in the block of the penultimate column. 


Note that a block is a subgrid $B$ of the form $P_4 \x P_m$, with $m\leq n$. The \emph{right border} (resp. \emph{left border}) of a block $B$ is its rightmost (resp. leftmost) column. Note that, if the right border (resp. left border) of a block is the $j^{th}$ column of $G$, then no vertex of the $j+1^{th}$ column (resp. of the $j-1^{th}$ column) of $G$ is colored.
When a block consists of a single column (except for the first column), we consider it as a left border. That is, when we mention the right border of a block, this implies that this block has at least two columns  (this will be important in the design of the algorithm and in its proof).
A vertex is \emph{inside} a block $B$ if it is a vertex of $B$ but not of its borders.
A right (resp. left) border of a block at column $1<j<n$ is \emph{free} if columns $j+1$ and $j+2$ (resp. $j-1$ and $j-2$) are empty.

\paragraph{Safe and sound vertices.} Let us now define some status of the vertices that will be important in what follows. Intuitively, a vertex $v$ of a block is called {\it safe} if whatever be the strategies of Alice and Bob (with $4$ colors) starting from this state ({\it i.e.} from the current partial $4$-coloring), the vertex $v$ will eventually be colored at the end of the game. More precisely, a vertex is \emph{safe} if (a) it is already colored, or (b) it has at most three neighbors, or (c) at least two of its neighbors are already colored with the same color, or (d) it has three neighbors colored with distinct colors $c,c',c''$ and its last neighbor is uncolored and adjacent to a vertex colored with the fourth color (which makes it always available for $v$).
In other words, for any uncolored safe vertex, at least one of the four colors is available to color this vertex at any moment of the game.

An unsafe (not safe) vertex $v$ in a block $B$ is \emph{sound} if it is associated with two uncolored safe neighbors $x$ and $y$ in $B$, called the \emph{doctors} of $v$, in such a way that every safe vertex is the doctor of at most one (unsafe) vertex called its \emph{patient}.
The intuition for soundness is that, if Bob colors a doctor $x$ or $y$ of a patient $v$, then, in the next turn, Alice can always use an available color for $v$ and make $v$ safe.

As an example, consider the configuration of Figure \ref{fig-firstex}, showing the middle of a block.
Recall that the rows are numbered $1$ to $4$ from bottom to top.
Notice that every vertex in column $j$ to $j+2$ is safe or sound. For this, note that every vertex in columns $j$ to $j+2$ is safe, except $v_{2,j+1}$ and $v_{3,j+2}$, which are sound, since we can associate the doctors $v_{2,j}$ and $v_{1,j+1}$ to the patient $v_{2,j+1}$, and the doctors $v_{2,j+2}$ and $v_{4,j+2}$ to the patient $v_{3,j+2}$. In this case, the safe uncolored vertex $v_{2,j+2}$ has two unsafe neighbors $v_{2,j+1}$ and $v_{3,j+2}$, but this is not a problem since $v_{2,j+2}$ is the doctor of only one patient ($v_{3,j+2}$).

\begin{figure}[ht]\centering
\begin{tikzpicture}[scale=0.195]
\tikzstyle{vertex}=[draw,circle,fill=black,inner sep=0pt, minimum size=1pt]

\begin{scope}[xshift=0cm, yshift=0cm, scale=2.5]
\foreach \i in {1,2,3,4}{
\path[draw, thin] (\i,0)--(\i,4);}
\foreach \j in {0,1,2,3,4}{
\path[draw, densely dotted] (0,\j)--(5,\j);} 
\foreach \j in {0,1,2,3,4}{ 
\path[draw, thin] (1,\j)--(4,\j);} 
\path[draw, very thick]  (0,0) --(1,0)--(2,0) --(3,0)--(4,0)--(5,0) ;
\path[draw, very thick]  (0,4) --(1,4)--(2,4) --(3,4)--(4,4)--(5,4) ;
\draw (1.5,4.5) node {$j$};
\draw (1.5,1.5) node {$0$};
\draw (1.5,0.5) node {$c$};
\draw (0.5,1.5) node {$c'$};
\draw (1.5,2.5) node {$c'$};
\draw (2.5,0.5) node {$0$};
\draw (2.5,1.5) node {$0$};
\draw (2.5,2.5) node {$c$};
\draw (2.5,3.5) node {$0$};
\draw (3.5,0.5) node {$c$};
\draw (3.5,1.5) node {$0$};
\draw (3.5,2.5) node {$0$};
\draw (3.5,3.5) node {$0$};
\draw (4.5,2.5) node {$0$};
\draw (4.5,1.5) node {$c$};
\end{scope}
\end{tikzpicture}

\caption{Illustration of the notions of safeness and soundness in the middle of a block. Note that a vertex is depicted by a square. The bold lines figure the surrounding of the block. A $0$ in a cell means that the corresponding vertex is not colored, $c$ and $c'$ denote any colors (not $0$). It can be deduced from the figure that $c'\neq c$ since there are two adjacent vertices colored $c$ and $c'$ respectively (and the coloring is proper). A white cell indicates that the status of the corresponding vertex is not constrained (it may be colored or not). The $j$ above denotes the index of the corresponding column.}
\label{fig-firstex}
\end{figure}

The intuition for the proof of the main theorem of this section is the following:
during the game, after Alice's move, every vertex of a block is safe or sound.
Actually, this is not completely sufficient and we will need to consider a few exceptions defined later.

\paragraph{Borders' configurations.} Let us describe some particular configurations for the border of a block. Let $j$ be the index of the border column of the block.

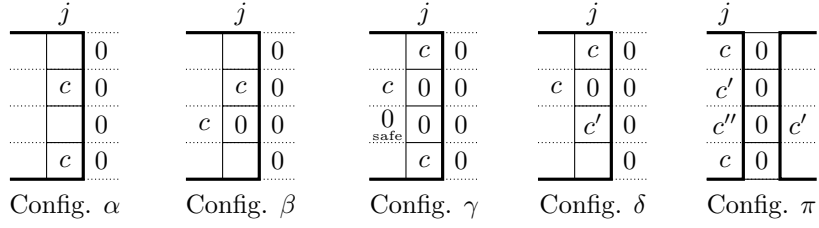
\begin{figure}[hb]\centering
\begin{tikzpicture}[scale=0.195]
\tikzstyle{vertex}=[draw,circle,fill=black,inner sep=0pt, minimum size=1pt]

\begin{scope}[xshift=0cm, yshift=0cm, scale=2.5]
\foreach \i in {2}{
\path[draw, thin] (\i,0)--(\i,4);}
\foreach \j in {0,1,2,3,4}{ 
\path[draw, densely dotted] (1,\j)--(4,\j);}
\foreach \j in {0,1,2,3,4}{
\path[draw, thin] (2,\j)--(3,\j);} 
\path[draw, very thick] (1,0)--(3,0)--(3,4)--(1,4);
\draw (2.5,4.5) node {$j$};
\draw (2.5,0.5) node {$c$};
\draw (2.5,2.5) node {$c$};
\draw (3.5,0.5) node {$0$};
\draw (3.5,1.5) node {$0$};
\draw (3.5,2.5) node {$0$};
\draw (3.5,3.5) node {$0$};
\draw (2.5,-0.75) node {Config. $\alpha$}; 
\end{scope}

\begin{scope}[xshift=12cm, yshift=0cm, scale=2.5]
\foreach \i in {2,3}{
\path[draw, thin] (\i,0)--(\i,4);}
\foreach \j in {0,1,2,3,4}{ 
\path[draw, densely dotted] (1,\j)--(4,\j);} 
\foreach \j in {0,1,2,3,4}{ 
\path[draw, thin] (2,\j)--(3,\j);} 
\path[draw, very thick]  (1,0)--(3,0)--(3,4)--(1,4);
\draw (2.5,4.5) node {$j$};
\draw (1.5,1.5) node {$c$};
\draw (2.5,0.5) node {};
\draw (2.5,1.5) node {$0$};
\draw (2.5,2.5) node {$c$};
\draw (2.5,3.5) node {};
\draw (3.5,0.5) node {$0$};
\draw (3.5,1.5) node {$0$};
\draw (3.5,2.5) node {$0$};
\draw (3.5,3.5) node {$0$};
\draw (2.5,-0.75) node {Config. $\beta$}; 
\end{scope}

\begin{scope}[xshift=22cm, yshift=0cm, scale=2.5]
\foreach \i in {3,4}{
\path[draw, thin] (\i,0)--(\i,4);}
\foreach \j in {0,1,2,3,4}{ 
\path[draw, densely dotted] (2,\j)--(3,\j);
\path[draw, densely dotted] (4,\j)--(5,\j);} 
\foreach \j in {0,1,2,3,4}{ 
\path[draw, thin] (3,\j)--(4,\j);} 
\path[draw, very thick]  (2,0)--(4,0)--(4,4)--(2,4) ;
\draw (3.5,4.5) node {$j$};
\draw (2.5,2.5) node {$c$};
\draw (2.5,1.65) node {$0$};
\draw (2.5,1.2) node {\tiny safe};
\draw (3.5,0.5) node {$c$};
\draw (3.5,1.5) node {$0$};
\draw (3.5,2.5) node {$0$};
\draw (3.5,3.5) node {$c$};
\draw (4.5,0.5) node {$0$};
\draw (4.5,1.5) node {$0$};
\draw (4.5,2.5) node {$0$};
\draw (4.5,3.5) node {$0$};
\draw (3.5,-0.75) node {Config. $\gamma$}; 
\end{scope}

\begin{scope}[xshift=36cm, yshift=0cm, scale=2.5]
\foreach \i in {2,3}{
\path[draw, thin] (\i,0)--(\i,4) ;}
\foreach \j in {0,1,2,3,4}{ 
\path[draw, densely dotted] (1,\j)--(4,\j);} 
\foreach \j in {0,1,2,3,4}{ 
\path[draw, thin] (2,\j)--(3,\j);} 
\path[draw, very thick]  (1,0)--(3,0)--(3,4)--(1,4) ;
\draw (2.5,4.5) node {$j$};
\draw (1.5,2.5) node {$c$};
\draw (2.5,3.5) node {$c$};
\draw (2.5,2.5) node {$0$};
\draw (2.5,1.5) node {$c'$};
\draw (3.5,0.5) node {$0$};
\draw (3.5,1.5) node {$0$};
\draw (3.5,2.5) node {$0$};
\draw (3.5,3.5) node {$0$};
\draw (2.5,-0.75) node {Config. $\delta$}; 
\end{scope}

\begin{scope}[xshift=50cm, yshift=0cm, scale=2.5]
\foreach \i in {1,2}{
\path[draw, thin] (\i,0)--(\i,4);}
\foreach \j in {0,1,2,3}{ 
\path[draw, densely dotted] (0,\j)--(3,\j);} 
\foreach \j in {0,1,2,3,4}{ 
\path[draw, thin] (1,\j)--(2,\j) ;} 
\path[draw, very thick]  (0,0)--(1,0)--(1,4)--(0,4);
\path[draw, very thick] (3,0)--(2,0)--(2,4)--(3,4);
\draw (0.5,4.5) node {$j$};
\draw (0.5,3.5) node {$c$};
\draw (0.5,2.5) node {$c'$};
\draw (0.5,1.5) node {$c''$};
\draw (0.5,0.5) node {$c$};
\draw (1.5,0.5) node {$0$};
\draw (1.5,1.5) node {$0$};
\draw (1.5,2.5) node {$0$};
\draw (1.5,3.5) node {$0$};
\draw (2.8,0.5) node {};
\draw (2.5,1.5) node {$c'$};
\draw (2.5,3.5) node {};
\draw (1.5,-0.75) node {Config. $\pi$}; 
\end{scope}
\end{tikzpicture}

\caption{Illustration of the different configurations in the case of a right border (column $j$).
The colors $c,c',c''\ne 0$ are pairwise distinct.
Each configuration has its symmetrical counterpart (according to the horizontal symmetry axis of the grid). The configurations for the left borders are defined symmetrically (according to the vertical symmetry axis of the grid).}
\label{fig:configs}
\end{figure}

\begin{description}
\item[Configuration $\alpha$.] A right/left border (at column $j$) is in configuration $\alpha$ if $c(v_{1,j})=c(v_{3,j}) \neq 0$ and, moreover, none of $v_{2,j}$ and $v_{4,j}$ are doctors, and if they are colored, then $c(v_{2,j})\neq c(v_{4,j})$ (or, symmetrically, if $c(v_{2,j})=c(v_{4,j}) \neq 0$  and moreover, if none of $v_{1,j}$ and $v_{3,j}$ are doctors and if they are colored, then $c(v_{1,j})\neq c(v_{3,j})$). 

Recall that, if a block consists of  a single column  (except for the first column) in configuration $\alpha$, we consider that it is a left border to avoid ambiguity.

\item[Configuration $\beta$.] A right border (at column $j$) is in configuration $\beta$ if $c(v_{2,j-1})=c(v_{3,j}) \neq 0$ and $c(v_{2,j})=0$ (or, symmetrically, if $c(v_{3,j-1})=c(v_{2,j}) \neq 0$ and $c(v_{3,j})=0$).
The $\beta$ configuration for a left border is defined symmetrically (according to the vertical symmetry axis of the grid).

\item[Configuration $\gamma$.] A right border (at column $j$) is in configuration $\gamma$ if $c(v_{1,j})=c(v_{4,j})=c(v_{3,j-1})=c\neq 0$, $c(v_{2,j-1})=c(v_{2,j})=c(v_{3,j})=0$ and $v_{2,j-1}$ is safe. 
The symmetric according to the horizontal axis of the grid is also a configuration $\gamma$.
Analogously for a left border.

Observe that $v_{2,j-1}$ being safe implies that either $c(v_{2,j-2})=c$ or $c(v_{2,j-2})=c(v_{1,j-1})$.
Therefore $v_{2,j-1}$ is not the doctor of a vertex inside its block, and a block with a border in configuration $\gamma$ contains at least $3$ columns.


\item[Configuration $\delta$.] A right border (at column $j$) is in configuration $\delta$ if $c(v_{3,j-1}) = c(v_{4,j})=c$ and $c(v_{2,j})=c'\neq c$.
The symmetric according to the horizontal axis of the grid is also a configuration $\delta$, and the configuration $\delta$ for a left border is defined symmetrically.

\item[Configuration $\pi$.] A right border  (at column $j$) is in configuration $\pi$ if $c(v_{1,j})=c(v_{4,j})=c$, $c(v_{3,j})=c(v_{2,j+2})=c'$, $c(v_{2,j})=c''$ and $c,c',c''\ne 0$ are pairwise distinct.
The symmetric according to the horizontal axis of the grid is also a configuration $\pi$, and the configuration $\pi$ for a left border is defined symmetrically. 
\end{description}

The five configurations are depicted in Figure~\ref{fig:configs} (in the case of a right border).
In the proof of the next theorem, the notation is the same as in Figure~\ref{fig:configs}, assuming the orientation of the configurations with right border at column $j$,
except when it will be explicitly mentioned that a left border is considered.
In order to avoid ambiguity, if a border is in both configurations $\alpha$ and $\beta$, then we say that it is in configuration $\beta$.

Note that all vertices in the border of a block in configuration $\alpha$, $\beta$, $\delta$ or $\pi$ are safe.
Also note that every vertex of configuration $\gamma$ shown in Figure \ref{fig:configs} is safe, except $v_{2,j}$, which is sound since we can associate the doctors $v_{2,j-1}$ and $v_{3,j}$ with $v_{2,j}$ (their only patient). 

\begin{observation}
Every vertex of a border of a block in configuration $\alpha$, $\beta$, $\gamma$, $\delta$ or $\pi$ is safe or sound.
\end{observation}

Moreover, notice that, in configurations $\beta$, $\delta$ and $\pi$, there is no doctor in the border. In configuration $\gamma$, there is one doctor 
and its patient in the border.
In configuration $\alpha$, there may be a doctor in the border and its possible patient would be inside the block, but, actually, we have added the constraint that no vertex of a border in configuration $\alpha$ can be a doctor.

\paragraph{Particular configurations.} Let us define seven particular configurations: $\Delta$, $\Delta'$, $\Delta'_2$, $\Lambda$, $\Lambda'$, $\Lambda_2$ and $\Lambda_2'$, five of them depicted in Figure~\ref{fig-delta}.
Configurations $\Lambda_2$ and $\Lambda_2'$ are obtained from Configurations $\Lambda$ and $\Lambda'$, respectively, by replacing exactly one 0 (zero) of column $j-2$ by a color different from $c$, with the additional constraint that column $j-3$ cannot be empty.

All seven particular configurations have their symmetrical counterpart according to the horizontal axis.
However, only configurations $\Delta'$ and $\Delta'_2$ have their symmetrical counterpart according to the vertical axis.
That is, configurations $\Lambda, \Lambda', \Lambda_2$ and $\Lambda'_2$ are only defined for column $j$ being their left border, and, in configuration $\Delta$, column $j+2$ will always be its right border.

\begin{figure}[ht]\centering
\begin{tikzpicture}[scale=0.195]
\tikzstyle{vertex}=[draw,circle,fill=black,inner sep=0pt, minimum size=1pt]

\begin{scope}[xshift=0cm, yshift=0cm, scale=2.5]
\fill[red!20] (1,2) rectangle (2,3);
\foreach \i in {1,2,3}{
\path[draw, thin] (\i,0)--(\i,4);}
\foreach \j in {0,1,2,3,4}{ 
\path[draw, densely dotted] (0,\j)--(4,\j) ;} 
\foreach \j in {0,1,2,3,4}{ 
\path[draw, thin] (1,\j)--(3,\j) ;} 
\path[draw, very thick]  (0,0)--(3,0)--(3,4)--(0,4);
\draw (0.5,4.5) node {$j$};
\draw (0.5,3.5) node {};
\draw (0.5,2.5) node {$c$};
\draw (0.5,1.5) node {};
\draw (0.5,0.5) node {$c$};
\draw (1.5,0.5) node {$0$};
\draw (1.5,1.5) node {$c'$};
\draw (1.5,2.5) node {$0$};
\draw (1.5,3.5) node {$0$};
\draw (2.5,3.5) node {$c$};
\draw (2.5,2.5) node {$0$};
\draw (2.5,1.5) node {$c$};
\draw (2.5,0.5) node {$0$};
\draw (3.5,3.5) node {$0$};
\draw (3.5,2.5) node {$0$};
\draw (3.5,1.5) node {$0$};
\draw (3.5,0.5) node {$0$};
\draw(1.4,-0.7) node {Configuration $\Delta$};
\end{scope}

\begin{scope}[xshift=17cm, yshift=0cm, scale=2.5]
\fill[red!20] (4,2) rectangle (5,3);
\foreach \i in {1,2,3,4,5}{
\path[draw, thin] (\i,0)--(\i,4);}
\foreach \j in {0,1,2,3,4}{
\path[draw, densely dotted] (0,\j)--(6,\j);} 
\foreach \j in {0,1,2,3,4}{ 
\path[draw, thin] (1,\j)--(5,\j);} 
\path[draw, very thick]  (0,0)--(6,0);
\path[draw, very thick]  (0,4)--(6,4);
\draw (2.5,4.5) node {$j$};
\draw (1.5,2.5) node {$c$};
\draw (1.5,1.65) node {$0$};
\draw (1.5,1.25) node {\tiny{safe}};
\draw (2.5,3.5) node {$c$};
\draw (2.5,2.5) node {$0$};
\draw (2.5,1.5) node {$0$};
\draw (2.5,0.5) node {$c$};
\draw (3.5,3.5) node {$0$};
\draw (3.5,2.5) node {$c$};
\draw (3.5,1.5) node {$0$};
\draw (3.5,0.5) node {$0$};
\draw (4.5,3.5) node {$0$};
\draw (4.5,2.5) node {$0$};
\draw (4.5,1.5) node {$0$};
\draw (4.5,0.5) node {$c$};
\draw (5.5,1.5) node {$c$};
\draw (3.0,-0.7) node {Configuration $\Delta'$};
\end{scope}

\begin{scope}[xshift=37cm, yshift=0cm, scale=2.5]
\fill[red!20] (3,1) rectangle (4,2);
\foreach \i in {1,2,3,4}{
\path[draw, thin] (\i,0)--(\i,4);}
\foreach \j in {0,1,2,3,4}{
\path[draw, densely dotted] (0,\j)--(5,\j);} 
\foreach \j in {0,1,2,3,4}{ 
\path[draw, thin] (1,\j)--(4,\j);} 
\path[draw, very thick]  (0,0)--(5,0);
\path[draw, very thick]  (0,4)--(5,4);
\draw (2.5,4.5) node {$j$};
\draw (1.5,2.5) node {$c$};
\draw (1.5,1.65) node {$0$};
\draw (1.5,1.25) node {\tiny{safe}};
\draw (2.5,3.5) node {$c$};
\draw (2.5,2.5) node {$0$};
\draw (2.5,1.5) node {$0$};
\draw (2.5,0.5) node {$c$};
\draw (3.5,3.5) node {$0$};
\draw (3.5,2.5) node {$c$};
\draw (3.5,1.5) node {$0$};
\draw (3.5,0.5) node {$0$};
\draw (4.5,1.5) node {$c'$};
\draw (4.5,0.5) node {$c$};
\draw (2.5,-0.7) node {Configuration $\Delta'_2$};
\end{scope}

\begin{scope}[xshift=5cm, yshift=-20cm, scale=2.5]
\fill[red!20] (2,2) rectangle (3,3);
\foreach \i in {1,2,3,4}{
\path[draw, thin] (\i,0)--(\i,4);}
\foreach \j in {0,1,2,3,4}{ 
\path[draw, densely dotted] (0,\j)--(5,\j);} 
\foreach \j in {0,1,2,3,4}{ 
\path[draw, thin] (1,\j)--(4,\j);}
\path[draw, very thick] (1,0)--(2,0)--(2,4)--(1,4);
\path[draw, very thick] (5,0)--(3,0)--(3,4)--(5,4);
\draw (3.5,4.5) node {$j$};
\draw (4.5,1.5) node {$c''$};
\draw (3.5,0.5) node {$c$};
\draw (3.5,1.5) node {$0$};
\draw (3.5,2.5) node {$c$};
\draw (3.5,3.5) node {$c'$};
\draw (2.5,0.5) node {$0$};
\draw (2.5,1.5) node {$0$};
\draw (2.5,2.5) node {$0$};
\draw (2.5,3.5) node {$0$};
\draw (1.5,0.5) node {$0$};
\draw (1.5,1.5) node {$c$};
\draw (1.5,2.5) node {$0$};
\draw (1.5,3.5) node {$c$};
\draw (2.5,-0.75) node {Configuration $\Lambda$};
\end{scope}

\begin{scope}[xshift=30cm, yshift=-20cm, scale=2.5]
\fill[red!20] (2,2) rectangle (3,3);
\foreach \i in {1,2,3,4,5}{
\path[draw, thin] (\i,0)--(\i,4);}
\foreach \j in {0,1,2,3,4}{ 
\path[draw, densely dotted] (0,\j)--(6,\j);} 
\foreach \j in {0,1,2,3,4}{ 
\path[draw, thin] (1,\j)--(4,\j);}
\path[draw, very thick] (1,0)--(2,0)--(2,4)--(1,4);
\path[draw, very thick] (6,0)--(3,0)--(3,4)--(6,4);
\draw (3.5,4.5) node {$j$};
\draw (4.5,0.5) node {$0$};
\draw (4.5,1.5) node {$0$};
\draw (4.5,2.5) node {$0$};
\draw (4.5,3.5) node {$c$};
\draw (5.6,2.5) node {\tiny{safe}};
\draw (3.5,0.5) node {$c$};
\draw (3.5,1.5) node {$0$};
\draw (3.5,2.5) node {$c$};
\draw (3.5,3.5) node {$c'$};
\draw (2.5,0.5) node {$0$};
\draw (2.5,1.5) node {$0$};
\draw (2.5,2.5) node {$0$};
\draw (2.5,3.5) node {$0$};
\draw (1.5,0.5) node {$0$};
\draw (1.5,1.5) node {$c$};
\draw (1.5,2.5) node {$0$};
\draw (1.5,3.5) node {$c$};
\draw (3.3,-0.75) node {Configuration $\Lambda'$};
\end{scope}

\end{tikzpicture}

\caption{Configurations $\Delta$, $\Delta'$, $\Delta'_2$, $\Lambda$ and $\Lambda'$, where $c'\ne c$ and $c''\ne c'$. The vertex marked ``safe" in configurations $\Delta'$ and  $\Delta'_2$ is not a doctor of some vertex in column $j-2$ (indeed, for it being safe, its neighbor in column $j-2$ must be colored).
Moreover, the vertex marked ``safe" in configuration $\Lambda'$ guarantees that $v_{2,j+1}$ is sound (with doctors $v_{1,j+1}$ and $v_{3,j+1}$).
Configurations $\Lambda, \Lambda', \Lambda_2, \Lambda'_2$ and $\Delta$ have no symmetrical counter-part according to the vertical axis.
Recall that column $j-3$ cannot be empty in configurations $\Lambda_2$ and $\Lambda'_2$.
The red vertices are called the sick vertices.}
\label{fig-delta}
\end{figure}
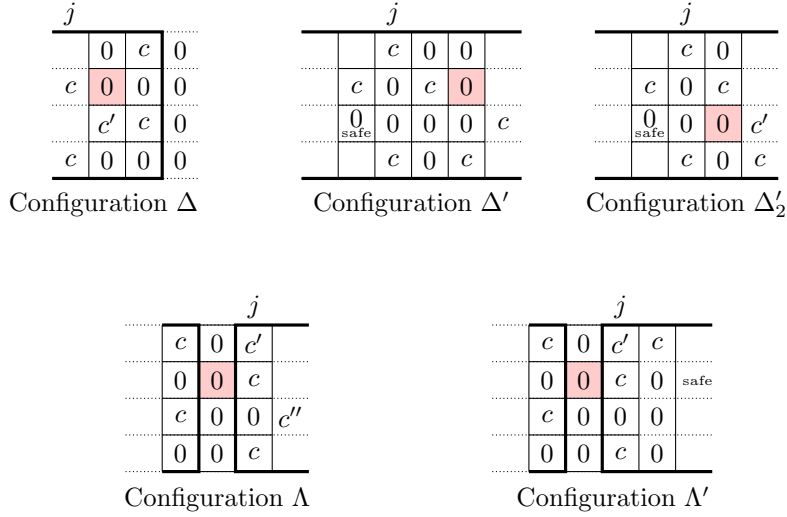

Note that in configuration $\Delta$, the vertex $v_{3,j+1}$ is not safe and is not sound either because $v_{3,j+2}$ (belonging to a border in configuration $\alpha$) cannot be its doctor.
Moreover, in configuration $\Delta'$ of Figure \ref{fig-delta}, there are three uncolored unsafe vertices which are potential patients: $v_{2,j}$, $v_{2,j+1}$, and $v_{3,j+2}$. The doctors $v_{2,j-1}$ and $v_{3,j}$ are associated to the patient $v_{2,j}$ (which is then sound). However, there is no way to associate doctors to the patients $v_{2,j+1}$ and $v_{3,j+2}$, since there are only three possible doctors for both of them (notice that $v_{3,j+3}$ can be colored or belong to a border in configuration $\alpha$).
In what follows, we will associate doctors $v_{1,j+1}$ and $v_{2,j+2}$ to the patient $v_{2,j+1}$ (making it sound) and leave the vertex $v_{3,j+2}$ neither safe nor sound.
Similarly, in configuration $\Delta'_2$, the vertex $v_{2,j}$ is sound (with doctors $v_{3,j}$ and $v_{2,j-1}$), but $v_{2,j+1}$ is neither safe nor sound.
Configurations $\Lambda$, $\Lambda'$, $\Lambda_2$, and $\Lambda'_2$ are defined to deal with an empty column between two blocks that contains exactly one unsafe vertex that we cannot make sound.

We say that a vertex is \emph{sick} if it is
either the vertex $v_{3,j+1}$ of configuration $\Delta$,
or the vertex $v_{3,j+2}$ of configuration $\Delta'$,
or the vertex $v_{2,j+1}$ of configuration $\Delta'_2$,
or the vertex $v_{3,j-1}$ of configurations $\Lambda$, $\Lambda'$, $\Lambda_2$ or $\Lambda_2'$ in Figure \ref{fig-delta} (indicated in red) or their symmetrical equivalents considering horizontal symmetry (for all the seven $\Delta$ and $\Lambda$ configurations) or vertical symmetry (only for configurations $\Delta'$ and $\Delta'_2$). 
Note that a sick vertex has two uncolored neighbors and so there is always an available color to be used for it, even after Bob colors one of its neighbors.
Also note that the sick vertices of configurations $\Lambda$, $\Lambda'$, $\Lambda_2$ and $\Lambda_2'$ are not inside a block.

\paragraph{Intuition of the algorithm.} 
Roughly, to prove our main theorem, we describe a strategy for Alice in $P_4 \x P_n$ such that, after each move of Alice, every border of a block is in configuration $\alpha$, $\beta$, $\gamma$, $\delta$ or $\pi$ (or is column $1$ or $n$). Moreover, every vertex of a block will be safe or sound or sick.
Then, for every possible move of Bob, we specify a valid move for Alice that ensures that the invariants still hold. 

In what follows, we say that two blocks $B$ and $B'$ (w.l.o.g., say $B$ is on the left of $B'$) are \emph{merged} if, before Bob's move, $B$ and $B'$ are separated by one or two empty columns and, after the next move of Bob and then of Alice, at least one vertex of each empty column separating $B$ and $B'$ is colored. In this case, the right border of $B'$ (resp. the left border of $B$) becomes the right (resp. left) border of the resulting new block.


We are now ready to prove the main result.

\begin{theorem}
For any $n \geq 1$, $\chi_g(P_4\x P_n) \leq  4$. Moreover, $\chi_g(P_4\x P_n) = 4$
for any $n \geq 18$.
\end{theorem}

\begin{proof}
The lower bound when $n\geq 18$ comes from Claim~\ref{claim:lower}. Now, let us describe a winning strategy for Alice with four colors. 

\begin{description}
\item[First move of Alice.] First, Alice colors $v_{3,1}$ with color $1$. Here, we consider that column $1$ is in configuration $\beta$ assuming a virtual vertex $v_{2,0}$ in column $0$ (which actually does not exist) also colored with $1$ (indeed, vertex $v_{2,1}$ is safe because it has only 3 neighbors).  
\end{description}

Now, by induction on the number of Alice's moves, let us assume that, after Alice's last move (and before Bob's next move), the reached partial proper coloring satisfies the following properties\footnote{Note that Properties $(3)$ and $(4)$ are redundant since they are included in the definition of configuration $\alpha$. We included them here to emphasize them.}:  

\begin{enumerate}
\item Every vertex of a block is safe or sound or sick.
\item Every border of a block is in configuration $\alpha$, $\beta$, $\gamma$, $\delta$ or $\pi$.
\item No border has its $4$ vertices colored with only two colors.
\item No vertex of a border in configuration $\alpha$ is a doctor.
\item A left border in configuration $\alpha$ has two uncolored vertices, except column $j$ of configurations $\Lambda$, $\Lambda'$, $\Lambda_2$ and $\Lambda'_2$ (following the orientation of Figure \ref{fig-delta}).
\end{enumerate}



The induction hypothesis clearly holds after the first move of Alice. So, let us assume it is Bob's turn.
Note that Bob can color any uncolored vertex using one of the four colors (it clearly holds for vertices not belonging to any block since they are safe or have at least two uncolored neighbors, and it holds for vertices in a block by Property $(1)$).
Hence, if we prove that whatever be the next move of Bob, Alice can color a vertex in such a way that the induction hypothesis still holds, then this will describe a winning strategy for Alice ({\it i.e.} all vertices will eventually be colored with at most $4$ colors). 
There are many cases depending on which is the next vertex colored by Bob. 
The cases must be considered in the order they are described below.
That is to say, when the last move of Bob fits into several cases, the one described first has priority. 

In the strategy for Alice as described below, it should be clear that, after each of Alice's moves, the properties of the induction hypothesis still hold.
The main difficulty is to check that all cases are considered.
We organized Bob's possible moves in such a way that it should be easy to check that no case has been forgotten. 

Before presenting all cases, we show in Figure \ref{fig:example} an example in $P_4\x P_n$ for $n\geq 10$.
After Alice colors $v_{3,1}$ with $1$ in her first move, suppose that Bob colors $v_{2,3}$, $v_{3,5}$, $v_{2,7}$ and $v_{3,9}$ with color $1$ in his following four moves. We will see in case {\bf $3$-new} that Alice's responses are the ones in Figure \ref{fig:example}(a), that is, $v_{4,3}$, $v_{1,5}$, $v_{4,7}$ and $v_{1,9}$, where Alice's (resp. Bob's) moves are represented in red (resp. blue). Notice that, at this point, there are five blocks, each one with exactly one column.
Column $1$ is in configuration $\beta$ (as explained before in the first move of Alice) and columns $3$, $5$, $7$, and $9$ are in configuration $\alpha$, considered as left borders.
From this, suppose that Bob colors $v_{1,7}$ with $2$. Then according to the strategy as explained in case $2\alpha'6$, Alice responds in $v_{3,8}$ with color $3$ or $4$ (say $3$), obtaining a configuration $\Lambda$ in columns $5$ to $7$ and a configuration $\Delta$ in columns $7$ to $9$ (see Figure \ref{fig:example}(b)).
Next, suppose that Bob colors $v_{4,5}$ with $2$. Then, following case $2\alpha'6$ again, Alice responds in $v_{2,6}$ with color $3$ or $4$ (say $3$), killing the configuration $\Lambda$ of columns $5$ to $7$, maintaining the configuration $\Delta$ in columns $7$ to $9$ and obtaining a new configuration $\Lambda$ in columns $3$ to $5$ (see Figure \ref{fig:example}(c)).
Finally, suppose that Bob colors $v_{1,3}$ with $2$. Then, according to the strategy as explained in case $2\alpha'1$, Alice responds in $v_{1,2}$ with color $1$, replacing the configuration $\Lambda$ of columns $3$ to $5$ with a configuration $\Lambda_2$ (see Figure \ref{fig:example}(d)).

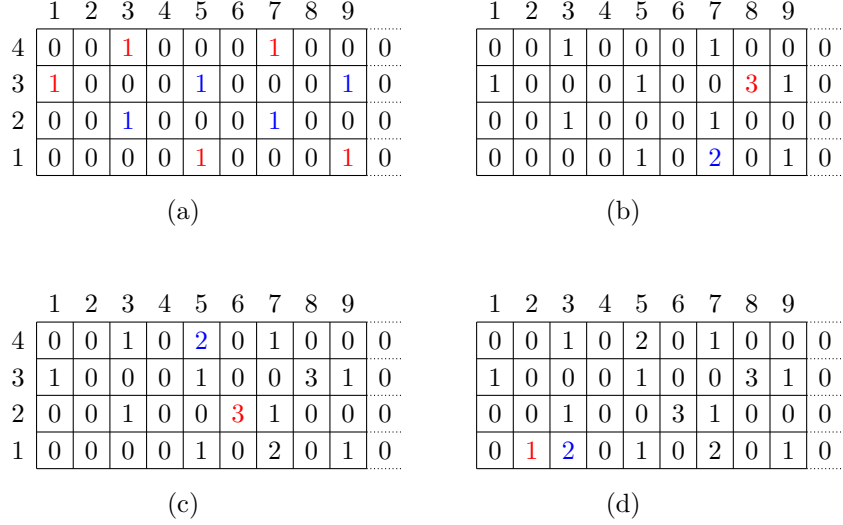
\begin{figure}[ht!]\centering
\scalebox{1.0}{
\begin{tikzpicture}[scale=0.195]
\tikzstyle{vertex}=[draw,circle,fill=black,inner sep=0pt, minimum size=1pt]

\begin{scope}[xshift=0cm, yshift=0cm, scale=2.5]
\foreach \i in {0,1,2,3,4,5,6,7,8,9}{
\path[draw, thin] (\i,0)--(\i,4);}
\foreach \j in {0,1,2,3,4}{ 
\path[draw, densely dotted] (9,\j)--(10,\j);} 
\foreach \j in {0,1,2,3,4}{ 
\path[draw, thin] (0,\j)--(9,\j);} 
\draw (-.5,0.5) node {$1$};
\draw (-.5,1.5) node {$2$};
\draw (-.5,2.5) node {$3$};
\draw (-.5,3.5) node {$4$};
\draw (0.5,0.5) node {$0$};
\draw (0.5,1.5) node {$0$};
\draw (0.5,2.5) node {$\red{1}$};
\draw (0.5,3.5) node {$0$};
\draw (1.5,0.5) node {$0$};
\draw (1.5,1.5) node {$0$};
\draw (1.5,2.5) node {$0$};
\draw (1.5,3.5) node {$0$};
\draw (2.5,0.5) node {$0$};
\draw (2.5,1.5) node {$\blue{1}$};
\draw (2.5,2.5) node {$0$};
\draw (2.5,3.5) node {$\red{1}$};
\draw (3.5,0.5) node {$0$};
\draw (3.5,1.5) node {$0$};
\draw (3.5,2.5) node {$0$};
\draw (3.5,3.5) node {$0$};
\draw (4.5,0.5) node {$\red{1}$};
\draw (4.5,1.5) node {$0$};
\draw (4.5,2.5) node {\blue{1}};
\draw (4.5,3.5) node {$0$};
\draw (5.5,0.5) node {$0$};
\draw (5.5,1.5) node {$0$};
\draw (5.5,2.5) node {$0$};
\draw (5.5,3.5) node {$0$};
\draw (6.5,0.5) node {$0$};
\draw (6.5,1.5) node {$\blue{1}$};
\draw (6.5,2.5) node {$0$};
\draw (6.5,3.5) node {$\red{1}$};
\draw (7.5,0.5) node {$0$};
\draw (7.5,1.5) node {$0$};
\draw (7.5,2.5) node {$0$};
\draw (7.5,3.5) node {$0$};
\draw (8.5,0.5) node {$\red{1}$};
\draw (8.5,1.5) node {$0$};
\draw (8.5,2.5) node {$\blue{1}$};
\draw (8.5,3.5) node {$0$};
\draw (9.5,0.5) node {$0$};
\draw (9.5,1.5) node {$0$};
\draw (9.5,2.5) node {$0$};
\draw (9.5,3.5) node {$0$};
\draw (4,-1) node {(a)};
\foreach \i in {1,...,9}{
\draw (\i-.5,4.5) node {\i};}
\end{scope}

\begin{scope}[xshift=30cm, yshift=0cm, scale=2.5]
\foreach \i in {0,1,2,3,4,5,6,7,8,9}{
\path[draw, thin] (\i,0)--(\i,4);}
\foreach \j in {0,1,2,3,4}{ 
\path[draw, densely dotted] (9,\j)--(10,\j);} 
\foreach \j in {0,1,2,3,4}{ 
\path[draw, thin] (0,\j)--(9,\j);} 
\draw (0.5,0.5) node {$0$};
\draw (0.5,1.5) node {$0$};
\draw (0.5,2.5) node {$1$};
\draw (0.5,3.5) node {$0$};
\draw (1.5,0.5) node {$0$};
\draw (1.5,1.5) node {$0$};
\draw (1.5,2.5) node {$0$};
\draw (1.5,3.5) node {$0$};
\draw (2.5,0.5) node {$0$};
\draw (2.5,1.5) node {$1$};
\draw (2.5,2.5) node {$0$};
\draw (2.5,3.5) node {$1$};
\draw (3.5,0.5) node {$0$};
\draw (3.5,1.5) node {$0$};
\draw (3.5,2.5) node {$0$};
\draw (3.5,3.5) node {$0$};
\draw (4.5,0.5) node {$1$};
\draw (4.5,1.5) node {$0$};
\draw (4.5,2.5) node {$1$};
\draw (4.5,3.5) node {$0$};
\draw (5.5,0.5) node {$0$};
\draw (5.5,1.5) node {$0$};
\draw (5.5,2.5) node {$0$};
\draw (5.5,3.5) node {$0$};
\draw (6.5,0.5) node {$\blue{2}$};
\draw (6.5,1.5) node {$1$};
\draw (6.5,2.5) node {$0$};
\draw (6.5,3.5) node {$1$};
\draw (7.5,0.5) node {$0$};
\draw (7.5,1.5) node {$0$};
\draw (7.5,2.5) node {$\red{3}$};
\draw (7.5,3.5) node {$0$};
\draw (8.5,0.5) node {$1$};
\draw (8.5,1.5) node {$0$};
\draw (8.5,2.5) node {$1$};
\draw (8.5,3.5) node {$0$};
\draw (9.5,0.5) node {$0$};
\draw (9.5,1.5) node {$0$};
\draw (9.5,2.5) node {$0$};
\draw (9.5,3.5) node {$0$};
\draw (4,-1) node {(b)};

\foreach \i in {1,...,9}{
\draw (\i-0.5,4.5) node {\i};}
\end{scope}

\begin{scope}[xshift=0cm, yshift=-20cm, scale=2.5]
\foreach \i in {0,1,2,3,4,5,6,7,8,9}{
\path[draw, thin] (\i,0)--(\i,4);}
\foreach \j in {0,1,2,3,4}{ 
\path[draw, densely dotted] (9,\j)--(10,\j);} 
\foreach \j in {0,1,2,3,4}{ 
\path[draw, thin] (0,\j)--(9,\j);} 
\draw (-.5,0.5) node {$1$};
\draw (-.5,1.5) node {$2$};
\draw (-.5,2.5) node {$3$};
\draw (-.5,3.5) node {$4$};
\draw (0.5,0.5) node {$0$};
\draw (0.5,1.5) node {$0$};
\draw (0.5,2.5) node {$1$};
\draw (0.5,3.5) node {$0$};
\draw (1.5,0.5) node {$0$};
\draw (1.5,1.5) node {$0$};
\draw (1.5,2.5) node {$0$};
\draw (1.5,3.5) node {$0$};
\draw (2.5,0.5) node {$0$};
\draw (2.5,1.5) node {$1$};
\draw (2.5,2.5) node {$0$};
\draw (2.5,3.5) node {$1$};
\draw (3.5,0.5) node {$0$};
\draw (3.5,1.5) node {$0$};
\draw (3.5,2.5) node {$0$};
\draw (3.5,3.5) node {$0$};
\draw (4.5,0.5) node {$1$};
\draw (4.5,1.5) node {$0$};
\draw (4.5,2.5) node {$1$};
\draw (4.5,3.5) node {$\blue{2}$};
\draw (5.5,0.5) node {$0$};
\draw (5.5,1.5) node {$\red{3}$};
\draw (5.5,2.5) node {$0$};
\draw (5.5,3.5) node {$0$};
\draw (6.5,0.5) node {$2$};
\draw (6.5,1.5) node {$1$};
\draw (6.5,2.5) node {$0$};
\draw (6.5,3.5) node {$1$};
\draw (7.5,0.5) node {$0$};
\draw (7.5,1.5) node {$0$};
\draw (7.5,2.5) node {$3$};
\draw (7.5,3.5) node {$0$};
\draw (8.5,0.5) node {$1$};
\draw (8.5,1.5) node {$0$};
\draw (8.5,2.5) node {$1$};
\draw (8.5,3.5) node {$0$};
\draw (9.5,0.5) node {$0$};
\draw (9.5,1.5) node {$0$};
\draw (9.5,2.5) node {$0$};
\draw (9.5,3.5) node {$0$};
\draw (4,-1) node {(c)};

\foreach \i in {1,...,9}{
\draw (\i-0.5,4.5) node {\i};}
\end{scope}

\begin{scope}[xshift=30cm, yshift=-20cm, scale=2.5]
\foreach \i in {0,1,2,3,4,5,6,7,8,9}{
\path[draw, thin] (\i,0)--(\i,4);}
\foreach \j in {0,1,2,3,4}{ 
\path[draw, densely dotted] (9,\j)--(10,\j);} 
\foreach \j in {0,1,2,3,4}{ 
\path[draw, thin] (0,\j)--(9,\j);} 
\draw (0.5,0.5) node {$0$};
\draw (0.5,1.5) node {$0$};
\draw (0.5,2.5) node {$1$};
\draw (0.5,3.5) node {$0$};
\draw (1.5,0.5) node {$\red{1}$};
\draw (1.5,1.5) node {$0$};
\draw (1.5,2.5) node {$0$};
\draw (1.5,3.5) node {$0$};
\draw (2.5,0.5) node {$\blue{2}$};
\draw (2.5,1.5) node {$1$};
\draw (2.5,2.5) node {$0$};
\draw (2.5,3.5) node {$1$};
\draw (3.5,0.5) node {$0$};
\draw (3.5,1.5) node {$0$};
\draw (3.5,2.5) node {$0$};
\draw (3.5,3.5) node {$0$};
\draw (4.5,0.5) node {$1$};
\draw (4.5,1.5) node {$0$};
\draw (4.5,2.5) node {$1$};
\draw (4.5,3.5) node {$2$};
\draw (5.5,0.5) node {$0$};
\draw (5.5,1.5) node {$3$};
\draw (5.5,2.5) node {$0$};
\draw (5.5,3.5) node {$0$};
\draw (6.5,0.5) node {$2$};
\draw (6.5,1.5) node {$1$};
\draw (6.5,2.5) node {$0$};
\draw (6.5,3.5) node {$1$};
\draw (7.5,0.5) node {$0$};
\draw (7.5,1.5) node {$0$};
\draw (7.5,2.5) node {$3$};
\draw (7.5,3.5) node {$0$};
\draw (8.5,0.5) node {$1$};
\draw (8.5,1.5) node {$0$};
\draw (8.5,2.5) node {$1$};
\draw (8.5,3.5) node {$0$};
\draw (9.5,0.5) node {$0$};
\draw (9.5,1.5) node {$0$};
\draw (9.5,2.5) node {$0$};
\draw (9.5,3.5) node {$0$};
\draw (4,-1) node {(d)};

\foreach \i in {1,...,9}{
\draw (\i-0.5,4.5) node {\i};}
\end{scope}

\end{tikzpicture}}
\caption{Example of a sequence of moves in the game from (a) to (d), where Alice's (resp. Bob's) moves are represented in red (resp. blue).\label{fig:example}}
\end{figure}

In the following, we present all possible cases.

\begin{description}
\item[Case 1.] {\bf When Bob colors a vertex:
\begin{itemize}
\item[(a)] inside a block, or
\item[(b)] in the border of configuration $\Delta$, or 
\item[(c)] in column $j$ or $j-1$ of configurations $\{\Lambda,\Lambda_2,\Lambda',\Lambda'_2\}$, or 
\item[(d)] in column $j-2$ of configurations $\{\Lambda,\Lambda'\}$, or
\item[(e)] in column $j-2$ of configurations $\{\Lambda_2,\Lambda'_2\}$ if column $j-3$ is not empty.
\end{itemize}}

In other words, Bob does not color a vertex of a border -- except in configurations $\Delta$, $\Lambda$ and $\Lambda'$ -- nor a vertex in an empty column -- except possibly in column $n$ if it is empty and column $n-1$ is not\footnote{This corresponds to the exception in the definition of blocks.}. 


\begin{description}
\item[Case 1$\Delta$.] When Bob colors  a vertex of column $j+1$ or $j+2$ of a configuration $\Delta$ in Figure~\ref{fig-delta} (recall that configuration $\Delta$ is only defined for column $j+2$ being the right border). If Bob has colored the sick vertex, then Alice colors any uncolored vertex of column $j+1$. Otherwise, Alice colors the sick vertex with any available color making safe all vertices of column $j+1$. Note that, if Bob has colored a vertex in column $j+2$, this remains a border in configuration $\alpha$. In all cases, the sick vertex is colored and so this configuration $\Delta$ disappears.

\item[Case 1$\Delta'$.] When Bob colors a vertex of Configuration $\Delta'$ or  $\Delta'_2$. The notations are the same as in Figure~\ref{fig-delta}, the symmetric case (according to the vertical or horizontal symmetries) can be dealt with similarly.

First, let us consider the configuration $\Delta'$.
If Bob colors $v_{2,j-1}$ or $v_{3,j}$, Alice colors $v_{2,j+1}$ with the same color, making $v_{2,j}$ and $v_{2,j+1}$ safe, and $v_{3,j+2}$ sound (with doctors $v_{2,j+2}$ and $v_{4,j+2}$).
If Bob colors $v_{2,j}$, Alice colors $v_{1,j+1}$ with the same color, making $v_{2,j+1}$ safe and $v_{3,j+2}$ sound.
From now on, $v_{2,j}$ is always sound (with doctors $v_{2,j-1}$ and $v_{3,j}$).
If Bob colors $v_{2,j+1}$ (resp. $v_{1,j+1}$), Alice colors $v_{1,j+1}$ (resp. $v_{2,j+1}$) with any available color, making $v_{3,j+2}$ sound.
If Bob colors $v_{3,j+2}$, Alice colors $v_{2,j+1}$ making it safe.
If Bob colors $v_{4,j+2}$ (resp., $v_{4,j+1}$), Alice colors $v_{3,j+2}$ with any color, making $v_{3,j+2}$ safe and $v_{2,j+1}$ sound (with doctors $v_{1,j+1}$ and $v_{2,j+2}$).
Finally, if Bob colors $v_{2,j+2}$ (with $c'\ne c$), Alice colors $v_{3,j+2}$ making it safe and obtaining the configuration $\Delta'_2$.

Now let us assume Bob colors a vertex in a configuration $\Delta'_2$.
If after Bob's move, Alice can color $v_{2,j}$ with $c'$, she does it, making all vertices of columns $j$ and $j+1$ safe. Otherwise, Alice colors $v_{1,j+1}$ with $c'$. There are two cases: either Bob's previous move colored $v_{2,j}$, then all vertices in columns $j$ and $j+1$ are safe, or Bob's previous move colored $z \in \{v_{2,j-1},v_{3,j}\}$ with color $c'$, in which case all vertices are safe but $v_{2,j}$ which is sound with doctors $v_{2,j+1}$ and $z'\in\{v_{2,j-1},v_{3,j}\}\setminus\{z\}$.

\item[Case 1$\Lambda$.] When Bob colors a vertex of column $j-2$, $j-1$ or $j$ of configuration $\Lambda$ or $\Lambda_2$.
We follow the orientation of Figure \ref{fig-delta}.
If Bob colors $v_{4,j-1}$, Alice colors $v_{2,j-1}$ with the same color.
If Bob colors $v_{2,j-1}$ (resp. $v_{3,j-1}$), Alice colors $v_{3,j-1}$ (resp. $v_{2,j-1}$) with any available color.
If Bob colors $v_{2,j}$, Alice colors $v_{3,j-1}$ with the same color if available, otherwise we are in Configuration $\Lambda_2$ (with $v_{3,j-2}$ colored and $v_{1,j-2}$ uncolored) and then Alice plays the same color as Bob on $v_{1,j-1}$, making $v_{3,j-1}$ sound with doctors $v_{2,j-1}$ and $v_{4,j-1}$.

If Bob colors $v_{1,j-1}$ with color $x$, Alice colors $v_{3,j-1}$ with $x$ if possible (in which case, all vertices of column $j-1$ are safe), otherwise we are in Configuration $\Lambda_2$ with $v_{3,j-2}$ colored $x$. If $x\ne c''$, Alice colors $v_{2,j}$ with $x$ (making $v_{3,j-1}$ sound with doctors $v_{4,j-1}$ and $v_{2,j-1}$). Otherwise, she colors $v_{4,j-1}$ with $c''$, making $v_{3,j-1}$ safe and $v_{2,j-1}$ sound with doctors $v_{3,j-1}$ and $v_{2,j}$.

Finally, assume that Bob colors $v_{1,j-2}$ or $v_{3,j-2}$.
First suppose that column $j-3$ is empty. Then we were in configuration $\Lambda$ before this move of Bob and column $j-2$ is a left border.
Then Alice's answer will be treated later in case $2\alpha\beta\gamma F$ (if the border is free to the left) or case $2\alpha'$ (otherwise).
Just getting ahead, in cases $2\alpha\beta\gamma F$, $2\alpha' 1$ and $2\alpha' 2$, Alice answers in column $j-3$, obtaining a configuration $\Lambda_2$. In cases $2\alpha'3$, $2\alpha'4$ and $2\alpha'5$, Alice merges two blocks, making the sick vertex safe or sound, and creating a new configuration $\Lambda$ or $\Lambda'$. Case $2\alpha'6$ cannot occur since $c(v_{1,j+2})\neq c$.
So, suppose that column $j-3$ is not empty.
Note that $v_{3,j-2}$ cannot be a doctor since it lies on a border in configuration $\alpha$.
Then Alice colors the sick vertex $v_{3,j-1}$, removing the configuration $\Lambda$ or $\Lambda_2$.

\item[Case 1$\Lambda'$.] When Bob colors a vertex of a column in $\{j-2,\ldots,j+1\}$ of configuration $\Lambda'$ or $\Lambda'_2$.
Notice that $v_{3,j+1}$ is safe and $v_{2,j+1}$ is sound (with doctors $v_{1,j+1}$ and $v_{3,j+1}$).
If Bob colors $v_{1,j+1}$ or $v_{3,j+1}$ of configuration $\Lambda'$ or $\Lambda'_2$, Alice colors $v_{2,j+1}$ with $c''\ne c'$, obtaining the configuration $\Lambda$ or $\Lambda_2$, respectively.
If Bob colors $v_{2,j+1}$ of configuration $\Lambda'$ or $\Lambda'_2$, Alice colors the sick vertex $v_{3,j-1}$ (merging the two blocks and leaving $v_{2,j-1}$ sound with doctors $v_{1,j-1}$ and $v_{2,j}$).
Finally assume that Bob colors $v_{1,j-2}$ or $v_{3,j-2}$.
First suppose that column $j-3$ is empty. Then we were in configuration $\Lambda'$ before this Bob's move and column $j-2$ is a left border.
Exactly as in case $1\Lambda$, Alice's answer will be treated in case $2\alpha\beta\gamma F$ or $2\alpha'$, obtaining a configuration $\Lambda_2'$ (instead of $\Lambda_2$) or a new configuration $\Lambda$ or $\Lambda'$, and making the sick vertex safe or sound.
So, assume that column $j-3$ is not empty.
As in case $1\Lambda$, $v_{3,j-2}$ cannot be a doctor since it lies on a border in configuration $\alpha$.
Then Alice colors the sick vertex $v_{3,j-1}$, removing the configuration $\Lambda'$ or $\Lambda'_2$.
All other possibilities are considered similarly as in Case $1\Lambda$.

\item[Case 1-doc.] If Bob colors a doctor of a sound vertex $v$ and $v$ is not in a border in configuration $\gamma$, then Alice colors $v$ with any available color, making it safe.
\item[Case $1\gamma$-doc.] If Bob colors the doctor (not in the border) of the sound vertex $v$ in a border (column $j$) in configuration $\gamma$, then Alice colors the second doctor $w$ of $v$ ($w$ is the uncolored neighbor of $v$ in the border in configuration $\gamma$)  with the same color, making $v$ safe and obtaining a border  (column $j$) in configuration $\beta$ (see Figure~\ref{fig-gamma-penultimate}).

\begin{figure}[ht!]\centering
\begin{tikzpicture}[scale=0.195]
\tikzstyle{vertex}=[draw,circle,fill=black,inner sep=0pt, minimum size=1pt]

\begin{scope}[xshift=0cm, yshift=0cm, scale=2.5]
\foreach \i in {1,2,3,4}{
\path[draw, thin] (\i,0)--(\i,4);}
\foreach \j in {0,1,2,3,4}{ 
\path[draw, densely dotted](0,\j)--(5,\j);} 
\foreach \j in {0,1,2,3,4}{ 
\path[draw, thin] (1,\j)--(2,\j)--(3,\j)--(4,\j);} 
\path[draw, very thick]  (0,0) --(1,0)--(2,0)--(3,0) --(3,1)--(3,2)--(3,3)--(3,4)--(2,4)--(1,4)-- (0,4)  ;
\draw (2.5,4.5) node {$j$};
\draw (0.5,1.5) node {};
\draw (1.5,1.5) node {$\blue{c'}$};
\draw (1.5,2.5) node {$c$};
\draw (2.5,0.5) node {$c$};
\draw (2.5,1.5) node {$0$};
\draw (2.5,2.5) node {$\red{c'}$};
\draw (2.5,3.5) node {$c$};
\draw (3.5,0.5) node {$0$};
\draw (3.5,1.5) node {$0$};
\draw (3.5,2.5) node {$0$};
\draw (3.5,3.5) node {$0$};
\draw (2.5,-.75) node {Case $1\gamma$-doc: $\gamma\to\beta$}; 
\end{scope}

\end{tikzpicture}
\caption{Case $1\gamma$-doc: Bob just colored the vertex in blue of the penultimate column of a border in configuration $\gamma$. Alice's answer is in red. The bold line figures the surrounding of the updated block (before Alice's move).}
\label{fig-gamma-penultimate}
\end{figure}

\item[Case 1-safe.] If Bob colors a safe or sound vertex, then:

\begin{itemize}
\item if there is a sick vertex, Alice colors it, removing the configuration $\Delta$, $\Delta'$, $\Delta'_2$, $\Lambda$, $\Lambda_2$, $\Lambda'$ or $\Lambda'_2$ it belongs to;
\item if there is a sound vertex, Alice colors it, making it safe;
\item if there is an uncolored safe vertex not in a border, Alice colors it (Note that it cannot be a doctor since, in that case, there are no sound vertices anymore);
\end{itemize}

Otherwise, every vertex inside a block is already colored and no configuration $\Lambda,\Lambda',\Lambda_2$ or $\Lambda'_2$ exists.

\begin{itemize}
\item {\bf Case $1\delta$.}  When there is a border in configuration $\delta$ (notation of Figure~\ref{fig-1safes}, symmetrical configurations must be dealt with accordingly). If $v_{3,j+2}$ is not colored $c'$ (Case $1\delta 1$), Alice colors $v_{3,j+1}$ with $c'$ obtaining a border in configuration $\beta$ or merging two blocks. Otherwise (Case $1\delta 2$), she colors $v_{4,j+1}$ with $c'$, merging the two blocks and making $v_{3,j+1}$ safe and $v_{2,j+1}$ sound (with doctors $v_{3,j+1}$ and $v_{1,j+1}$).

\begin{figure}[ht]\centering
\begin{tikzpicture}[scale=0.195]
\tikzstyle{vertex}=[draw,circle,fill=black,inner sep=0pt, minimum size=1pt]

\begin{scope}[xshift=0cm, yshift=0cm, scale=2.5]
\foreach \i in {2,3,4}{
\path[draw, thin] (\i,0)--(\i,4) ;}
\foreach \j in {0,1,2,3,4}{
\path[draw, densely dotted] (1,\j)--(5,\j);} 
\foreach \j in {0,1,2,3,4}{ 
\path[draw, thin] (2,\j)--(3,\j);} 
\path[draw, very thick]  (1,0)--(3,0)--(3,4)--(1,4) ;
\draw (2.5,4.5) node {$j$};
\draw (1.5,2.5) node {$c$};
\draw (2.5,3.5) node {$c$};
\draw (2.5,2.5) node {$0$};
\draw (2.5,1.5) node {$c'$};
\draw (3.5,0.5) node {$0$};
\draw (3.5,1.5) node {$0$};
\draw (3.5,2.5) node {$\red{c'}$};
\draw (3.5,3.5) node {$0$};
\draw (4.7,2.5) node {$\ne c'$};
\draw (2.5,-0.75) node {Case $1\delta 1$: $\delta\to\beta$/merge};
\end{scope}

\begin{scope}[xshift=25cm, yshift=0cm, scale=2.5]
\foreach \i in {2,3,4}{
\path[draw, thin] (\i,0)--(\i,4) ;}
\foreach \j in {0,1,2,3,4}{
\path[draw, densely dotted] (1,\j)--(5,\j);} 
\foreach \j in {0,1,2,3,4}{ 
\path[draw, thin] (2,\j)--(3,\j);} 
\path[draw, very thick]  (1,0)--(3,0)--(3,4)--(1,4) ;
\path[draw, very thick]  (5,0)--(4,0)--(4,4)--(5,4) ;
\draw (2.5,4.5) node {$j$};
\draw (1.5,2.5) node {$c$};
\draw (2.5,3.5) node {$c$};
\draw (2.5,2.5) node {$0$};
\draw (2.5,1.5) node {$c'$};
\draw (3.5,0.5) node {$0$};
\draw (3.5,1.5) node {$0$};
\draw (3.5,2.5) node {$0$};
\draw (3.5,3.5) node {$\red{c'}$};
\draw (4.5,2.5) node {$c'$};
\draw (2.5,-0.75) node {Case $1\delta 2$: $\delta\to$ merge};
\end{scope}
\end{tikzpicture}

\caption{Case $1\delta$: Bob has played a safe or sound vertex and Alice plays in a border in configuration $\delta$, where $c,c'$ are distinct colors. In case $1\delta 1$, $c(v_{3,j+2})\neq c'$.  The notation ``$\delta\to\beta$/merge" (in Case $1\delta 1$) means that the former block of column $j$ (before Alice's move) is extended by one column to the right and its new right border is in configuration $\beta$, or it is merged with the former block of column $j+2$.}
\label{fig-1safes}
\end{figure}

\item {\bf Case $1\pi$.} If there is a border in configuration $\pi$ (notation of Figure~\ref{fig:configs}), Alice colors $v_{1,j+1}$ with color $c'$, making $v_{2,j+1}$ safe and $v_{3,j+1}$ sound, and merging the two blocks.

\item {\bf Case $1\gamma$.} If there is a border in configuration $\gamma$, Alice plays in the border (vertex $v_{2,j}$ in Case $1\gamma$ of Figure \ref{fig-1safeFreeAB}) obtaining a border in configuration $\delta$. 

\item {\bf Case $1\alpha\beta F$.} If there is a free border in configuration $\alpha$ or $\beta$, Alice plays as described in Figure~\ref{fig-1safeFreeAB}. 

\begin{figure}[ht!]\centering
\begin{tikzpicture}[scale=0.195]
\tikzstyle{vertex}=[draw,circle,fill=black,inner sep=0pt, minimum size=1pt]

\begin{scope}[xshift=0cm, yshift=0cm, scale=2.5]
\foreach \i in {3,4}{
\path[draw, thin] (\i,0)--(\i,4);}
\foreach \j in {0,1,2,3,4}{ 
\path[draw, densely dotted] (2,\j)--(5,\j);} 
\foreach \j in {0,1,2,3,4}{ 
\path[draw, thin] (3,\j)--(4,\j);} 
\path[draw, very thick]  (2,0)--(4,0)--(4,4)--(2,4) ;
\draw (3.5,4.5) node {$j$};
\draw (2.5,2.5) node {$c$};
\draw (2.5,1.65) node {$0$};
\draw (2.5,1.25) node {\tiny{safe}};
\draw (3.5,0.5) node {$c$};
\draw (3.5,1.5) node {\red{$c'$}};
\draw (3.5,2.5) node {$0$};
\draw (3.5,3.5) node {$c$};
\draw (4.5,0.5) node {$0$};
\draw (4.5,1.5) node {$0$};
\draw (4.5,2.5) node {$0$};
\draw (4.5,3.5) node {$0$};
\draw (3.5,-.75) node {Case $1\gamma$: $\gamma\to \delta$}; 
\end{scope}

\begin{scope}[xshift=22cm, yshift=0cm, scale=2.5]
\foreach \i in {1,2,3,4}{
\path[draw, thin] (\i,0)--(\i,4);}
\foreach \j in {0,1,2,3,4}{ 
\path[draw, densely dotted](0,\j)--(5,\j);} 
\foreach \j in {0,1,2,3,4}{ 
\path[draw, thin] (1,\j)--(2,\j)--(3,\j)--(4,\j);} 
\path[draw, very thick]  (0,0) --(1,0)--(2,0)--(3,0) --(3,1)--(3,2)--(3,3)--(3,4)--(2,4)--(1,4)-- (0,4) ;
\draw (2.5,4.5) node {$j$};
\draw (2.5,0.5) node {$c$};
\draw (2.5,2.5) node {$c$};
\draw (3.5,0.5) node {$0$};
\draw (3.5,1.5) node {\red{$c$}};
\draw (3.5,2.5) node {$0$};
\draw (3.5,3.5) node {$0$};
\draw (4.5,0.5) node {$0$};
\draw (4.5,1.5) node {$0$};
\draw (4.5,2.5) node {$0$};
\draw (4.5,3.5) node {$0$};
\draw (2.5,-0.75) node {Case $1\alpha\beta F1$: $\alpha\to\beta$}; 
\end{scope}

\begin{scope}[xshift=45cm, yshift=0cm, scale=2.5]
\foreach \i in {1,2,3,4}{
\path[draw, thin] (\i,0)--(\i,4);}
\foreach \j in {0,1,2,3,4}{ 
\path[draw, densely dotted](0,\j)--(5,\j);} 
\foreach \j in {0,1,2,3,4}{ 
\path[draw, thin] (1,\j)--(2,\j)--(3,\j)--(4,\j);} 
\path[draw, very thick]  (0,0) --(1,0)--(2,0)--(3,0) --(3,1)--(3,2)--(3,3)--(3,4)--(2,4)--(1,4)-- (0,4)  ;
\draw (2.5,4.5) node {$j$};
\draw (1.5,1.5) node {$c$};
\draw (2.5,2.5) node {$c$};
\draw (3.5,0.5) node {$0$};
\draw (3.5,1.5) node {\red{$c$}};
\draw (3.5,2.5) node {$0$};
\draw (3.5,3.5) node {$0$};
\draw (3.5,3.5) node {$0$};
\draw (4.5,0.5) node {$0$};
\draw (4.5,1.5) node {$0$};
\draw (4.5,2.5) node {$0$};
\draw (4.5,3.5) node {$0$};
\draw (2.5,-.75) node {Case $1\alpha\beta F2$: $\beta\to\beta$}; 
\end{scope}

\end{tikzpicture}
\caption{Cases $1\alpha\beta F$ and $1\gamma$: Bob has played a safe or sound vertex and Alice plays in a free border in configuration $\alpha$ or $\beta$, or in a border in configuration $\gamma$.}
\label{fig-1safeFreeAB}
\end{figure}

\end{itemize}

Therefore, we can assume that every vertex inside a block is already colored and that there are no free borders nor borders in configuration $\gamma, \delta$ or $\pi$. Unless all vertices are already colored, there must be an empty column between two borders: the different possibilities are described in Figure~\ref{fig:Case8}. Recall that, by Property $(5)$, a left border in configuration $\alpha$ must have exactly two uncolored vertices (this is important for cases $1\alpha$ below).

\begin{itemize}
\item {\bf Case $1\beta$:} 
There exists a border in configuration $\beta$.
Alice colors $v_{3,j+1}$ with any available color $c'\ne c$, making it safe and $v_{2,j+1}$ sound (with doctors $v_{2,j}$ and $v_{1,j+1}$), and merging the two blocks.
\end{itemize}
The only remaining cases occur when the empty column lies between two borders in Configuration $\alpha$.
\begin{itemize}
\item {\bf Case $1\alpha 1$:} If $v_{3,j+2}$ is not colored $c$, Alice colors $v_{3,j+1}$ with color $c$, making it and $v_{2,j+1}$ safe.
\item {\bf Case $1\alpha 2$:} Else, if column $j+3$ is empty, Alice colors $v_{3,j+1}$ with color $c'$ obtaining the configuration $\Delta$.   
\item {\bf Case $1\alpha 3$:} Finally, if column $j+3$ is not empty, then, by Property $(4)$, $v_{2,j+2}$ is not a doctor. Alice colors $v_{3,j+1}$ with any available color $c'$, making it safe and $v_{2,j+1}$ sound (with doctors $v_{1,j+1}$ and $v_{2,j+2}$).

\end{itemize}

\begin{figure}[ht!]\centering
\begin{tikzpicture}[scale=0.195]
\tikzstyle{vertex}=[draw,circle,fill=black,inner sep=0pt, minimum size=1pt]

\begin{scope}[xshift=0cm, yshift=0cm, scale=2.5]
\foreach \i in {1,2,3}{
\path[draw, thin] (\i,0)--(\i,4);}
\foreach \j in {0,1,2,3,4}{ 
\path[draw, densely dotted](0,\j)--(4,\j);} 
\foreach \j in {0,1,2,3,4}{
\path[draw, thin] (1,\j)--(2,\j);} 
\path[draw, very thick] (0,0)--(2,0)--(2,4)--(0,4);
\path[draw, very thick] (4,0)--(3,0)--(3,4)--(4,4);
\draw (1.5,4.5) node {$j$};
\draw (0.5,1.5) node {$c$};
\draw (1.5,1.5) node {$0$};
\draw (1.5,2.5) node {$c$};
\draw (2.5,3.5) node {$0$};
\draw (2.5,2.5) node {\red{$c'$}};
\draw (2.5,1.5) node {$0$};
\draw (2.5,0.5) node {$0$};
\draw (1.8,-.75) node {Case $1\beta$}; 
\draw (1.8,-1.5) node {$\beta\to$ merge}; 
\end{scope}

\begin{scope}[xshift=20cm, yshift=0cm, scale=2.5]
\foreach \i in {1,2}{
\path[draw, thin] (\i,0)--(\i,4);}
\foreach \j in {0,1,2,3,4}{ 
\path[draw, thin] (0,\j)--(2,\j)--(3,\j);} 
\path[draw, very thick]  (0,0) --(1,0)--(1,1) --(1,2)--(1,3)--(1,4)--(0,4) ;
\path[draw, very thick]  (3,0) --(2,0)--(2,1) --(2,2)--(2,3)--(2,4)--(3,4) ;
\draw (0.5,4.5) node {$j$};
\draw (0.5,0.5) node {};
\draw (0.5,1.5) node {$c$};
\draw (0.5,2.5) node {};
\draw (0.5,3.5) node {$c$};
\draw (1.5,0.5) node {$0$};
\draw (1.5,1.5) node {$0$};
\draw (1.5,2.5) node {\red{$c$}};
\draw (1.5,3.5) node {$0$};
\draw (2.8,2.5) node {$\ne c$};
\draw (1.5,-.75) node {Case $1\alpha 1$};
\draw (1.5,-1.5) node {$\alpha$+$\alpha\to$ merge}; 
\end{scope}

\begin{scope}[xshift=40cm, yshift=0cm, scale=2.5]
\foreach \i in {1,2,3}{
\path[draw, thin] (\i,0)--(\i,4);}
\foreach \j in {0,1,2,3,4}{ 
\path[draw, thin] (0,\j)--(4,\j);} 
\path[draw, very thick]  (0,0)--(1,0)--(1,4)--(0,4);
\path[draw, very thick]  (3,0)--(2,0)--(2,4)--(3,4)--(3,0);
\draw (0.5,4.5) node {$j$};
\draw (0.5,0.5) node {};
\draw (0.5,1.5) node {$c$};
\draw (0.5,2.5) node {};
\draw (0.5,3.5) node {$c$};
\draw (1.5,0.5) node {$0$};
\draw (1.5,1.5) node {$0$};
\draw (1.5,2.5) node {$\red{c'}$};
\draw (1.5,3.5) node {$0$};
\draw (2.5,0.5) node {$c$};
\draw (2.5,1.5) node {$0$};
\draw (2.5,2.5) node {$c$};
\draw (2.5,3.5) node {$0$};
\draw (3.5,0.5) node {$0$};
\draw (3.5,1.5) node {$0$};
\draw (3.5,2.5) node {$0$};
\draw (3.5,3.5) node {$0$};
\draw (2.0,-.75) node {Case $1\alpha 2$};
\draw (2.0,-1.5) node {$\alpha$+$\alpha\to\Delta$}; 
\end{scope}

\begin{scope}[xshift=60cm, yshift=0cm, scale=2.5]
\foreach \i in {1,2,3}{
\path[draw, thin] (\i,0)--(\i,4);}
\foreach \j in {0,1,2,3,4}{ 
\path[draw, thin] (0,\j)--(3,\j);} 
\foreach \j in {0,1,2,3,4}{ 
\path[draw, densely dotted] (0,\j)--(1,\j)--(2,\j)--(3,\j)--(4,\j);}
\path[draw, very thick]  (0,0)--(1,0)--(1,4)--(0,4);
\path[draw, very thick]  (4,0)--(2,0)--(2,4)--(4,4);
\draw (0.5,4.5) node {$j$};
\draw (0.5,0.5) node {};
\draw (0.5,1.5) node {$c$};
\draw (0.5,2.5) node {};
\draw (0.5,3.5) node {$c$};
\draw (1.5,0.5) node {$0$};
\draw (1.5,1.5) node {$0$};
\draw (1.5,2.5) node {$\red{c'}$};
\draw (1.5,3.5) node {$0$};
\draw (2.5,0.5) node {$c$};
\draw (2.5,1.5) node {$0$};
\draw (2.5,2.5) node {$c$};
\draw (2.5,3.5) node {$0$};
\draw (3.8,4.5) node {$\ne 0$};
\draw (2.0,-0.75) node {Case $1\alpha 3$};
\draw (2.0,-1.5) node {$\alpha$+$\alpha\to$ merge};
\end{scope}

\end{tikzpicture}
\caption{Cases $1\beta$, $1\gamma$ and $1\alpha$: Bob has colored a safe or sound vertex inside a block and there is no uncolored vertex inside a block. Moreover, there are no free borders nor borders in configuration $\gamma, \delta$ or $\pi$. The red cell indicates Alice's move. All symmetrical configurations according to vertical/horizontal symmetry axis (column $j$ / line in the middle) are considered similarly.}
\label{fig:Case8}
\end{figure}
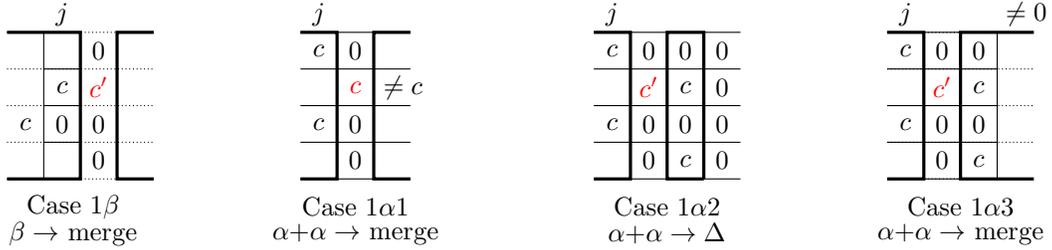
\end{description}

It can be checked that, in all subcases of Case $1$, after Alice's move, all invariants of the induction hypothesis still hold.

\item[Case 2. When Bob colors a vertex of a border (not in configuration $\Delta$, $\Lambda$ or $\Lambda'$).]~\\
Note that Bob cannot color a vertex of a border in configuration $\pi$ since all its vertices are already colored.
Moreover, a border in configuration $\pi$ cannot be free.

\begin{itemize}
\item  {\bf Case $2\delta$.} Bob colored $v_{1,j}$ or $v_{3,j}$ in a border in configuration $\delta$. If $v_{3,j+2}$ is not colored $c'$ (Case $2\delta 1$), Alice colors $v_{3,j+1}$ with $c'$ obtaining a border in configuration $\beta$ or merging two blocks. Otherwise (Case $2\delta 2$), she colors $v_{4,j+1}$ with $c'$, making $v_{3,j+1}$ safe and $v_{2,j+1}$ sound (with doctors $v_{3,j+1}$ and $v_{1,j+1}$) (see Figure~\ref{fig:2S}).

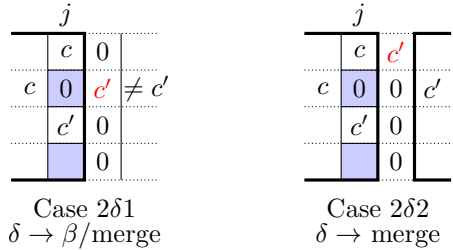
\begin{figure}[ht]\centering
\begin{tikzpicture}[scale=0.195]
\tikzstyle{vertex}=[draw,circle,fill=black,inner sep=0pt, minimum size=1pt]

\begin{scope}[xshift=0cm, yshift=0cm, scale=2.5]
\fill[blue!20] (2,0) rectangle (3,1);
\fill[blue!20] (2,2) rectangle (3,3);
\foreach \i in {2,3,4}{
\path[draw, thin] (\i,0)--(\i,4) ;}
\foreach \j in {0,1,2,3,4}{
\path[draw, densely dotted] (1,\j)--(5,\j);} 
\foreach \j in {0,1,2,3,4}{ 
\path[draw, thin] (2,\j)--(3,\j);} 
\path[draw, very thick]  (1,0)--(3,0)--(3,4)--(1,4) ;
\draw (2.5,4.5) node {$j$};
\draw (1.5,2.5) node {$c$};
\draw (2.5,3.5) node {$c$};
\draw (2.5,2.5) node {$0$};
\draw (2.5,1.5) node {$c'$};
\draw (3.5,0.5) node {$0$};
\draw (3.5,1.5) node {$0$};
\draw (3.5,2.5) node {$\red{c'}$};
\draw (3.5,3.5) node {$0$};
\draw (4.7,2.5) node {$\ne c'$};
\draw (3.0,-.75) node {Case $2\delta 1$}; 
\draw (3.0,-1.5) node {$\delta\to\beta$/merge}; 
\end{scope}

\begin{scope}[xshift=20cm, yshift=0cm, scale=2.5]
\fill[blue!20] (2,0) rectangle (3,1);
\fill[blue!20] (2,2) rectangle (3,3);
\foreach \i in {2,3,4}{
\path[draw, thin] (\i,0)--(\i,4) ;}
\foreach \j in {0,1,2,3,4}{
\path[draw, densely dotted] (1,\j)--(5,\j);} 
\foreach \j in {0,1,2,3,4}{ 
\path[draw, thin] (2,\j)--(3,\j);} 
\path[draw, very thick]  (1,0)--(3,0)--(3,4)--(1,4) ;
\path[draw, very thick]  (5,0)--(4,0)--(4,4)--(5,4) ;
\draw (2.5,4.5) node {$j$};
\draw (1.5,2.5) node {$c$};
\draw (2.5,3.5) node {$c$};
\draw (2.5,2.5) node {$0$};
\draw (2.5,1.5) node {$c'$};
\draw (3.5,0.5) node {$0$};
\draw (3.5,1.5) node {$0$};
\draw (3.5,2.5) node {$0$};
\draw (3.5,3.5) node {$\red{c'}$};
\draw (4.5,2.5) node {$c'$};
\draw (3.0,-.75) node {Case $2\delta 2$}; 
\draw (3.0,-1.5) node {$\delta\to$ merge}; 
\end{scope}
\end{tikzpicture}
\caption{Case $2\delta$: Bob just colored a vertex (in blue) of a border in configuration $\delta$.}
\label{fig:2S}
\end{figure}

\item  {\bf Case $2\alpha\beta\gamma F$ (Free border in Configuration $\alpha$, $\beta$ or $\gamma$):} Bob colors a vertex $v$ of a free border of a block $B$ in configuration $\alpha$, $\beta$ or $\gamma$. Then Alice can extend this block $B$ with a new border in configuration $\beta$ as described in Figure~\ref{fig:nextMoveAlice1} (in case of a right border).
In all subcases, Bob just colored a vertex of column $j$: if we need to specify which vertex, it is depicted in blue, otherwise it is any of the white cells of column $j$.
The answer of Alice is in red. 

Also recall that, if $B$ is just one column $j$ (in Configuration $\alpha$), then it is a left border and therefore, it is free if columns $j-2$ and $j-1$ are empty. That is, if $B$ is a single column $j$ and column $j-2$ is not empty, it must not be considered in this case (even if column $j+1$ and $j+2$ are empty) but in the following Case 2$\alpha$ (indeed, applying Rule $2\alpha F$ ``to the right" in the case of a single column would not guarantee Property $(5)$).

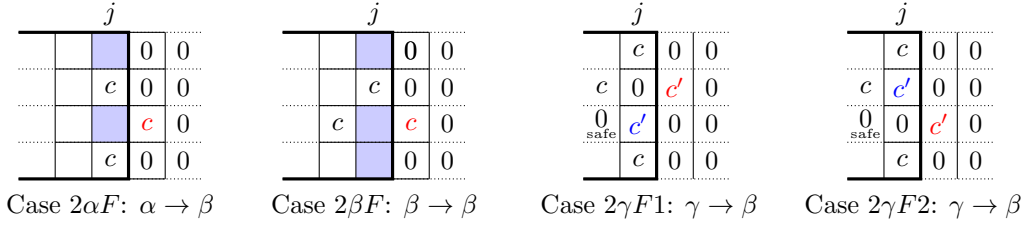
\begin{figure}[ht!]\centering
\begin{tikzpicture}[scale=0.195]
\tikzstyle{vertex}=[draw,circle,fill=black,inner sep=0pt, minimum size=1pt]

\begin{scope}[xshift=0cm, yshift=0cm, scale=2.5]
\fill[blue!20] (2,1) rectangle (3,2);
\fill[blue!20] (2,3) rectangle (3,4);
\foreach \i in {1,2,3,4}{
\path[draw, thin] (\i,0)--(\i,4);}
\foreach \j in {0,1,2,3,4}{ 
\path[draw, densely dotted](0,\j)--(5,\j);} 
\foreach \j in {0,1,2,3,4}{ 
\path[draw, thin] (1,\j)--(2,\j)--(3,\j)--(4,\j);} 
\path[draw, very thick]  (0,0) --(1,0)--(2,0)--(3,0) --(3,1)--(3,2)--(3,3)--(3,4)--(2,4)--(1,4)-- (0,4) ;
\draw (2.5,4.5) node {$j$};
\draw (2.5,0.5) node {$c$};
\draw (2.5,2.5) node {$c$};
\draw (3.5,0.5) node {$0$};
\draw (3.5,1.5) node {\red{$c$}};
\draw (3.5,2.5) node {$0$};
\draw (3.5,3.5) node {$0$};
\draw (4.5,0.5) node {$0$};
\draw (4.5,1.5) node {$0$};
\draw (4.5,2.5) node {$0$};
\draw (4.5,3.5) node {$0$};
\draw (2.5,-0.75) node {Case $2\alpha F$: $\alpha\to\beta$}; 
\end{scope}

\begin{scope}[xshift=18cm, yshift=0cm, scale=2.5]
\fill[blue!20] (2,0) rectangle (3,2);
\fill[blue!20] (2,3) rectangle (3,4);
\foreach \i in {1,2,3,4}{
\path[draw, thin] (\i,0)--(\i,4);}
\foreach \j in {0,1,2,3,4}{ 
\path[draw, densely dotted](0,\j)--(5,\j);} 
\foreach \j in {0,1,2,3,4}{ 
\path[draw, thin] (1,\j)--(2,\j)--(3,\j)--(4,\j);} 
\path[draw, very thick]  (0,0) --(1,0)--(2,0)--(3,0) --(3,1)--(3,2)--(3,3)--(3,4)--(2,4)--(1,4)-- (0,4)  ;
\draw (2.5,4.5) node {$j$};
\draw (1.5,1.5) node {$c$};
\draw (2.5,2.5) node {$c$};
\draw (3.5,0.5) node {$0$};
\draw (3.5,1.5) node {\red{$c$}};
\draw (3.5,2.5) node {$0$};
\draw (3.5,3.5) node {$0$};
\draw (3.5,3.5) node {$0$};
\draw (4.5,0.5) node {$0$};
\draw (4.5,1.5) node {$0$};
\draw (4.5,2.5) node {$0$};
\draw (4.5,3.5) node {$0$};
\draw (2.5,-.75) node {Case $2\beta F$: $\beta\to\beta$}; 
\end{scope}

\begin{scope}[xshift=36cm, yshift=0cm, scale=2.5]
\foreach \i in {2,3,4}{
\path[draw, thin] (\i,0)--(\i,4);}
\foreach \j in {0,1,2,3,4}{
\path[draw, densely dotted] (1,\j)--(5,\j);} 
\foreach \j in {0,1,2,3,4}{
\path[draw, thin] (2,\j)--(3,\j);} 
\path[draw, very thick]  (1,0)--(3,0)--(3,4)--(1,4)  ;
\draw (2.5,4.5) node {$j$};
\draw (1.5,1.65) node {$0$};
\draw (1.5,1.25) node {\tiny{safe}};
\draw (1.5,2.5) node {$c$};
\draw (2.5,0.5) node {$c$};
\draw (2.5,1.5) node {$\blue{c'}$};
\draw (2.5,2.5) node {$0$};
\draw (2.5,3.5) node {$c$};
\draw (3.5,0.5) node {$0$};
\draw (3.5,1.5) node {$0$};
\draw (3.5,2.5) node {$\red{c'}$};
\draw (3.5,3.5) node {$0$};
\draw (4.5,0.5) node {$0$};
\draw (4.5,1.5) node {$0$};
\draw (4.5,2.5) node {$0$};
\draw (4.5,3.5) node {$0$};
\draw (2.8,-.75) node {Case $2\gamma F1$: $\gamma\to\beta$}; 
\end{scope}

\begin{scope}[xshift=54cm, yshift=0cm, scale=2.5]
\foreach \i in {2,3,4}{
\path[draw, thin] (\i,0)--(\i,4);}
\foreach \j in {0,1,2,3,4}{
\path[draw, densely dotted] (1,\j)--(5,\j);} 
\foreach \j in {0,1,2,3,4}{
\path[draw, thin] (2,\j)--(3,\j);} 
\path[draw, very thick]  (1,0)--(3,0)--(3,4)--(1,4)  ;
\draw (2.5,4.5) node {$j$};
\draw (1.5,1.65) node {$0$};
\draw (1.5,1.25) node {\tiny{safe}};
\draw (1.5,2.5) node {$c$};
\draw (2.5,0.5) node {$c$};
\draw (2.5,1.5) node {$0$};
\draw (2.5,2.5) node {$\blue{c'}$};
\draw (2.5,3.5) node {$c$};
\draw (3.5,0.5) node {$0$};
\draw (3.5,1.5) node {$\red{c'}$};
\draw (3.5,2.5) node {$0$};
\draw (3.5,3.5) node {$0$};
\draw (4.5,0.5) node {$0$};
\draw (4.5,1.5) node {$0$};
\draw (4.5,2.5) node {$0$};
\draw (4.5,3.5) node {$0$};
\draw (2.8,-.75) node {Case $2\gamma F2$: $\gamma\to\beta$}; 
\end{scope}

\end{tikzpicture}
\caption{Case $2\alpha\beta\gamma F$: Bob just colored a vertex (in blue) of a free border in configuration $\alpha$, $\beta$ or $\gamma$.
Alice's answer is in red. The bold line figures the surrounding of the updated block (before Alice's move). The block increases by one column.}
\label{fig:nextMoveAlice1}
\end{figure}

Following the orientation of Figure \ref{fig:nextMoveAlice1}:
\begin{itemize}
\item {\bf Case $2\alpha F$:} 
If Bob colors a vertex of a border in configuration $\alpha$, Alice colors $v_{2,j+1}$ with $c$, obtaining a border in configuration $\beta$.
\item {\bf Case $2\beta F$:} 
If Bob colors a vertex of a border in configuration $\beta$, Alice colors $v_{2,j+1}$ with $c$, obtaining a border in configuration $\beta$.
\item {\bf Case $2\gamma F1$:} 
If Bob plays $c'\ne c$ in the vertex $v_{2,j}$ of a border in configuration $\gamma$, Alice colors $v_{3,j+1}$ with $c'$, obtaining a border in configuration $\beta$.
\item {\bf Case $2\gamma F2$:} 
If Bob plays $c'\ne c$ in the vertex $v_{3,j}$ of a border in configuration $\gamma$, Alice colors $v_{2,j+1}$ with $c'$, obtaining a border in configuration $\beta$.
\end{itemize}

\item {\bf Case $2\beta$ (Non-Free border in Configuration $\beta$):} Bob colors a vertex $v$ of a non-free border of a block $B$ in configuration $\beta$. 
See Figure \ref{fig:nextMoveAliceXb} (the symmetric cases can be dealt with accordingly). The blue squares are the possible vertices that Bob colored.
If Bob colored a vertex of column $j$ other than $v_{2,j}$ (Case $2\beta 1$), Alice colors $v_{3,j+1}$ with any available color $c'$ and we are done, since $v_{2,j+1}$ is sound (with doctors $v_{2,j}$ and $v_{1,j+1}$).
Thus assume Bob just colored $v_{2,j}$ with $c'$.
If $v_{3,j+2}$ is not colored $c'$ (Case $2\beta 2$), Alice colors $v_{3,j+1}$ with $c'$ making $v_{2,j+1}$ and $v_{3,j+1}$ safe.
Otherwise, if $v_{4,j}$ is not colored $c'$ (Case $2\beta 3$), Alice colors $v_{4,j+1}$ with $c'$ making $v_{3,j+1}$ safe and $v_{2,j+1}$ sound (with doctors $v_{1,j+1}$ and $v_{3,j+1}$).
Finally, if $v_{4,j}$ is colored $c'$ (Case $2\beta 4$), then this is equivalent to the situation in which $v_{2,j}$ and $v_{4,j}$ were colored $c'$ (configuration $\alpha$) and Bob just colored $v_{3,j}$ with $c$, which is treated in the next case $2\alpha2$ regarding the configuration $\alpha$.

\begin{figure}[ht!]\centering
\begin{tikzpicture}[scale=0.195]
\tikzstyle{vertex}=[draw,circle,fill=black,inner sep=0pt, minimum size=1pt]

\begin{scope}[xshift=0cm, yshift=0cm, scale=2.5]
\fill[blue!20] (1,0) rectangle (2,1) ;
\fill[blue!20] (1,3) rectangle (2,4);
\foreach \i in {1,2,3}{
\path[draw, thin] (\i,0)--(\i,4);}
\foreach \j in {0,1,2,3,4}{ 
\path[draw, densely dotted] (0,\j)--(4,\j);} 
\foreach \j in {0,1,2,3,4}{ 
\path[draw, thin] (1,\j)--(3,\j);} 
\path[draw, very thick] (0,0)--(2,0)--(2,4)--(0,4);
\path[draw, very thick] (4,0)--(3,0)--(3,4)--(4,4);
\draw (1.5,4.5) node {$j$};
\draw (0.5,1.5) node {$c$};
\draw (1.5,2.5) node {$c$};
\draw (2.5,2.5) node {\red{$c'$}};
\draw (1.5,1.5) node {$0$};
\draw (2.5,0.5) node {$0$};
\draw (2.5,1.5) node {$0$};
\draw (2.5,3.5) node {$0$};
\draw (2.0,-.75) node {Case $2\beta 1$};
\draw (2.0,-1.5) node {merge};
\end{scope}

\begin{scope}[xshift=17cm, yshift=0cm, scale=2.5]
\foreach \i in {1,2,3}{
\path[draw, thin] (\i,0)--(\i,4);}
\foreach \j in {0,1,2,3,4}{ 
\path[draw, densely dotted] (0,\j)--(4,\j);} 
\foreach \j in {0,1,2,3,4}{ 
\path[draw, thin] (1,\j)--(3,\j) ;} 
\path[draw, very thick] (0,0)--(2,0)--(2,4)--(0,4);
\path[draw, very thick] (4,0)--(3,0)--(3,4)--(4,4);
\draw (1.5,4.5) node {$j$};
\draw (0.5,1.5) node {$c$};
\draw (1.5,2.5) node {$c$};
\draw (1.5,1.5) node {\blue{$c'$}};
\draw (3.75,2.5) node {$\neq c'$};
\draw (2.5,0.5) node {$0$};
\draw (2.5,2.5) node {\red{$c'$}};
\draw (2.5,3.5) node {$0$};
\draw (2.5,1.5) node {$0$};
\draw (2.0,-.75) node {Case $2\beta 2$};
\draw (2.0,-1.5) node {merge};
\end{scope}

\begin{scope}[xshift=34cm, yshift=0cm, scale=2.5]
\foreach \i in {2,3}{
\path[draw, thin] (\i,0)--(\i,4);}
\path[draw, densely dotted] (1,0)--(1,4);
\foreach \j in {0,1,2,3,4}{ 
\path[draw, densely dotted] (0,\j)--(4,\j) ;} 
\foreach \j in {0,1,2,3,4}{ 
\path[draw, thin] (1,\j)--(3,\j) ;} 
\path[draw, very thick] (0,0)--(2,0)--(2,4)--(0,4);
\path[draw, very thick] (4,0)--(3,0)--(3,4)--(4,4);
\draw (1.5,4.5) node {$j$};
\draw (0.5,1.5) node {$c$};
\draw (1.5,2.5) node {$c$};
\draw (1.25,3.5) node {$\neq c'$};
\draw (1.5,1.5) node {\blue{$c'$}};
\draw (3.5,2.5) node {$c'$};
\draw (2.5,0.5) node {$0$};
\draw (2.5,2.5) node {$0$};
\draw (2.5,3.5) node {\red{$c'$}};
\draw (2.5,1.5) node {$0$};
\draw (2.0,-.75) node {Case $2\beta 3$};
\draw (2.0,-1.5) node {merge};
\end{scope}

\begin{scope}[xshift=51cm, yshift=0cm, scale=2.5]
\foreach \i in {2,3}{
\path[draw, thin] (\i,0)--(\i,4);}
\path[draw, densely dotted] (1,0)--(1,4);
\foreach \j in {0,1,2,3,4}{ 
\path[draw, densely dotted] (0,\j)--(4,\j) ;} 
\foreach \j in {0,1,2,3,4}{ 
\path[draw, thin] (1,\j)--(3,\j) ;} 
\path[draw, very thick] (0,0)--(2,0)--(2,4)--(0,4);
\path[draw, very thick] (4,0)--(3,0)--(3,4)--(4,4);
\draw (1.5,4.5) node {$j$};
\draw (0.5,1.5) node {$c$};
\draw (1.5,2.5) node {$c$};
\draw (1.5,3.5) node {$c'$};
\draw (1.5,1.5) node {\blue{$c'$}};
\draw (3.5,2.5) node {$c'$};
\draw (2.5,0.5) node {$0$};
\draw (2.5,2.5) node {$0$};
\draw (2.5,1.5) node {$0$};
\draw (2.0,-.75) node {Case $2\beta 4$:};
\draw (2.0,-1.5) node {treated as Case $2\alpha2$};
\end{scope}

\end{tikzpicture}
\caption{Case $2\beta$: Bob colored a vertex of a non-free border in configuration $\beta$ at column $j$. Alice merge the two blocks in her next move.}
\label{fig:nextMoveAliceXb}
\end{figure}
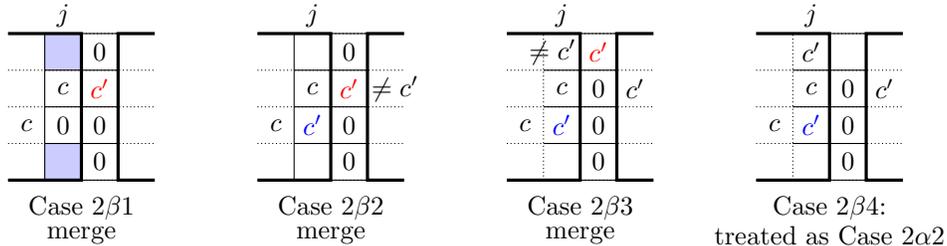

\item  {\bf Case $2\gamma$ (Non-Free border in Configuration $\gamma$):} Bob colors a vertex $v$ of a non-free border of a block $B$ in configuration $\gamma$ at column $j$. 
See Figure \ref{fig:nextMoveAliceXc} (the symmetric cases can be dealt with accordingly).
Alice plays, merging the two blocks (with all vertices safe or sound) or obtaining the configuration $\pi$.

\begin{figure}[ht!]\centering
\begin{tikzpicture}[scale=0.195]
\tikzstyle{vertex}=[draw,circle,fill=black,inner sep=0pt, minimum size=1pt]

\begin{scope}[xshift=0cm, yshift=0cm, scale=2.5]
\foreach \i in {3,4,5}{
\path[draw, thin] (\i,0)--(\i,4);}
\foreach \j in {0,1,2,3,4}{
\path[draw, densely dotted] (2,\j)--(6,\j);} 
\foreach \j in {0,1,2,3,4}{ 
\path[draw, thin] (3,\j)--(5,\j);}
\path[draw, very thick]  (2,0)--(4,0)--(4,4)--(2,4);
\path[draw, very thick]  (6,0)--(5,0)--(5,4)--(6,4);
\draw (3.5,4.5) node {$j$};
\draw (2.5,2.5) node {$c$};
\draw (2.5,1.65) node {$0$};
\draw (2.5,1.25) node {\tiny{safe}};
\draw (3.5,0.5) node {$c$};
\draw (3.5,1.5) node {$\blue{c'}$};
\draw (3.5,2.5) node {$0$};
\draw (3.5,3.5) node {$c$};
\draw (4.5,0.5) node {$0$};
\draw (4.5,1.5) node {$\red{c''}$};
\draw (4.5,2.5) node {$0$};
\draw (4.5,3.5) node {$0$};
\draw (3.5,-0.75) node {Case $2\gamma 1$: merge}; 
\end{scope}

\begin{scope}[xshift=25cm, yshift=0cm, scale=2.5]
\foreach \i in {3,4,5}{
\path[draw, thin] (\i,0)--(\i,4);}
\foreach \j in {0,1,2,3,4}{
\path[draw, densely dotted] (2,\j)--(6,\j);} 
\foreach \j in {0,1,2,3,4}{ 
\path[draw, thin] (3,\j)--(5,\j);}
\path[draw, very thick]  (2,0)--(4,0)--(4,4)--(2,4);
\path[draw, very thick]  (6,0)--(5,0)--(5,4)--(6,4);
\draw (3.5,4.5) node {$j$};
\draw (2.5,2.5) node {$c$};
\draw (2.5,1.65) node {$0$};
\draw (2.5,1.25) node {\tiny{safe}};
\draw (3.5,0.5) node {$c$};
\draw (3.5,1.5) node {$0$};
\draw (3.5,2.5) node {$\blue{c'}$};
\draw (3.5,3.5) node {$c$};
\draw (4.5,0.5) node {$0$};
\draw (4.5,1.5) node {$\red{c'}$};
\draw (4.5,2.5) node {$0$};
\draw (4.5,3.5) node {$0$};
\draw (5.8,1.5) node {$\ne c'$};
\draw (3.5,-0.75) node {Case $2\gamma 2$: merge}; 
\end{scope}

\begin{scope}[xshift=50cm, yshift=0cm, scale=2.5]
\foreach \i in {3,4,5}{
\path[draw, thin] (\i,0)--(\i,4);}
\foreach \j in {0,1,2,3,4}{
\path[draw, densely dotted] (2,\j)--(6,\j);} 
\foreach \j in {0,1,2,3,4}{ 
\path[draw, thin] (3,\j)--(5,\j);}
\path[draw, very thick]  (2,0)--(4,0)--(4,4)--(2,4);
\path[draw, very thick]  (6,0)--(5,0)--(5,4)--(6,4);
\draw (3.5,4.5) node {$j$};
\draw (2.5,2.5) node {$c$};
\draw (2.5,1.65) node {$0$};
\draw (2.5,1.25) node {\tiny{safe}};
\draw (3.5,0.5) node {$c$};
\draw (3.5,1.5) node {$\red{c''}$};
\draw (3.5,2.5) node {$\blue{c'}$};
\draw (3.5,3.5) node {$c$};
\draw (4.5,0.5) node {$0$};
\draw (4.5,1.5) node {$0$};
\draw (4.5,2.5) node {$0$};
\draw (4.5,3.5) node {$0$};
\draw (5.5,1.5) node {$c'$};
\draw (3.5,-0.75) node {Case $2\gamma 3$: $\gamma\to\pi$}; 
\end{scope}

\end{tikzpicture}
\caption{Case $2\gamma$: Bob colors a vertex of border in configuration $\gamma$ at column $j$.}
\label{fig:nextMoveAliceXc}
\end{figure}
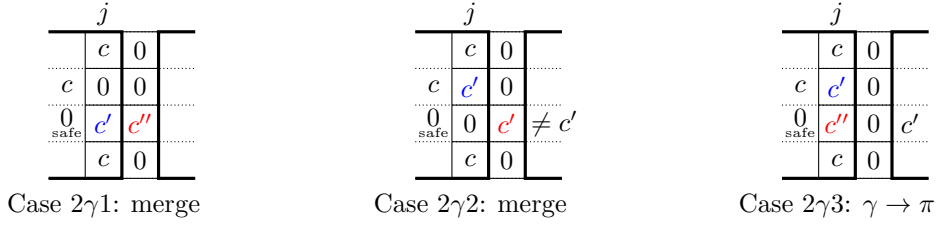

\begin{itemize}
\item {\bf Case $2\gamma 1$:} 
If Bob colors $v_{2,j}$, then Alice colors $v_{2,j+1}$ with any available color $c''$, making $v_{3,j+1}$ either safe or sound (with doctors $v_{3,j}$ and $v_{4,j+1}$).
\item {\bf Case $2\gamma 2$:} 
If Bob colors $v_{3,j}$ with $c'$ and $v_{2,j+2}$ is not colored $c'$, Alice colors $v_{2,j+1}$ with $c'$, making $v_{2,j}$ and $v_{3,j+1}$ safe.
\item {\bf Case $2\gamma 3$:} 
Otherwise, Bob colors $v_{3,j}$ with $c'$ and $v_{2,j+2}$ is colored $c'$.
Then Alice colors $v_{2,j}$ with $c''$ such that $c''\ne c$ and $c''\ne c'$, obtaining a border in configuration $\pi$. 
\end{itemize}

\item {\bf Case $2\alpha$ (Non-Free border in Configuration $\alpha$ of a block with at least 2 columns):} Bob colors a vertex $v$ of a non-free border in configuration $\alpha$ of a block $B$ with at least 2 columns.
The different possibilities are depicted in Figure~\ref{fig:nextMoveAliceX} (all symmetric cases can be dealt with accordingly).
If we need to specify which vertex Bob colored, it is depicted in blue. Alice' answer is in red in some subcases. The bold line figures the surrounding of the updated block (before Alice's move). In each case, two blocks are merged into one.

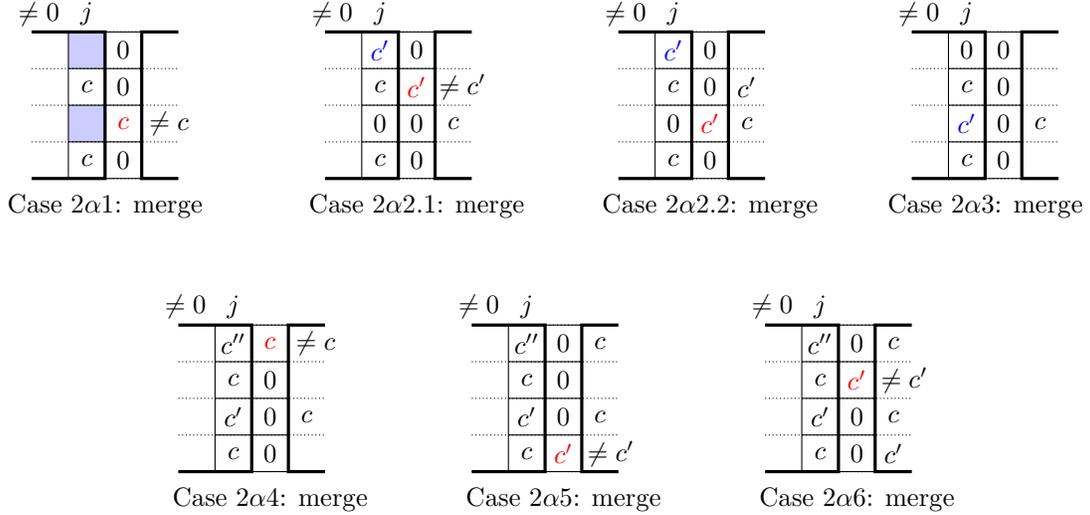
\begin{figure}[ht!]\centering
\begin{tikzpicture}[scale=0.195]
\tikzstyle{vertex}=[draw,circle,fill=black,inner sep=0pt, minimum size=1pt]

\begin{scope}[xshift=0cm, yshift=0cm, scale=2.5]
\fill[blue!20] (1,1) rectangle (2,2);
\fill[blue!20] (1,3) rectangle (2,4);
\foreach \i in {1,2,3}{
\path[draw, thin] (\i,0)--(\i,4);}
\foreach \j in {0,1,2,3,4}{ 
\path[draw, densely dotted] (0,\j)--(4,\j);} 
\foreach \j in {0,1,2,3,4}{ 
\path[draw, thin] (1,\j)--(3,\j);} 
\path[draw, very thick] (0,0)--(2,0)--(2,4)--(0,4);
\path[draw, very thick] (4,0)--(3,0)--(3,4)--(4,4);
\draw (0.2,4.5) node {$\ne 0$};
\draw (1.5,4.5) node {$j$};
\draw (1.5,2.5) node {$c$};
\draw (1.5,0.5) node {$c$};
\draw (2.5,0.5) node {$0$};
\draw (2.5,3.5) node {$0$};
\draw (2.5,2.5) node {$0$};
\draw (2.5,1.5) node {\red{$c$}};
\draw (3.75,1.5) node {$\neq c$};
\draw (2,-0.75) node {Case $2\alpha 1$: merge}; 
\end{scope}

\begin{scope}[xshift=20cm, yshift=0cm, scale=2.5]
\foreach \i in {1,2,3}{
\path[draw, thin] (\i,0)--(\i,4);}
\foreach \j in {0,1,2,3,4}{ 
\path[draw, densely dotted] (0,\j)--(4,\j);} 
\foreach \j in {0,1,2,3,4}{ 
\path[draw, thin] (1,\j)--(3,\j);} 
\path[draw, very thick] (0,0)--(2,0)--(2,4)--(0,4);
\path[draw, very thick] (4,0)--(3,0)--(3,4)--(4,4);
\draw (0.2,4.5) node {$\ne 0$};
\draw (1.5,4.5) node {$j$};
\draw (1.5,0.5) node {$c$};
\draw (1.5,1.5) node {$0$};
\draw (1.5,2.5) node {$c$};
\draw (1.5,3.5) node {\blue{$c'$}};
\draw (2.5,0.5) node {$0$};
\draw (2.5,1.5) node {$0$};
\draw (2.5,2.5) node {\red{$c'$}};
\draw (2.5,3.5) node {$0$};
\draw (3.5,1.5) node {$c$};
\draw (3.75,2.5) node {$\neq c'$};
\draw (2.5,-0.75) node {Case $2\alpha 2.1$: merge};
\end{scope}

\begin{scope}[xshift=40cm, yshift=0cm, scale=2.5]
\foreach \i in {1,2,3}{
\path[draw, thin] (\i,0)--(\i,4);}
\foreach \j in {0,1,2,3,4}{ 
\path[draw, densely dotted] (0,\j)--(4,\j);} 
\foreach \j in {0,1,2,3,4}{ 
\path[draw, thin] (1,\j)--(3,\j);} 
\path[draw, very thick] (0,0)--(2,0)--(2,4)--(0,4);
\path[draw, very thick] (4,0)--(3,0)--(3,4)--(4,4);
\draw (0.2,4.5) node {$\ne 0$};
\draw (1.5,4.5) node {$j$};
\draw (1.5,0.5) node {$c$};
\draw (1.5,1.5) node {$0$};
\draw (1.5,2.5) node {$c$};
\draw (1.5,3.5) node {\blue{$c'$}};
\draw (2.5,0.5) node {$0$};
\draw (2.5,1.5) node {\red{$c'$}};
\draw (2.5,2.5) node {$0$};
\draw (2.5,3.5) node {$0$};
\draw (3.5,1.5) node {$c$};
\draw (3.5,2.5) node {$c'$};
\draw (2.5,-0.75) node {Case $2\alpha 2.2$: merge};
\end{scope}

\begin{scope}[xshift=60cm, yshift=0cm, scale=2.5]
\foreach \i in {1,2,3}{
\path[draw, thin] (\i,0)--(\i,4);}
\foreach \j in {0,1,2,3,4}{ 
\path[draw, densely dotted](0,\j)--(4,\j) ;} 
\foreach \j in {0,1,2,3,4}{ 
\path[draw, thin] (1,\j)--(3,\j);} 
\path[draw, very thick] (0,0)--(2,0)--(2,4)--(0,4);
\path[draw, very thick] (4,0)--(3,0)--(3,4)--(4,4);
\draw (0.2,4.5) node {$\ne 0$};
\draw (1.5,4.5) node {$j$};
\draw (1.5,0.5) node {$c$};
\draw (1.5,1.5) node {$\blue{c'}$};
\draw (1.5,2.5) node {$c$};
\draw (1.5,3.5) node {$0$};
\draw (2.5,0.5) node {$0$};
\draw (2.5,1.5) node {$0$};
\draw (2.5,2.5) node {$0$};
\draw (2.5,3.5) node {$0$};
\draw (3.5,0.5) node {};
\draw (3.5,1.5) node {$c$};
\draw (3.5,2.5) node {};
\draw (3.5,3.5) node {};
\draw (2,-0.75) node {Case $2\alpha 3$: merge};
\end{scope}

\begin{scope}[xshift=10cm, yshift=-20cm, scale=2.5]
\foreach \i in {1,2,3}{
\path[draw, thin] (\i,0)--(\i,4);}
\foreach \j in {0,1,2,3,4}{ 
\path[draw, densely dotted] (0,\j)--(4,\j) ;} 
\foreach \j in {0,1,2,3,4}{ 
\path[draw, thin] (1,\j)--(3,\j) ;} 
\path[draw, very thick] (0,0)--(2,0)--(2,4)--(0,4);
\path[draw, very thick] (4,0)--(3,0)--(3,4)--(4,4);
\draw (0.2,4.5) node {$\ne 0$};
\draw (1.5,4.5) node {$j$};
\draw (1.5,0.5) node {$c$};
\draw (1.5,1.5) node {$c'$};
\draw (1.5,3.5) node {$c''$};
\draw (1.5,2.5) node {$c$};
\draw (2.5,0.5) node {$0$};
\draw (2.5,1.5) node {$0$};
\draw (2.5,2.5) node {$0$};
\draw (2.5,3.5) node {$\red{c}$};
\draw (3.5,1.5) node {$c$};
\draw (3.75,3.5) node {$\neq c$};
\draw (2.5,-0.75) node {Case $2\alpha 4$: merge}; 
\end{scope}

\begin{scope}[xshift=30cm, yshift=-20cm, scale=2.5]
\foreach \i in {1,2,3}{
\path[draw, thin] (\i,0)--(\i,4);}
\foreach \j in {0,1,2,3,4}{ 
\path[draw, densely dotted] (0,\j)--(4,\j) ;} 
\foreach \j in {0,1,2,3,4}{ 
\path[draw, thin] (1,\j)--(3,\j) ;} 
\path[draw, very thick] (0,0)--(2,0)--(2,4)--(0,4);
\path[draw, very thick] (4,0)--(3,0)--(3,4)--(4,4);
\draw (0.2,4.5) node {$\ne 0$};
\draw (1.5,4.5) node {$j$};
\draw (1.5,0.5) node {$c$};
\draw (1.5,1.5) node {$c'$};
\draw (1.5,3.5) node {$c''$};
\draw (1.5,2.5) node {$c$};
\draw (2.5,0.5) node {$\red{c'}$};
\draw (2.5,1.5) node {$0$};
\draw (2.5,2.5) node {$0$};
\draw (2.5,3.5) node {$0$};
\draw (3.8,0.5) node {$\ne c'$};
\draw (3.5,1.5) node {$c$};
\draw (3.5,3.5) node {$c$};
\draw (2.5,-0.75) node {Case $2\alpha 5$: merge}; 
\end{scope}

\begin{scope}[xshift=50cm, yshift=-20cm, scale=2.5]
\foreach \i in {1,2,3}{
\path[draw, thin] (\i,0)--(\i,4);}
\foreach \j in {0,1,2,3,4}{ 
\path[draw, densely dotted] (0,\j)--(4,\j) ;} 
\foreach \j in {0,1,2,3,4}{ 
\path[draw, thin] (1,\j)--(3,\j) ;} 
\path[draw, very thick] (0,0)--(2,0)--(2,4)--(0,4);
\path[draw, very thick] (4,0)--(3,0)--(3,4)--(4,4);
\draw (0.2,4.5) node {$\ne 0$};
\draw (1.5,4.5) node {$j$};
\draw (1.5,0.5) node {$c$};
\draw (1.5,1.5) node {$c'$};
\draw (1.5,3.5) node {$c''$};
\draw (1.5,2.5) node {$c$};
\draw (2.5,0.5) node {$0$};
\draw (2.5,1.5) node {$0$};
\draw (2.5,2.5) node {$\red{c'}$};
\draw (2.5,3.5) node {$0$};
\draw (3.5,0.5) node {$c'$};
\draw (3.5,1.5) node {$c$};
\draw (3.8,2.5) node {$\ne c'$};
\draw (3.5,3.5) node {$c$};
\draw (2.5,-0.75) node {Case $2\alpha 6$: merge}; 
\end{scope}

\end{tikzpicture}
\caption{Case $2\alpha$: Bob colors a vertex of a border in configuration $\alpha$ at column $j$ with a color distinct from $c$, in a block containing at least two columns. In cases $2\alpha 4$,  $2\alpha 5$ and  $2\alpha 6$, $c''$ is any color distinct from $c$ (in particular, $c''$ might be equal to $c'$).} 
\label{fig:nextMoveAliceX}
\end{figure}

\begin{itemize}
\item {\bf Case $2\alpha 1$:} If $v_{2,j+2}$ is not colored $c$,
Alice colors $v_{2,j+1}$ with $c$ making $v_{2,j+1}$ and $v_{3,j+1}$ safe.
\end{itemize}
From now on, let us assume that $c(v_{2,j+2})=c$.
\begin{itemize}
\item {\bf Case $2\alpha 2$:}
If Bob colors $v_{4,j}$ with $c'$ and $v_{2,j}$ is not colored yet, then Alice colors $v_{3,j+1}$ with $c'$ if available (case $2\alpha 2.1$), making $v_{3,j+1}$ safe and $v_{2,j+1}$ sound (with doctors $v_{2,j}$ and $v_{1,j+1}$).
If $c'$ is unavailable (case $2\alpha 2.2$), meaning that $v_{3,j+2}$ is colored $c'$, then Alice can color $v_{2,j+1}$ with color $c'$, making $v_{2,j+1}$ and $v_{3,j+1}$ safe.
\item {\bf Case $2\alpha 3$:} Assume that Bob colored $v_{2,j}$ with color $c'$ and $v_{4,j}$ is not colored yet. If $v_{3,j+2}$ is not colored $c'$, Alice colors $v_{3,j+1}$ with $c'$ making $v_{2,j+1}$ and $v_{3,j+1}$ safe. Otherwise, Alice colors $v_{4,j+1}$ with $c'$ making $v_{3,j+1}$ safe and $v_{2,j+1}$ sound (with doctors $v_{1,j+1}$ and $v_{3,j+1}$).
\end{itemize}

In the next three subcases, assume that after Bob's move, all four vertices of column $j$ are colored.

\begin{itemize}
\item {\bf Case $2\alpha 4$:} 
If $v_{4,j+2}$ is not colored $c$, Alice colors $v_{4,j+1}$ with $c$, making $v_{3,j+1}$ safe and $v_{2,j+1}$ sound (with doctors $v_{1,j+1}$ and $v_{3,j+1}$).

\item {\bf Case $2\alpha 5$:} 
If $v_{1,j+2}$ is not colored $c'$, Alice colors $v_{1,j+1}$ with $c'$, making $v_{2,j+1}$ safe and $v_{3,j+1}$ sound (with doctors $v_{2,j+1}$ and $v_{4,j+1}$).

\item {\bf Case $2\alpha 6$:} 
By Property $(3)$ of the induction hypothesis, $v_{3,j+2}$ is not colored $c'$. Then Alice colors $v_{3,j+1}$ with $c'$, making $v_{2,j+1}$ and $v_{3,j+1}$ safe.
\end{itemize}

\item {\bf Case $2\alpha'$ (Non-Free border in Configuration $\alpha$ of a 1-column block):} Bob colors a vertex $v$ of a non-free border of a 1-column block $B$ (column $j$) in configuration $\alpha$. Again, $B$ is considered as a left border. 
Note that, by Property $(3)$, before Bob's move, column $j$ has exactly two colored vertices.
The different possibilities are depicted in Figure~\ref{fig:nextMoveAliceX} (all symmetric cases, according to the horizontal symmetry axis of the grid, can be dealt with accordingly).
If we need to specify which vertex Bob colored, it is depicted in blue. Alice' answer is in red in some subcases. The bold line figures the surrounding of the updated block (before Alice's move). So, in all cases except one subcase of $2\alpha'3$, two blocks are merged into one.

\begin{figure}[ht!]\centering
\begin{tikzpicture}[scale=0.195]
\tikzstyle{vertex}=[draw,circle,fill=black,inner sep=0pt, minimum size=1pt]

\begin{scope}[xshift=0cm, yshift=0cm, scale=2.5]
\foreach \i in {2,3,4}{
\path[draw, thin] (\i,0)--(\i,4);}
\foreach \j in {0,1,2,3,4}{ 
\path[draw, densely dotted] (0,\j)--(5,\j);} 
\foreach \j in {0,1,2,3,4}{ 
\path[draw, thin] (1,\j)--(4,\j);} 
\path[draw, very thick] (0,0)--(2,0)--(2,4)--(0,4);
\path[draw, thin] (1,0)--(1,2);
\path[draw, very thick] (4,0)--(3,0)--(3,4)--(4,4)--(4,0);
\draw (3.5,4.5) node {$j$};
\draw (4.5,0.5) node {$0$};
\draw (4.5,1.5) node {$0$};
\draw (4.5,2.5) node {$0$};
\draw (4.5,3.5) node {$0$};
\draw (3.5,0.5) node {$c$};
\draw (3.5,1.5) node {$0$};
\draw (3.5,2.5) node {$c$};
\draw (3.5,3.5) node {\blue{$c'$}};
\draw (2.5,0.5) node {$0$};
\draw (2.5,1.5) node {$0$};
\draw (2.5,2.5) node {$0$};
\draw (2.5,3.5) node {$0$};
\draw (1.5,0.5) node {};
\draw (1.5,1.5) node {$c$};
\draw (0.8,2.5) node {or $\ne 0$};
\draw (1.25,3.5) node {$\ne c$};
\draw (2.5,-0.75) node {Case $2\alpha' 1$: merge};
\end{scope}

\begin{scope}[xshift=17cm, yshift=0cm, scale=2.5]
\foreach \i in {1,2,3,4}{
\path[draw, thin] (\i,0)--(\i,4);}
\foreach \j in {0,1,2,3,4}{ 
\path[draw, densely dotted] (0,\j)--(5,\j);} 
\foreach \j in {0,1,2,3,4}{ 
\path[draw, thin] (1,\j)--(4,\j);} 
\path[draw, very thick] (0,0)--(2,0)--(2,4)--(0,4);
\path[draw, very thick] (4,0)--(3,0)--(3,4)--(4,4)--(4,0);
\draw (0.2,4.5) node {$\ne 0$};
\draw (3.5,4.5) node {$j$};
\draw (4.5,0.5) node {$0$};
\draw (4.5,1.5) node {$0$};
\draw (4.5,2.5) node {$0$};
\draw (4.5,3.5) node {$0$};
\draw (3.5,0.5) node {$c$};
\draw (3.5,1.5) node {$0$};
\draw (3.5,2.5) node {$c$};
\draw (3.5,3.5) node {\blue{$c'$}};
\draw (2.5,0.5) node {$0$};
\draw (2.5,1.5) node {\red{$c'$}};
\draw (2.5,2.5) node {$0$};
\draw (2.5,3.5) node {$0$};
\draw (1.5,0.5) node {};
\draw (1.5,1.5) node {$c$};
\draw (1.5,2.5) node {$0$};
\draw (1.5,3.5) node {$c$};
\draw (2.5,-0.75) node {Case $2\alpha' 2$: merge};
\end{scope}

\begin{scope}[xshift=34cm, yshift=0cm, scale=2.5]
\foreach \i in {1,2,3,4,5}{
\path[draw, thin] (\i,0)--(\i,4);}
\foreach \j in {0,1,2,3,4}{ 
\path[draw, densely dotted] (0,\j)--(6,\j);} 
\foreach \j in {0,1,2,3,4}{ 
\path[draw, thin] (1,\j)--(5,\j);} 
\path[draw, very thick] (1,0)--(2,0)--(2,4)--(1,4)--(1,0);
\path[draw, very thick] (4,0)--(3,0)--(3,4)--(4,4)--(4,0);
\draw (3.5,4.5) node {$j$};
\draw (0.5,0.5) node {$0$};
\draw (0.5,1.5) node {$0$};
\draw (0.5,2.5) node {$0$};
\draw (0.5,3.5) node {$0$};
\draw (4.5,0.5) node {$0$};
\draw (4.5,1.5) node {$\red{c}$};
\draw (4.5,2.5) node {$0$};
\draw (4.5,3.5) node {$0$};
\draw (5.7,1.5) node {$\ne c$};
\draw (3.5,0.5) node {$c$};
\draw (3.5,1.5) node {$0$};
\draw (3.5,2.5) node {$c$};
\draw (3.5,3.5) node {{\blue{$c'$}}};
\draw (2.5,0.5) node {$0$};
\draw (2.5,1.5) node {$0$};
\draw (2.5,2.5) node {$0$};
\draw (2.5,3.5) node {$0$};
\draw (1.5,0.5) node {$0$};
\draw (1.5,1.5) node {$c$};
\draw (1.5,2.5) node {$0$};
\draw (1.5,3.5) node {$c$};
\draw (2.5,-0.75) node {Case $2\alpha' 3$: $\Lambda+\beta$};
\draw (2.5,-1.5) node {or $\Lambda$+merge};
\end{scope}

\begin{scope}[xshift=53cm, yshift=0cm, scale=2.5]
\foreach \i in {1,2,3,4,5}{
\path[draw, thin] (\i,0)--(\i,4);}
\foreach \j in {0,1,2,3,4}{ 
\path[draw, densely dotted] (0,\j)--(6,\j);} 
\foreach \j in {0,1,2,3,4}{ 
\path[draw, thin] (1,\j)--(5,\j);} 
\path[draw, very thick] (1,0)--(2,0)--(2,4)--(1,4)--(1,0);
\path[draw, very thick] (4,0)--(3,0)--(3,4)--(4,4)--(4,0);
\path[draw, very thick] (6,0)--(5,0)--(5,4)--(6,4);
\draw (3.5,4.5) node {$j$};
\draw (0.5,0.5) node {$0$};
\draw (0.5,1.5) node {$0$};
\draw (0.5,2.5) node {$0$};
\draw (0.5,3.5) node {$0$};
\draw (4.5,0.5) node {$0$};
\draw (4.5,1.5) node {$0$};
\draw (4.5,2.5) node {$0$};
\draw (5.6,2.5) node {\tiny{safe}};
\draw (4.5,3.5) node {$\red{c}$};
\draw (5.7,3.5) node {$\ne c$};
\draw (5.5,2.5) node {};
\draw (5.5,1.5) node {$c$};
\draw (3.5,0.5) node {$c$};
\draw (3.5,1.5) node {$0$};
\draw (3.5,2.5) node {$c$};
\draw (3.5,3.5) node {{\blue{$c'$}}};
\draw (2.5,0.5) node {$0$};
\draw (2.5,1.5) node {$0$};
\draw (2.5,2.5) node {$0$};
\draw (2.5,3.5) node {$0$};
\draw (1.5,0.5) node {$0$};
\draw (1.5,1.5) node {$c$};
\draw (1.5,2.5) node {$0$};
\draw (1.5,3.5) node {$c$};
\draw (2.5,-0.75) node {Case $2\alpha' 4$: $\Lambda'$+merge};
\end{scope}

\begin{scope}[xshift=10cm, yshift=-20cm, scale=2.5]
\foreach \i in {1,2,3,4,5,6}{
\path[draw, thin] (\i,0)--(\i,4);}
\foreach \j in {0,1,2,3,4}{ 
\path[draw, densely dotted] (0,\j)--(7,\j);} 
\foreach \j in {0,1,2,3,4}{
\path[draw, thin] (1,\j)--(5,\j);} 
\path[draw, very thick] (1,0)--(2,0)--(2,4)--(1,4)--(1,0);
\path[draw, very thick] (4,0)--(3,0)--(3,4)--(4,4)--(4,0);
\path[draw, very thick] (7,0)--(5,0)--(5,4)--(7,4);
\draw (6.7,4.5) node {$\ne 0$};
\draw (3.5,4.5) node {$j$};
\draw (0.5,0.5) node {$0$};
\draw (0.5,1.5) node {$0$};
\draw (0.5,2.5) node {$0$};
\draw (0.5,3.5) node {$0$};
\draw (4.5,0.5) node {$0$};
\draw (4.5,1.5) node {$\red{c''}$};
\draw (4.5,2.5) node {$0$};
\draw (4.5,3.5) node {$0$};
\draw (5.5,3.5) node {$c$};
\draw (5.5,2.5) node {$0$};
\draw (5.5,1.5) node {$c$};
\draw (3.5,0.5) node {$c$};
\draw (3.5,1.5) node {$0$};
\draw (3.5,2.5) node {$c$};
\draw (3.5,3.5) node {{\blue{$c'$}}};
\draw (2.5,0.5) node {$0$};
\draw (2.5,1.5) node {$0$};
\draw (2.5,2.5) node {$0$};
\draw (2.5,3.5) node {$0$};
\draw (1.5,0.5) node {$0$};
\draw (1.5,1.5) node {$c$};
\draw (1.5,2.5) node {$0$};
\draw (1.5,3.5) node {$c$};
\draw (3.5,-0.75) node {Case $2\alpha' 5$:};
\draw (3.5,-1.5) node {$\Lambda$+merge};
\end{scope}

\begin{scope}[xshift=35cm, yshift=-20cm, scale=2.5]
\foreach \i in {1,2,3,4,5,6}{
\path[draw, thin] (\i,0)--(\i,4);}
\foreach \j in {0,1,2,3,4}{ 
\path[draw, densely dotted] (0,\j)--(7,\j);} 
\foreach \j in {0,1,2,3,4}{ 
\path[draw, thin] (1,\j)--(6,\j);} 
\path[draw, very thick] (1,0)--(2,0)--(2,4)--(1,4)--(1,0);
\path[draw, very thick] (4,0)--(3,0)--(3,4)--(4,4)--(4,0);
\path[draw, very thick] (6,0)--(5,0)--(5,4)--(6,4)--(6,0);
\draw (3.5,4.5) node {$j$};
\draw (6.5,0.5) node {$0$};
\draw (6.5,1.5) node {$0$};
\draw (6.5,2.5) node {$0$};
\draw (6.5,3.5) node {$0$};
\draw (0.5,0.5) node {$0$};
\draw (0.5,1.5) node {$0$};
\draw (0.5,2.5) node {$0$};
\draw (0.5,3.5) node {$0$};
\draw (4.5,0.5) node {$0$};
\draw (4.5,1.5) node {$\red{c''}$};
\draw (4.5,2.5) node {$0$};
\draw (4.5,3.5) node {$0$};
\draw (5.5,3.5) node {$c$};
\draw (5.5,2.5) node {$0$};
\draw (5.5,1.5) node {$c$};
\draw (5.5,0.5) node {$0$};
\draw (3.5,0.5) node {$c$};
\draw (3.5,1.5) node {$0$};
\draw (3.5,2.5) node {$c$};
\draw (3.5,3.5) node {{\blue{$c'$}}};
\draw (2.5,0.5) node {$0$};
\draw (2.5,1.5) node {$0$};
\draw (2.5,2.5) node {$0$};
\draw (2.5,3.5) node {$0$};
\draw (1.5,0.5) node {$0$};
\draw (1.5,1.5) node {$c$};
\draw (1.5,2.5) node {$0$};
\draw (1.5,3.5) node {$c$};
\draw (3.5,-0.75) node {Case $2\alpha' 6$:};
\draw (3.5,-1.5) node {$\Lambda$+$\Delta$};
\end{scope}

\end{tikzpicture}
\caption{Case $2\alpha'$: Bob colors a vertex of border in configuration $\alpha$ at column $j$ with a color distinct from $c$ of a 1-column block.\label{fig:nextMoveAliceX2}}
\end{figure}
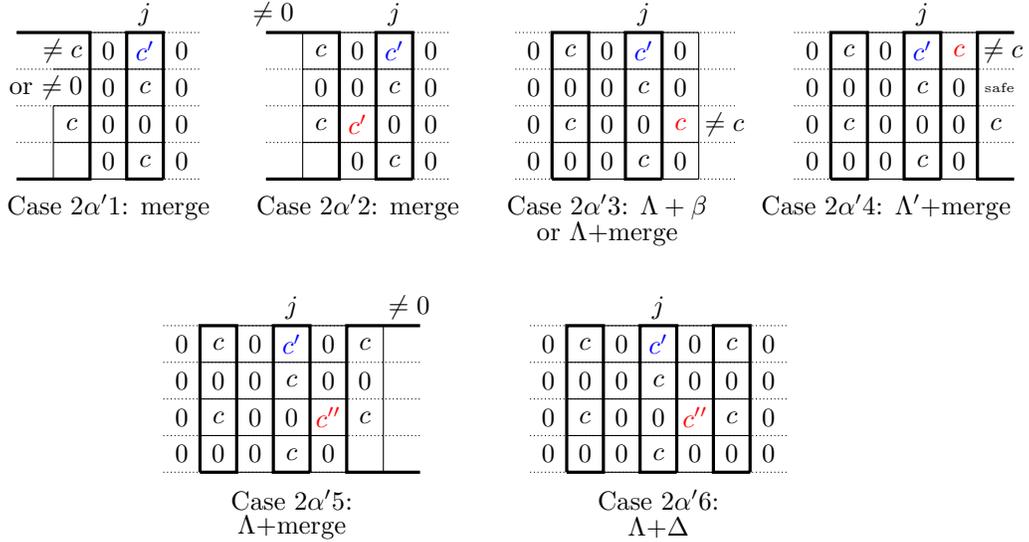

Regardless of what Bob played on column $j$, if $v_{2,j-2}$ is not colored $c$, then Alice can color $v_{2,j-1}$ with $c$, merging the two blocks.
So, in the next cases, we  assume that $v_{2,j-2}$ is colored with color $c$.
Moreover, if Bob colors $v_{2,j}$, Alice can respond similarly as treated in case $2\alpha 3$ (but considering the symmetry according to the vertical axis).
Therefore, we may consider that Bob colors $v_{4,j}$ with $c'$.

\begin{itemize}
\item {\bf Case $2\alpha' 1$:} If $v_{4,j-2}$ is not colored $c$,
Alice colors $v_{4,j-1}$ with $c$ making $v_{3,j-1}$ safe and $v_{2,j-1}$ sound (with doctors $v_{1,j-1}$ and $v_{3,j-1}$). So, assume that $v_{4,j-2}$ is colored $c$. If $v_{3,j-2}$ is colored some color $c''$ (possibly $c'' = c'$), Alice colors $v_{2,j-1}$ with $c'$, making $v_{2,j-1}$ and $v_{3,j-1}$ safe.
Otherwise, $v_{3,j-2}$ is uncolored: the next cases deal with this situation.
\end{itemize}

From now on, suppose that $v_{4,j-2}$ is colored $c$ and $v_{3,j-2}$ is uncolored.
\begin{itemize}
\item {\bf Case $2\alpha' 2$:} If column $j-3$ is not empty, then Alice colors $v_{2,j-1}$ with $c'$, making $v_{2,j-1}$ safe and $v_{3,j-1}$ sound (with doctors $v_{3,j-2}$ and $v_{4,j-1}$).
Note that after Alice's move, Column $j-2$ is not a border anymore, so $v_{3,j-2}$ (that was not a doctor before due to Property $(4)$) can be used as a doctor.
\end{itemize}

So, assume that column $j-3$ is empty and consequently, by Property $(5)$, $v_{1,j-2}$ is not colored.

\begin{itemize}
\item {\bf Case $2\alpha' 3$:} If $v_{2,j+2}$ is not colored $c$, Alice colors $v_{2,j+1}$ with $c$, obtaining a configuration $\Lambda$ on the left and merging the blocks on the right (or obtaining a border $j+1$ in configuration $\beta$ if column $j+2$ was empty) with $v_{2,j+1}$ and $v_{3,j+1}$ safe.
\end{itemize}

So assume that $v_{2,j+2}$ is colored $c$. Note that column $j+2$ cannot be in configuration $\gamma$ and so, $v_{3,j+2}$ must be safe.

\begin{itemize}
\item {\bf Case $2\alpha' 4$:} 
If $v_{4,j+2}$ is not colored $c$, then Alice colors $v_{4,j+1}$ with $c$, making $v_{3,j+1}$ safe and $v_{2,j+1}$ sound (with doctors $v_{1,j+1}$ and $v_{3,j+1}$), merging the blocks on the right and also obtaining a configuration $\Lambda'$ on the left.
\end{itemize}

Finally, assume that $v_{4,j+2}$ is colored $c$.
Since column $j+2$ is a left border in configuration $\alpha$, then $v_{3,j+2}$ is not colored by Property $(5)$.

\begin{itemize}
\item {\bf Case $2\alpha' 5$:}
If column $j+3$ is not empty, then Alice colors $v_{2,j+1}$ with $c''\ne c'$, obtaining a configuration $\Lambda$ on the left and merging the blocks on the right, with $v_{2,j+1}$ safe and $v_{3,j+1}$ sound (with doctors $v_{4,j+1}$ and $v_{3,j+2}$ which was not a doctor before, by Property $(4)$).

\item {\bf Case $2\alpha' 6$:} 
So also assume that column $j+3$ is empty and consequently (by Property $(5)$) $v_{1,j+2}$ is not colored. Then Alice colors $v_{2,j+1}$ with $c''\ne c'$, obtaining a configuration $\Lambda$ on the left and a configuration $\Delta$ on the right.
It is important to notice that while the two configurations $\Delta$ and $\Lambda$ share a column, there is no ambiguity on the response of Alice to further Bob's move in these columns.
If Bob plays on column $j-2$, $j-1$, or $j$, his move is considered as Case $1\Lambda$.
On the other hand, if Bob plays on column $j+1$ or $j+2$, then Alice answers as described in Case $1\Delta$.
\end{itemize}

\end{itemize}

Again, in all subcases of Case $2$, after Alice's move, all properties of the induction hypothesis hold.

\medskip

\item[Case 3.] {\bf When Bob colors a vertex of an empty column (not in column $j-1$ of a configuration $\Lambda,\Lambda_2,\Lambda'$ or $\Lambda'_2$).}~\\

\begin{itemize}
\item  {\bf Case $3$-new (New block):} Bob colors a vertex $v_{a,j}$ with color $c$ such that no vertices in columns $j-1,j,j+1$ are colored (or $j=n$ and column $j-1$ is empty). 
Then, Alice colors $v_{b,j}$ (with $|a-b|=2$) with color $c$. This creates a new block (restricted to column $j$) with border in configuration $\alpha$. Then the induction hypothesis still holds.

\item \noindent {\bf Case $3\pi$.} Bob colors a vertex in the empty column of configuration $\pi$ (see Figure~\ref{fig:configs}).
Notice that, in the Alice's responses below, she always plays in a vertex of the column $j+1$.
If Bob colors $v_{2,j+1}$ (resp. $v_{3,j+1}$), Alice colors $v_{3,j+1}$ (resp. $v_{2,j+1}$), and we are done.
If Bob colors $v_{1,j+1}$ with $c'$ or $c''$, Alice colors $v_{3,j+1}$ with any available color, and we are done (all vertices of column $j+1$ are safe) since $v_{2,j+1}$ is safe.
Thus, suppose that Bob colors $v_{1,j+1}$ with $w \notin \{c,c',c''\}$.
If $v_{3,j+2}$ is not colored $w$, Alice colors $v_{3,j+1}$ with $w$, making $v_{2,j+1}$ and $v_{3,j+1}$ safe.
If $v_{3,j+2}$ is colored $w$, Alice colors $v_{3,j+1}$ with $c''$, making $v_{3,j+1}$ and $v_{2,j+1}$ safe.
Finally, suppose that Bob colors $v_{4,j+1}$.
If the color of $v_{4,j+1}$ is $c'$, then $v_{3,j+1}$ is safe and Alice colors $v_{2,j+1}$ with any available color.
If the color of $v_{4,j+1}$ is $w \notin \{c,c',c''\}$, Alice colors $v_{2,j+1}$ with $w$, making $v_{2,j+1}$ and $v_{3,j+1}$ safe.
Thus assume that Bob colored $v_{4,j+1}$ with $c''$.
If $v_{3,j+2}$ is colored $c''$, then $v_{3,j+1}$ is safe and Alice colors $v_{2,j+1}$ with any available color.
If $v_{3,j+2}$ is colored either $c$ or $w$, Alice colors $v_{2,j+1}$ with the same color, making $v_{2,j+1}$ and $v_{3,j+1}$ safe.
Thus, also assume that $v_{3,j+2}$ is not colored.
Then Alice colors $v_{3,j+1}$ with $w$.
Since $v_{1,j+1}$ cannot receive the color $c$ during the game and the other neighbors of $v_{2,j+1}$ are colored $w$, $c'$ and $c''$, then it will always be possible to color $v_{2,j+1}$ in the game. In other words, $v_{1,j+1}$ is  safe and the induction hypothesis still holds.

\item {\bf Case $3\delta$:} Bob colors in the column adjacent to a border (column $j$) in configuration $\delta$  (Bob plays in column $j+1$ in case of a right border).
Alice will color a vertex of the column $j$ or $j+1$, obtaining a border in configuration $\alpha$, $\beta$, $\gamma$ or $\delta$, or merging two blocks, maintaining the induction hypothesis.
Note that all vertices of column $j$ are safe by induction.

\begin{figure}[ht]\centering
\begin{tikzpicture}[scale=0.195]
\tikzstyle{vertex}=[draw,circle,fill=black,inner sep=0pt, minimum size=1pt]

\begin{scope}[xshift=0cm, yshift=0cm, scale=2.5]
\foreach \i in {2,3,4}{
\path[draw, thin] (\i,0)--(\i,4) ;}
\foreach \j in {0,1,2,3,4}{
\path[draw, densely dotted] (1,\j)--(5,\j);} 
\foreach \j in {0,1,2,3,4}{ 
\path[draw, thin] (2,\j)--(4,\j);} 
\path[draw, very thick]  (1,0)--(3,0)--(3,4)--(1,4) ;
\draw (2.5,4.5) node {$j$};
\draw (1.5,2.5) node {$x$};
\draw (2.5,3.5) node {$x$};
\draw (2.5,2.5) node {$0$};
\draw (2.5,1.5) node {$y$};
\draw (3.5,0.5) node {$\blue{y}$};
\draw (3.5,1.5) node {$0$};
\draw (3.5,3.5) node {$0$};
\draw (3.5,2.5) node {$\red{y}$};
\draw (4.7,2.5) node {$\ne y$};
\draw (2.5,-0.75) node {Case $3\delta 1a$}; 
\draw (2.5,-1.75) node {$\delta \to \alpha/$merge}; 
\end{scope}

\begin{scope}[xshift=15cm, yshift=0cm, scale=2.5]
\foreach \i in {2,3,4}{
\path[draw, thin] (\i,0)--(\i,4) ;}
\foreach \j in {0,1,2,3,4}{
\path[draw, densely dotted] (1,\j)--(5,\j);} 
\foreach \j in {0,1,2,3,4}{ 
\path[draw, thin] (2,\j)--(4,\j);} 
\path[draw, very thick]  (1,0)--(3,0)--(3,4)--(1,4) ;
\path[draw, very thick]  (5,0)--(4,0)--(4,4)--(5,4) ;
\draw (2.5,4.5) node {$j$};
\draw (1.5,2.5) node {$x$};
\draw (2.5,3.5) node {$x$};
\draw (2.5,2.5) node {$0$};
\draw (2.5,1.5) node {$y$};
\draw (3.5,3.5) node {$\red{y}$};
\draw (3.5,2.5) node {$0$};
\draw (3.5,1.5) node {$0$};
\draw (3.5,0.5) node {$\blue{y}$};
\draw (4.5,2.5) node {$y$};
\draw (2.5,-0.75) node {Case $3\delta 1b$}; 
\draw (2.5,-1.75) node {$\delta \to$ merge}; 
\end{scope}

\begin{scope}[xshift=30cm, yshift=0cm, scale=2.5]
\foreach \i in {2,3}{
\path[draw, thin] (\i,0)--(\i,4) ;}
\foreach \j in {0,1,2,3,4}{
\path[draw, densely dotted] (1,\j)--(4.2,\j);} 
\foreach \j in {0,1,2,3,4}{ 
\path[draw, thin] (2,\j)--(3,\j);} 
\path[draw, very thick]  (1,0)--(3,0)--(3,4)--(1,4) ;
\draw (2.5,4.5) node {$j$};
\draw (1.5,2.5) node {$x$};
\draw (2.5,3.5) node {$x$};
\draw (2.5,2.5) node {$0$};
\draw (2.5,1.5) node {$y$};
\draw (3.5,3.5) node {$0$};
\draw (3.7,2.5) node {$\red{c/y}$};
\draw (3.5,1.5) node {$0$};
\draw (3.5,0.5) node {$\blue{c}$};
\draw (2.5,-0.75) node {Case $3\delta 2$}; 
\draw (2.5,-1.75) node {$\delta \to \alpha/\beta/$merge}; 
\end{scope}

\begin{scope}[xshift=43cm, yshift=0cm, scale=2.5]
\foreach \i in {2,3,4,5}{
\path[draw, thin] (\i,0)--(\i,4) ;}
\foreach \j in {0,1,2,3,4}{
\path[draw, densely dotted] (1,\j)--(6,\j);} 
\foreach \j in {0,1,2,3,4}{ 
\path[draw, thin] (2,\j)--(5,\j);} 
\path[draw, very thick]  (1,0)--(3,0)--(3,4)--(1,4) ;
\draw (2.5,4.5) node {$j$};
\draw (1.5,2.5) node {$x$};
\draw (2.5,3.5) node {$x$};
\draw (2.5,2.5) node {$0$};
\draw (2.5,1.5) node {$y$};
\draw (3.5,3.5) node {$0$};
\draw (3.5,2.5) node {$0$};
\draw (3.5,1.5) node {$\blue{c}$};
\draw (3.5,0.5) node {$0$};
\draw (4.5,3.5) node {$0$};
\draw (4.5,2.5) node {$\red{c}$};
\draw (4.5,1.5) node {$0$};
\draw (4.5,0.5) node {$0$};
\draw (5.7,2.5) node {$\ne c$};
\draw (3.5,-0.75) node {Case $3\delta 3$a}; 
\draw (3.5,-1.75) node {$\delta \to \beta/$merge}; 
\end{scope}

\begin{scope}[xshift=60cm, yshift=0cm, scale=2.5]
\foreach \i in {2,3,4,5}{
\path[draw, thin] (\i,0)--(\i,4) ;}
\foreach \j in {0,1,2,3,4}{
\path[draw, densely dotted] (1,\j)--(6,\j);} 
\foreach \j in {0,1,2,3,4}{ 
\path[draw, thin] (2,\j)--(5,\j);} 
\path[draw, very thick]  (1,0)--(3,0)--(3,4)--(1,4) ;
\path[draw, very thick]  (6,0)--(5,0)--(5,4)--(6,4) ;
\draw (2.5,4.5) node {$j$};
\draw (1.5,2.5) node {$x$};
\draw (2.5,3.5) node {$x$};
\draw (2.5,2.5) node {$0$};
\draw (2.5,1.5) node {$y$};
\draw (3.5,3.5) node {$0$};
\draw (3.5,2.5) node {$0$};
\draw (3.5,1.5) node {$\blue{c}$};
\draw (3.5,0.5) node {$0$};
\draw (4.5,3.5) node {$\red{c}$};
\draw (4.5,2.5) node {$0$};
\draw (4.5,1.5) node {$0$};
\draw (4.5,0.5) node {$0$};
\draw (5.5,2.5) node {$c$};
\draw (3.5,-0.75) node {Case $3\delta 3$b}; 
\draw (3.5,-1.75) node {$\delta \to$  merge}; 
\end{scope}

\begin{scope}[xshift=0cm, yshift=-20cm, scale=2.5]
\foreach \i in {2,3,4}{
\path[draw, thin] (\i,0)--(\i,4) ;}
\foreach \j in {0,1,2,3,4}{
\path[draw, densely dotted] (1,\j)--(5,\j);} 
\foreach \j in {0,1,2,3,4}{ 
\path[draw, thin] (2,\j)--(4,\j);} 
\path[draw, very thick]  (1,0)--(3,0)--(3,4)--(1,4) ;
\draw (2.5,4.5) node {$j$};
\draw (1.5,2.5) node {$x$};
\draw (2.5,3.5) node {$x$};
\draw (2.5,2.5) node {$0$};
\draw (2.5,1.5) node {$y$};
\draw (3.5,3.5) node {$0$};
\draw (3.5,2.5) node {$\blue{c}$};
\draw (3.5,1.5) node {$0$};
\draw (3.5,0.5) node {$\red{y}$};
\draw (4.5,3.5) node {$0$};
\draw (4.5,2.5) node {$0$};
\draw (4.5,1.5) node {$0$};
\draw (4.5,0.5) node {$0$};
\draw (3,-0.75) node {Case $3\delta 4$};
\draw (3,-1.75) node {$\delta \to \delta/\alpha$}; 
\end{scope}

\begin{scope}[xshift=15cm, yshift=-20cm, scale=2.5]
\foreach \i in {2,3,4}{
\path[draw, thin] (\i,0)--(\i,4) ;}
\foreach \j in {0,1,2,3,4}{
\path[draw, densely dotted] (1,\j)--(5,\j);} 
\foreach \j in {0,1,2,3,4}{ 
\path[draw, thin] (2,\j)--(4,\j);} 
\path[draw, very thick]  (1,0)--(3,0)--(3,4)--(1,4) ;
\draw (2.5,4.5) node {$j$};
\draw (1.5,2.5) node {$x$};
\draw (2.5,3.5) node {$x$};
\draw (2.5,2.5) node {$0$};
\draw (2.5,1.5) node {$y$};
\draw (3.5,3.5) node {$\blue{c}$};
\draw (3.5,2.5) node {$\red{y}$};
\draw (3.5,1.5) node {$0$};
\draw (3.5,0.5) node {$0$};
\draw (4.7,2.5) node {$\ne y$};
\draw (3,-0.75) node {Case $3\delta 5a$}; 
\draw (3,-1.75) node {$\delta \to \beta/$merge}; 
\end{scope}

\begin{scope}[xshift=30cm, yshift=-20cm, scale=2.5]
\foreach \i in {2,3,4}{
\path[draw, thin] (\i,0)--(\i,4) ;}
\foreach \j in {0,1,2,3,4}{
\path[draw, densely dotted] (1,\j)--(5,\j);} 
\foreach \j in {0,1,2,3,4}{ 
\path[draw, thin] (2,\j)--(4,\j);} 
\path[draw, very thick]  (1,0)--(3,0)--(3,4)--(1,4) ;
\path[draw, very thick]  (5,0)--(4,0)--(4,4)--(5,4) ;
\draw (2.5,4.5) node {$j$};
\draw (1.5,2.5) node {$x$};
\draw (2.5,3.5) node {$x$};
\draw (2.5,2.5) node {$\red{c}$};
\draw (2.5,1.5) node {$y$};
\draw (3.5,3.5) node {$\blue{c}$};
\draw (3.5,2.5) node {$0$};
\draw (3.5,1.5) node {$0$};
\draw (3.5,0.5) node {$0$};
\draw (4.5,2.5) node {$y$};
\draw (3,-0.75) node {Case $3\delta 5b$}; 
\draw (3,-1.75) node {$\delta \to$  merge};  
\end{scope}

\begin{scope}[xshift=45cm, yshift=-20cm, scale=2.5]
\foreach \i in {2,3,4}{
\path[draw, thin] (\i,0)--(\i,4) ;}
\foreach \j in {0,1,2,3,4}{
\path[draw, densely dotted] (1,\j)--(5,\j);} 
\foreach \j in {0,1,2,3,4}{ 
\path[draw, thin] (2,\j)--(4,\j);} 
\path[draw, very thick]  (1,0)--(3,0)--(3,4)--(1,4) ;
\draw (2.5,4.5) node {$j$};
\draw (1.5,2.5) node {$x$};
\draw (2.5,3.5) node {$x$};
\draw (2.5,2.5) node {$0$};
\draw (2.5,1.5) node {$y$};
\draw (3.5,3.5) node {$\blue{y}$};
\draw (3.5,2.5) node {$0$};
\draw (3.5,1.5) node {$0$};
\draw (3.5,0.5) node {$\red{y}$};
\draw (4.7,0.5) node {$\ne y$};
\draw (3.0,-0.75) node {Case $3\delta 5c$}; 
\draw (3.0,-1.75) node {$\delta \to \gamma/$merge};
\end{scope}

\begin{scope}[xshift=61cm, yshift=-20cm, scale=2.5]
\foreach \i in {2,3,4}{
\path[draw, thin] (\i,0)--(\i,4) ;}
\foreach \j in {0,1,2,3,4}{
\path[draw, densely dotted] (1,\j)--(5,\j);} 
\foreach \j in {0,1,2,3,4}{ 
\path[draw, thin] (2,\j)--(4,\j);} 
\path[draw, very thick]  (1,0)--(3,0)--(3,4)--(1,4) ;
\path[draw, very thick]  (5,0)--(4,0)--(4,4)--(5,4) ;
\draw (2.5,4.5) node {$j$};
\draw (1.5,2.5) node {$x$};
\draw (2.5,3.5) node {$x$};
\draw (2.5,2.5) node {$0$};
\draw (2.5,1.5) node {$y$};
\draw (3.5,3.5) node {$\blue{y}$};
\draw (3.5,2.5) node {$0$};
\draw (3.5,1.5) node {$0$};
\draw (3.5,0.5) node {$0$};
\draw (4.5,2.5) node {$0$};
\draw (4.5,1.5) node {$0$};
\draw (4.5,0.5) node {$y$};
\draw (3.0,-.75) node {Case $3\delta 5d$:}; 
\draw (3.0,-1.5) node {not possible}; 

\end{scope}

\end{tikzpicture}

\caption{Case $3\delta$: Bob plays in the empty column adjacent to a border in configuration $\delta$, where $x,y$ are distinct colors. The move of Bob is depicted in blue and the answer of Alice is in red.}
\label{fig-s-rudini}
\end{figure}
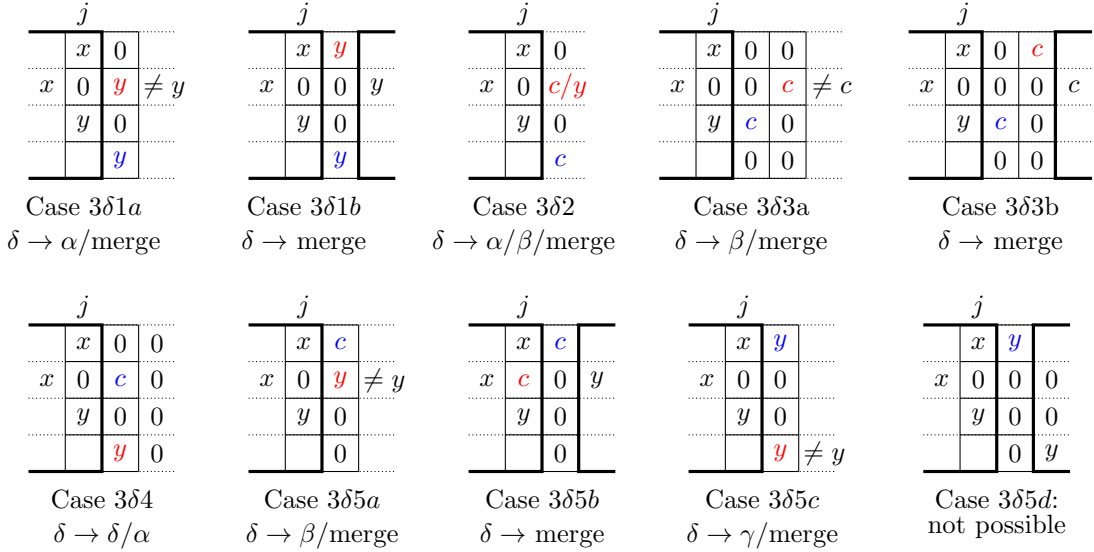

\begin{itemize}
\item {\bf Case $3\delta 1$:} 
Suppose that Bob colored $v_{1,j+1}$ with color $y$. If $v_{3,j+2}$ is not colored $y$ (Case $3\delta 1a$), she colors $v_{3,j+1}$ with $y$ obtaining a border in configuration $\alpha$ or merging two blocks. Otherwise (Case $3\delta 1b$), she colors $v_{4,j+1}$ with $y$, merging two blocks, making $v_{2,j+1}$ and $v_{3,j+1}$ safe.
\item {\bf Case $3\delta 2$:} 
Suppose that Bob colored $v_{1,j+1}$ with color $c\ne y$. Then Alice can color $v_{3,j+1}$ with $c$ or $y$, obtaining a border in configuration $\alpha$ or $\beta$, or merging two blocks.
\item {\bf Case $3\delta 3$:} 
Suppose that Bob colored $v_{2,j+1}$ with $c\ne y$. If column $j+2$ is not empty, Alice colors $v_{3,j+1}$ with any available color, merging the two blocks. So, assume that column $j+2$ is empty. If $v_{3,j+3}$ is not colored $c$ (Case $3\delta 3$a), then Alice colors $v_{3,j+2}$ with color $c$, obtaining a border in configuration $\beta$ or merging two blocks. If $v_{3,j+3}$ is colored $c$ (Case $3\delta 3$b), then she colors $v_{4,j+2}$ with color $c$, making $v_{3,j+2}$ safe, $v_{2,j+2}$ sound (with doctors $v_{3,j+2}$ and $v_{1,j+2}$) and $v_{3,j+1}$ sound (with doctors $v_{3,j}$ and $v_{4,j+1}$).
\item {\bf Case $3\delta 4$:} 
Suppose that Bob colored $v_{3,j+1}$ with $c$. If column $j+2$ is not empty, Alice colors $v_{2,j+1}$ with any available color, merging two blocks. So, assume that column $j+2$ is empty. Then Alice colors $v_{1,j+1}$ with $y$, obtaining a border in configuration $\alpha$ (if $c=y$) or in configuration $\delta$ (if $c\ne y$).
\item {\bf Case $3\delta 5$:} 
Suppose that Bob colored $v_{4,j+1}$ with $c\ne x$.
If $c\ne y$ and $v_{3,j+2}$ is not colored $y$ (Case $3\delta 5a$), Alice colors $v_{3,j+1}$ with color $y$ obtaining a border in configuration $\beta$ or merging two blocks.
If $c\ne y$ and $v_{3,j+2}$ is colored $y$ (Case $3\delta 5b$), Alice colors $v_{3,j}$ with $c$, merging two blocks, making $v_{3,j+1}$ safe and $v_{2,j+1}$ sound (with doctors $v_{3,j+1}$ and $v_{1,j+1}$).
So, suppose that $c=y$.
If $v_{1,j+2}$ is not colored $y$ (Case $3\delta 5c$), Alice colors $v_{1,j+1}$ with $y$, obtaining a border in configuration $\gamma$ or merging two blocks (with $v_{3,j+1}$ sound). So, assume that $c=y$ and $v_{1,j+2}$ is colored $y$. If $v_{2,j+2}$ is colored $c'\neq 0$, Alice colors $v_{3,j+1}$ with $c'$, merging two blocks, making $v_{2,j+1}$ and $v_{3,j+1}$ safe. So also assume that $v_{2,j+2}$ is not colored. If $v_{3,j+2}$ is colored $c'$, Alice colors $v_{2,j+1}$ with $c'$ if $c'\ne y$ or any other color otherwise, making $v_{2,j+1}$ and $v_{3,j+1}$ safe. So also assume that $v_{3,j+2}$ is not colored (Case $3\delta 5d$). Then column $j+2$ must be a border in configuration $\gamma$ and $v_{4,j+2}$ is colored $y$, a contradiction since Bob just colored $v_{4,j+1}$ with $y$.
\end{itemize}

\end{itemize}

From now on, we assume that Bob just played in an empty column adjacent to some border (column $j$) but not adjacent to a border in configuration $\pi$ or $\delta$. The next cases first assume that border at column $j$ is free.

\begin{itemize}
\item {\bf Case $3\alpha F$. Empty column adjacent to a free configuration $\alpha$.}  The next case occurs if Bob colors a vertex in an empty column (column $j+1$ in the illustrations) that is adjacent to a free border in configuration $\alpha$. The different subcases are described in Figure~\ref{fig:nextMoveAlice2}.

\begin{figure}[ht]\centering
\begin{tikzpicture}[scale=0.195]
\tikzstyle{vertex}=[draw,circle,fill=black,inner sep=0pt, minimum size=1pt]

\begin{scope}[xshift=0cm, yshift=0cm, scale=2.5]
\fill[blue!20] (2,0) rectangle (3,1);
\fill[blue!20] (2,3) rectangle (3,4);
\foreach \i in {1,2,3,4}{
\path[draw, thin] (\i,0)--(\i,4);}
\foreach \j in {0,1,2,3,4}{ 
\path[draw, densely dotted](0,\j)--(5,\j);} 
\foreach \j in {0,1,2,3,4}{ 
\path[draw, thin] (1,\j)--(2,\j)--(3,\j)--(4,\j);} 
\path[draw, very thick]  (0,0)--(2,0)--(2,4)--(0,4);
\draw (1.5,4.5) node {$j$};
\draw (1.5,0.5) node {$c$};
\draw (1.5,2.5) node {$c$};
\draw (3.5,0.5) node {$0$};
\draw (3.5,1.5) node {$0$};
\draw (3.5,2.5) node {$0$};
\draw (3.5,3.5) node {$0$};
\draw (2.5,1.5) node {$\red{c}$};
\draw (2.5,2.5) node {$0$};
\draw (2.5,-.75) node {Case $3\alpha F1$:};
\draw (2.5,-1.5) node {$\alpha\to\beta$};
\end{scope}

\begin{scope}[xshift=20cm, yshift=0cm, scale=2.5]
\foreach \i in {1,2,3,4}{
\path[draw, thin] (\i,0)--(\i,4);}
\foreach \j in {0,1,2,3,4}{
\path[draw, densely dotted](0,\j)--(5,\j);} 
\foreach \j in {0,1,2,3,4}{ 
\path[draw, thin] (1,\j)--(2,\j)--(3,\j)--(4,\j);} 
\path[draw, very thick] (0,0)--(2,0)--(2,4)--(0,4);
\draw (1.5,4.5) node {$j$};
\draw (1.5,0.5) node {$c$};
\draw (1.5,2.5) node {$c$};
\draw (3.5,0.5) node {$0$};
\draw (3.5,1.5) node {$0$};
\draw (3.5,2.5) node {$0$};
\draw (3.5,3.5) node {$0$};
\draw (2.5,0.5) node {$\red{c'}$};
\draw (2.5,1.5) node {$0$};
\draw (2.5,2.5) node {\blue{$c'$}};
\draw (2.5,3.5) node {$0$};
\draw (2.5,-.75) node {Case $3\alpha F2$:}; 
\draw (2.5,-1.5) node {$\alpha\to\alpha$}; 
\end{scope}

\begin{scope}[xshift=40cm, yshift=0cm, scale=2.5]
\foreach \i in {1,2,3,4}{
\path[draw, thin] (\i,0)--(\i,4);}
\foreach \j in {0,1,2,3,4}{
\path[draw, densely dotted](0,\j)--(5,\j);} 
\foreach \j in {0,1,2,3,4}{ 
\path[draw, thin] (1,\j)--(2,\j)--(3,\j)--(4,\j);} 
\path[draw, very thick] (0,0)--(2,0)--(2,4)--(0,4);
\draw (1.5,4.5) node {$j$};
\draw (1.5,0.5) node {$c$};
\draw (1.5,2.5) node {$c$};
\draw (2.5,0.5) node {$0$};
\draw (2.5,1.5) node {$\blue{c'}$};
\draw (2.5,2.5) node {$0$};
\draw (2.5,3.5) node {$\red{c}$};
\draw (3.5,0.5) node {$0$};
\draw (3.5,1.5) node {$0$};
\draw (3.5,2.5) node {$0$};
\draw (3.5,3.5) node {$0$};
\draw (2.5,-.75) node {Case $3\alpha F3$:}; 
\draw (2.5,-1.5) node {$\alpha\to\alpha/\delta$}; 
\end{scope}

\end{tikzpicture}
\caption{Case $3\alpha F$: Bob colors a vertex in an empty column adjacent to a free border in configuration $\alpha$. The blue cell is Bob's move and the red one is Alice's answer. The bold lines figure the surrounding of the block(s) before Alice's move.}
\label{fig:nextMoveAlice2}
\end{figure}
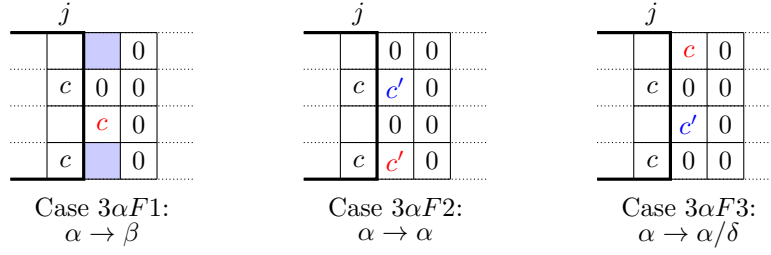

\begin{itemize}
\item {\bf Case $3\alpha F 1$:} If Bob plays on vertex $v_{1,j+1}$ or $v_{4,j+1}$, then Alice colors $v_{2,j+1}$ with color $c$, obtaining a new border in configuration $\beta$.
\item {\bf Case $3\alpha F 2$:} If Bob colors $v_{3,j+1}$ with $c'$ (note that $c'\neq c$),  Alice plays $c'$ on $v_{1,j+1}$, obtaining a new border in configuration $\alpha$. Notice that $v_{2,j+1}$ is not a doctor, satisfying Property $(4)$.
\item {\bf Case $3\alpha F 3$:} 
If Bob plays $c'$ on $v_{2,j+1}$, Alice colors $v_{4,j+1}$ with $c$ obtaining a border in configuration $\alpha$ (if $c'=c$) or in configuration $\delta$ (if $c'\ne c$).
\end{itemize}

\item {\bf Case $3\beta F$. Empty column adjacent to a free configuration $\beta$.}  The next case occurs if Bob colors a vertex in an empty column adjacent to a free border in configuration $\beta$. The different subcases are described in Figure~\ref{fig:nextMoveAlice3}.

\begin{figure}[ht]\centering
\begin{tikzpicture}[scale=0.195]
\tikzstyle{vertex}=[draw,circle,fill=black,inner sep=0pt, minimum size=1pt]

\begin{scope}[xshift=0cm, yshift=0cm, scale=2.5]
\foreach \i in {1,2,3}{
\path[draw, thin] (\i,0)--(\i,4);}
\foreach \j in {0,1,2,3,4}{ 
\path[draw, densely dotted] (0,\j)--(4,\j);} 
\foreach \j in {0,1,2,3,4}{ 
\path[draw, thin] (1,\j)--(3,\j);} 
\path[draw, very thick] (0,0)--(2,0)--(2,4)--(0,4);
\draw (1.5,4.5) node {$j$};
\draw (1.5,1.5) node {$0$};
\draw (0.5,1.5) node {$c$};
\draw (1.5,2.5) node {$c$};
\draw (3.5,0.5) node {$0$};
\draw (3.5,1.5) node {$0$};
\draw (3.5,2.5) node {$0$};
\draw (3.5,3.5) node {$0$};
\draw (2.5,0.5) node {\blue{$c'$}};
\draw (2.5,1.5) node {$0$};
\draw (2.5,2.5) node {\red{$c'$}};
\draw (2.5,3.5) node {$0$};
\draw (2.5,-.75) node {Case $3\beta F1$:}; 
\draw (2.5,-1.5) node {$\beta\to\alpha$}; 
\end{scope}

\begin{scope}[xshift=17cm, yshift=0cm, scale=2.5]
\foreach \i in {1,2,3}{
\path[draw, thin] (\i,0)--(\i,4);}
\foreach \j in {0,1,2,3,4}{ 
\path[draw, densely dotted] (0,\j)--(4,\j);} 
\foreach \j in {0,1,2,3,4}{ 
\path[draw, thin] (1,\j)--(3,\j);} 
\path[draw, very thick] (0,0)--(2,0)--(2,4)--(0,4);
\draw (1.5,4.5) node {$j$};
\draw (1.5,1.5) node {$0$};
\draw (0.5,1.5) node {$c$};
\draw (1.5,2.5) node {$c$};
\draw (3.5,0.5) node {$0$};
\draw (3.5,1.5) node {$0$};
\draw (3.5,2.5) node {$0$};
\draw (3.5,3.5) node {$0$};
\draw (2.5,0.5) node {\blue{$c$}};
\draw (2.5,1.5) node {$0$};
\draw (2.5,2.5) node {$0$};
\draw (2.5,3.5) node {\red{$c$}};
\draw (2.5,-.75) node {Case $3\beta F2$:}; 
\draw (2.5,-1.5) node {$\beta\to\gamma$}; 
\end{scope}

\begin{scope}[xshift=34cm, yshift=0cm, scale=2.5]
\foreach \i in {1,2,3}{
\path[draw, thin] (\i,0)--(\i,4);}
\foreach \j in {0,1,2,3,4}{ 
\path[draw, densely dotted] (0,\j)--(4,\j);} 
\foreach \j in {0,1,2,3,4}{ 
\path[draw, thin] (1,\j)--(3,\j);} 
\path[draw, very thick] (0,0)--(2,0)--(2,4)--(0,4);
\draw (1.5,4.5) node {$j$};
\draw (1.5,1.5) node {$0$};
\draw (0.5,1.5) node {$c$};
\draw (1.5,2.5) node {$c$};
\draw (3.5,0.5) node {$0$};
\draw (3.5,1.5) node {$0$};
\draw (3.5,2.5) node {$0$};
\draw (3.5,3.5) node {$0$};
\draw (2.5,0.5) node {$0$};
\draw (2.5,1.5) node {\blue{$c'$}};
\draw (2.5,2.5) node {$0$};
\draw (2.5,3.5) node {\red{$c$}};
\draw (2.5,-.75) node {Case $3\beta F3$:}; 
\draw (2.5,-1.5) node {$\beta\to\alpha$ or $\delta$}; 
\end{scope}

\begin{scope}[xshift=53cm, yshift=0cm, scale=2.5]
\foreach \i in {1,2,3}{
\path[draw, thin] (\i,0)--(\i,4);}
\foreach \j in {0,1,2,3,4}{ 
\path[draw, densely dotted] (0,\j)--(4,\j);} 
\foreach \j in {0,1,2,3,4}{ 
\path[draw, thin] (1,\j)--(3,\j);} 
\path[draw, very thick] (0,0)--(2,0)--(2,4)--(0,4);
\draw (1.5,4.5) node {$j$};
\draw (1.5,1.5) node {$0$};
\draw (0.5,1.5) node {$c$};
\draw (1.5,2.5) node {$c$};
\draw (3.5,0.5) node {$0$};
\draw (3.5,1.5) node {$0$};
\draw (3.5,2.5) node {$0$};
\draw (3.5,3.5) node {$0$};
\draw (2.5,0.5) node {$0$};
\draw (2.5,1.5) node {\red{$c$}};
\draw (2.5,2.5) node {$0$};
\draw (2.5,3.5) node {\blue{$c'$}};
\draw (2.5,-.75) node {Case $3\beta F4$:}; 
\draw (2.5,-1.5) node {$\beta\to\alpha$ or $\beta$}; 
\end{scope}

\begin{scope}[xshift=0cm, yshift=-20cm, scale=2.5]
\foreach \i in {2,3}{
\path[draw, thin] (\i,0)--(\i,4);}
\path[draw, densely dotted]  (1,0)--(1,4);
\foreach \j in {0,1,2,3,4}{ 
\path[draw, densely dotted] (0,\j)--(4,\j);} 
\foreach \j in {0,1,2,3,4}{ 
\path[draw, thin] (1,\j)--(3,\j);} 
\path[draw, very thick] (0,0)--(2,0)--(2,4)--(0,4);
\draw (1.5,4.5) node {$j$};
\draw (1.5,1.5) node {$0$};
\draw (1.25,0.5) node {$\neq c'$};
\draw (0.5,1.5) node {$c$};
\draw (1.5,2.5) node {$c$};
\draw (3.5,0.5) node {$0$};
\draw (3.5,1.5) node {$0$};
\draw (3.5,2.5) node {$0$};
\draw (3.5,3.5) node {$0$};
\draw (2.5,0.5) node {\red{$c'$}};
\draw (2.5,1.5) node {$0$};
\draw (2.5,2.5) node {\blue{$c'$}};
\draw (2.5,3.5) node {$0$};
\draw (2.5,-.75) node {Case $3\beta F5$:};
\draw (2.5,-1.5) node {$\beta\to\alpha$};
\end{scope}

\begin{scope}[xshift=18cm, yshift=-20cm, scale=2.5]
\foreach \i in {1,2,3,4}{
\path[draw, thin] (\i,0)--(\i,4);}
\foreach \j in {0,1,2,3,4}{ 
\path[draw, densely dotted] (0,\j)--(5,\j);} 
\foreach \j in {0,1,2,3,4}{ 
\path[draw, thin] (1,\j)--(4,\j);} 
\path[draw, very thick] (0,0)--(2,0)--(2,4)--(0,4);
\draw (1.5,4.5) node {$j$};
\draw (1.5,1.5) node {$0$};
\draw (1.5,0.5) node {$c'$};
\draw (0.5,1.5) node {$c$};
\draw (1.5,2.5) node {$c$};
\draw (3.5,0.5) node {$0$};
\draw (3.5,1.5) node {\red{$c'$}};
\draw (3.5,2.5) node {$0$};
\draw (3.5,3.5) node {$0$};
\draw (2.5,0.5) node {$0$};
\draw (2.5,1.5) node {$0$};
\draw (2.5,2.5) node {\blue{$c'$}};
\draw (4.75,1.5) node {$\neq c'$};
\draw (2.5,3.5) node {$0$};
\draw (2.5,-.75) node {Case $3\beta F6$:}; 
\draw (2.5,-1.5) node {$\beta\to\beta$ or merge};
\end{scope}

\begin{scope}[xshift=36cm, yshift=-20cm, scale=2.5]
\foreach \i in {1,2,3,4}{
\path[draw, thin] (\i,0)--(\i,4);}
\foreach \j in {0,1,2,3,4}{ 
\path[draw, densely dotted] (0,\j)--(5,\j);} 
\foreach \j in {0,1,2,3,4}{ 
\path[draw, thin] (1,\j)--(4,\j);} 
\path[draw, very thick] (0,0)--(2,0)--(2,4)--(0,4);
\path[draw, very thick]  (5,0) --(4,0)--(4,1) --(4,2)--(4,3)--(4,4)--(5,4) ;
\draw (1.5,4.5) node {$j$};
\draw (1.5,1.5) node {$0$};
\draw (1.5,0.5) node {$c'$};
\draw (0.5,1.5) node {$c$};
\draw (1.5,2.5) node {$c$};
\draw (3.5,0.5) node {\red{$c'$}};
\draw (3.5,1.5) node {$0$};
\draw (3.5,2.5) node {$0$};
\draw (3.5,3.5) node {$0$};
\draw (2.5,0.5) node {$0$};
\draw (2.5,1.5) node {$0$};
\draw (2.5,2.5) node {\blue{$c'$}};
\draw (4.5,1.5) node {$c'$};
\draw (2.5,3.5) node {$0$};
\draw (2.5,-.75) node {Case $3\beta F7$};
\draw (2.5,-1.5) node {$\beta\to$ merge};
\end{scope}

\end{tikzpicture}
\caption{Case $3\beta F$: Bob colors a vertex in an empty column adjacent to a free border in configuration $\beta$. The blue cell is Bob's move and the red one is Alice's answer. The bold lines figure the surrounding of the block(s) before Alice's move.}
\label{fig:nextMoveAlice3}
\end{figure}
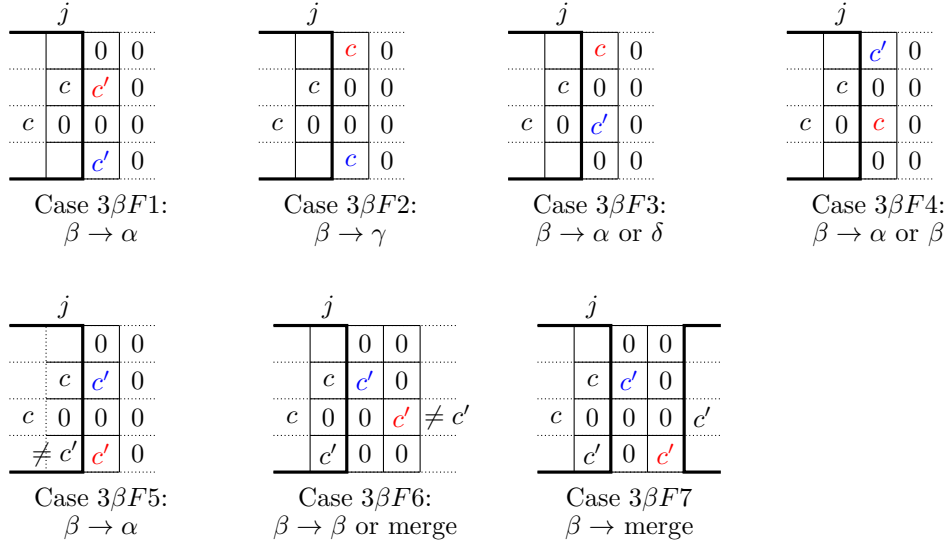

First assume that Bob colored $v_{1,j+1}$ with $c'$.
\begin{itemize}
\item {\bf Case $3\beta F 1$:} If $c'\ne c$, Alice plays $c'$ on $v_{3,j+1}$, obtaining a border in configuration $\alpha$ (notice that $v_{2,j+1}$ is not a doctor).
\item {\bf Case $3\beta F 2$:} If $c'=c$, Alice colors $v_{4,j+1}$ with $c$, obtaining a border in configuration $\gamma$.
\end{itemize}

Now assume that Bob colored $v_{2,j+1}$ with $c'$, which may be equal to $c$.

\begin{itemize}
\item {\bf Case $3\beta F 3$:} Alice colors $v_{4,j+1}$ with $c$, obtaining a border in configuration $\alpha$ (if $c'=c$) or in configuration $\delta$ (if $c'\ne c$).
\end{itemize}

Now assume that Bob colored $v_{4,j+1}$ with $c'$, which may be equal to $c$.

\begin{itemize}
\item {\bf Case $3\beta F 4$:} Alice plays $c$ on $v_{2,j+1}$, obtaining a border in configuration $\alpha$ (if $c'=c$) or $\beta$ (if $c'\ne c$).
\end{itemize}

Finally assume that Bob colored $v_{3,j+1}$ with $c'\ne c$.

\begin{itemize}
\item {\bf Case $3\beta F 5$:} If $v_{1,j}$ is not colored $c'$, Alice plays $c'$ on $v_{1,j+1}$, obtaining a border in configuration $\alpha$ (notice that $v_{2,j+1}$ is not a doctor).
\item {\bf Case $3\beta F 6$:} So, suppose that $v_{1,j}$ is colored $c'$. If $v_{2,j+3}$ is not colored $c'$, Alice colors $v_{2,j+2}$ with $c'$, obtaining a border in configuration $\beta$ or merging two blocks.
\item {\bf Case $3\beta F 7$:} So, also suppose that $v_{2,j+3}$ is colored $c'$. Then Alice colors $v_{1,j+2}$ with $c'$, merging two blocks and making $v_{2,j+1}$ sound (with doctors $v_{2,j}$ and $v_{1,j+1}$) and $v_{3,j+2}$ sound (with doctors $v_{2,j+2}$ and $v_{4,j+2}$).
\end{itemize}

\item {\bf Case $3\gamma F$. Empty column adjacent to a free configuration $\gamma$.}  The next case occurs if Bob colors a vertex in an empty column adjacent to a free border in configuration $\gamma$. The different subcases are described in Figure~\ref{fig-6gamma}. The move of Bob is in blue and the answer of Alice is in red.

\begin{figure}[ht!]\centering
\begin{tikzpicture}[scale=0.195]
\tikzstyle{vertex}=[draw,circle,fill=black,inner sep=0pt, minimum size=1pt]

\begin{scope}[xshift=0cm, yshift=0cm, scale=2.5]
\foreach \i in {3,4,5}{
\path[draw, thin] (\i,0)--(\i,4);}
\foreach \j in {0,1,2,3,4}{
\path[draw, densely dotted] (2,\j)--(6,\j);} 
\foreach \j in {0,1,2,3,4}{ 
\path[draw, thin] (3,\j)--(5,\j);}
\path[draw, very thick]  (2,0)--(4,0)--(4,4)--(2,4);
\draw (3.5,4.5) node {$j$};
\draw (2.5,1.65) node {$0$};
\draw (2.5,1.25) node {\tiny{safe}};
\draw (2.5,2.5) node {$c$};
\draw (3.5,0.5) node {$c$};
\draw (3.5,1.5) node {$0$};
\draw (3.5,2.5) node {$0$};
\draw (3.5,3.5) node {$c$};
\draw (4.5,3.5) node {$\blue{c'}$};
\draw (4.5,2.5) node {$0$};
\draw (4.5,1.5) node {$\red{c'}$};
\draw (4.5,0.5) node {$0$};
\draw (5.5,3.5) node {$0$};
\draw (5.5,2.5) node {$0$};
\draw (5.5,1.5) node {$0$};
\draw (5.5,0.5) node {$0$};
\draw (4.0,-.75) node {Case $3\gamma F1$:}; 
\draw (4.0,-1.5) node {$\gamma \to \alpha$}; 
\end{scope}

\begin{scope}[xshift=17cm, yshift=0cm, scale=2.5]
\foreach \i in {3,4,5}{
\path[draw, thin] (\i,0)--(\i,4);}
\foreach \j in {0,1,2,3,4}{
\path[draw, densely dotted] (2,\j)--(6,\j);} 
\foreach \j in {0,1,2,3,4}{ 
\path[draw, thin] (3,\j)--(5,\j);}
\path[draw, very thick]  (2,0)--(4,0)--(4,4)--(2,4);
\draw (3.5,4.5) node {$j$};
\draw (2.5,1.65) node {$0$};
\draw (2.5,1.25) node {\tiny{safe}};
\draw (2.5,2.5) node {$c$};
\draw (3.5,0.5) node {$c$};
\draw (3.5,1.5) node {$0$};
\draw (3.5,2.5) node {$0$};
\draw (3.5,3.5) node {$c$};
\draw (4.5,3.5) node {$0$};
\draw (4.5,2.5) node {$\red{c'}$};
\draw (4.5,1.5) node {$0$};
\draw (4.5,0.5) node {$\blue{c'}$};
\draw (5.5,3.5) node {$0$};
\draw (5.5,2.5) node {$0$};
\draw (5.5,1.5) node {$0$};
\draw (5.5,0.5) node {$0$};
\draw (4.0,-.75) node {Case $3\gamma F2$:}; 
\draw (4.0,-1.5) node {$\gamma \to \alpha$}; 
\end{scope}

\begin{scope}[xshift=34cm, yshift=0cm, scale=2.5]
\foreach \i in {3,4,5}{
\path[draw, thin] (\i,0)--(\i,4);}
\foreach \j in {0,1,2,3,4}{
\path[draw, densely dotted] (2,\j)--(6,\j);} 
\foreach \j in {0,1,2,3,4}{ 
\path[draw, thin] (3,\j)--(5,\j);}
\path[draw, very thick]  (2,0)--(4,0)--(4,4)--(2,4);
\draw (3.5,4.5) node {$j$};
\draw (2.5,1.65) node {$0$};
\draw (2.5,1.25) node {\tiny{safe}};
\draw (2.5,2.5) node {$c$};
\draw (3.5,0.5) node {$c$};
\draw (3.5,2.5) node {$\red{c'}$};
\draw (3.5,1.5) node {$0$};
\draw (3.5,3.5) node {$c$};
\draw (4.5,3.5) node {$0$};
\draw (4.5,1.5) node {$\blue{c'}$};
\draw (4.5,2.5) node {$0$};
\draw (4.5,0.5) node {$0$};
\draw (5.5,3.5) node {$0$};
\draw (5.5,2.5) node {$0$};
\draw (5.5,1.5) node {$0$};
\draw (5.5,0.5) node {$0$};
\draw (4.0,-.75) node {Case $3\gamma F3$:}; 
\draw (4.0,-1.5) node {$\gamma \to \beta$}; 
\end{scope}

\begin{scope}[xshift=51cm, yshift=0cm, scale=2.5]
\foreach \i in {3,4,5,6}{
\path[draw, thin] (\i,0)--(\i,4);}
\foreach \j in {0,1,2,3,4}{
\path[draw, densely dotted] (2,\j)--(7,\j);} 
\foreach \j in {0,1,2,3,4}{ 
\path[draw, thin] (3,\j)--(6,\j);}
\path[draw, very thick]  (2,0)--(4,0)--(4,4)--(2,4);
\draw (3.5,4.5) node {$j$};
\draw (2.5,1.65) node {$0$};
\draw (2.5,1.25) node {\tiny{safe}};
\draw (2.5,2.5) node {$c$};
\draw (3.5,0.5) node {$c$};
\draw (3.5,2.5) node {$0$};
\draw (3.5,1.5) node {$0$};
\draw (3.5,3.5) node {$c$};
\draw (4.5,3.5) node {$0$};
\draw (4.5,1.5) node {$\blue{c}$};
\draw (4.5,2.5) node {$0$};
\draw (4.5,0.5) node {$0$};
\draw (5.5,3.5) node {$0$};
\draw (5.5,2.5) node {$\red{c}$};
\draw (6.75,2.5) node {$\neq c$};
\draw (5.5,1.5) node {$0$};
\draw (5.5,0.5) node {$0$};
\draw (4.5,-.75) node {Case $3\gamma F4$:};
\draw (4.5,-1.5) node {$\gamma\to\beta$ or merge}; 
\end{scope}

\begin{scope}[xshift=0cm, yshift=-20cm, scale=2.5]
\foreach \i in {3,4,5,6}{
\path[draw, thin] (\i,0)--(\i,4);}
\foreach \j in {0,1,2,3,4}{
\path[draw, densely dotted] (2,\j)--(7,\j);} 
\foreach \j in {0,1,2,3,4}{ 
\path[draw, thin] (3,\j)--(6,\j);}
\path[draw, very thick]  (2,0)--(4,0)--(4,4)--(2,4);
\path[draw, very thick]  (7,0)--(6,0)--(6,4)--(7,4);
\draw (3.5,4.5) node {$j$};
\draw (2.5,1.65) node {$0$};
\draw (2.5,1.25) node {\tiny{safe}};
\draw (2.5,2.5) node {$c$};
\draw (3.5,0.5) node {$c$};
\draw (3.5,2.5) node {$0$};
\draw (3.5,1.5) node {$0$};
\draw (3.5,3.5) node {$c$};
\draw (4.5,3.5) node {$0$};
\draw (4.5,1.5) node {$\blue{c}$};
\draw (4.5,2.5) node {$0$};
\draw (4.5,0.5) node {$0$};
\draw (5.5,3.5) node {$\red{c}$};
\draw (5.5,2.5) node {$0$};
\draw (6.5,2.5) node {$c$};
\draw (5.5,1.5) node {$0$};
\draw (5.5,0.5) node {$0$};
\draw (4.2,-.75) node {Case $3\gamma F5$:}; 
\draw (4.2,-1.5) node {$\gamma\to$ merge}; 
\end{scope}

\begin{scope}[xshift=17cm, yshift=-20cm, scale=2.5]
\foreach \i in {3,4,5}{
\path[draw, thin] (\i,0)--(\i,4);}
\foreach \j in {0,1,2,3,4}{
\path[draw, densely dotted] (2,\j)--(6,\j);} 
\foreach \j in {0,1,2,3,4}{ 
\path[draw, thin] (3,\j)--(5,\j);}
\path[draw, very thick]  (2,0)--(4,0)--(4,4)--(2,4);
\draw (3.5,4.5) node {$j$};
\draw (2.5,1.65) node {$0$};
\draw (2.5,1.25) node {\tiny{safe}};
\draw (2.5,2.5) node {$c$};
\draw (3.5,0.5) node {$c$};
\draw (3.5,1.5) node {$\red{c'}$};
\draw (3.5,2.5) node {$0$};
\draw (3.5,3.5) node {$c$};
\draw (4.5,3.5) node {$0$};
\draw (4.5,2.5) node {$\blue{c'}$};
\draw (4.5,1.5) node {$0$};
\draw (4.5,0.5) node {$0$};
\draw (5.5,3.5) node {$0$};
\draw (5.5,2.5) node {$0$};
\draw (5.5,1.5) node {$0$};
\draw (5.5,0.5) node {$0$};
\draw (4.0,-.75) node {Case $3\gamma F6$:}; 
\draw (4.0,-1.5) node {$\gamma\to\beta$}; 
\end{scope}

\begin{scope}[xshift=34cm, yshift=-20cm, scale=2.5]
\foreach \i in {3,4,5,6}{
\path[draw, thin] (\i,0)--(\i,4);}
\foreach \j in {0,1,2,3,4}{
\path[draw, densely dotted] (2,\j)--(7,\j);} 
\foreach \j in {0,1,2,3,4}{ 
\path[draw, thin] (3,\j)--(6,\j);}
\path[draw, very thick]  (2,0)--(4,0)--(4,4)--(2,4);
\draw (3.5,4.5) node {$j$};
\draw (2.5,1.65) node {$0$};
\draw (2.5,1.25) node {\tiny{safe}};
\draw (2.5,2.5) node {$c$};
\draw (3.5,0.5) node {$c$};
\draw (3.5,1.5) node {$0$};
\draw (3.5,2.5) node {$0$};
\draw (3.5,3.5) node {$c$};
\draw (4.5,3.5) node {$0$};
\draw (4.5,2.5) node {$\blue{c}$};
\draw (4.5,1.5) node {$0$};
\draw (4.5,0.5) node {$0$};
\draw (5.5,3.5) node {$0$};
\draw (5.5,2.5) node {$0$};
\draw (5.5,1.5) node {$\red{c}$};
\draw (5.5,0.5) node {$0$};
\draw (6.8,1.5) node {$\ne c$};
\draw (4.5,-.75) node {Case $3\gamma F7$:}; 
\draw (4.5,-1.5) node {$\gamma\to\beta$ or merge}; 
\end{scope}

\begin{scope}[xshift=51cm, yshift=-20cm, scale=2.5]
\fill[red!20] (5,2) rectangle (6,3);
\foreach \i in {3,4,5,6}{
\path[draw, thin] (\i,0)--(\i,4);}
\foreach \j in {0,1,2,3,4}{
\path[draw, densely dotted] (2,\j)--(7,\j);} 
\foreach \j in {0,1,2,3,4}{ 
\path[draw, thin] (3,\j)--(6,\j);}
\path[draw, very thick]  (2,0)--(4,0)--(4,4)--(2,4);
\path[draw, very thick]  (7,0)--(6,0)--(6,4)--(7,4);
\draw (3.5,4.5) node {$j$};
\draw (2.5,1.65) node {$0$};
\draw (2.5,1.25) node {\tiny{safe}};
\draw (2.5,2.5) node {$c$};
\draw (3.5,0.5) node {$c$};
\draw (3.5,1.5) node {$0$};
\draw (3.5,2.5) node {$0$};
\draw (3.5,3.5) node {$c$};
\draw (4.5,3.5) node {$0$};
\draw (4.5,2.5) node {$\blue{c}$};
\draw (4.5,1.5) node {$0$};
\draw (4.5,0.5) node {$0$};
\draw (5.5,3.5) node {$0$};
\draw (5.5,2.5) node {$0$};
\draw (5.5,1.5) node {$0$};
\draw (5.5,0.5) node {$\red{c}$};
\draw (6.5,1.5) node {$c$};
\draw (4.5,-.75) node {Case $3\gamma F 8$:};
\draw (4.5,-1.5) node {$\gamma\to\Delta'$};
\end{scope}

\end{tikzpicture}
\caption{Case $3\gamma F$: Bob colors a vertex in an empty column adjacent to a free border in configuration $\gamma$ at column $j$.}
\label{fig-6gamma}
\end{figure}
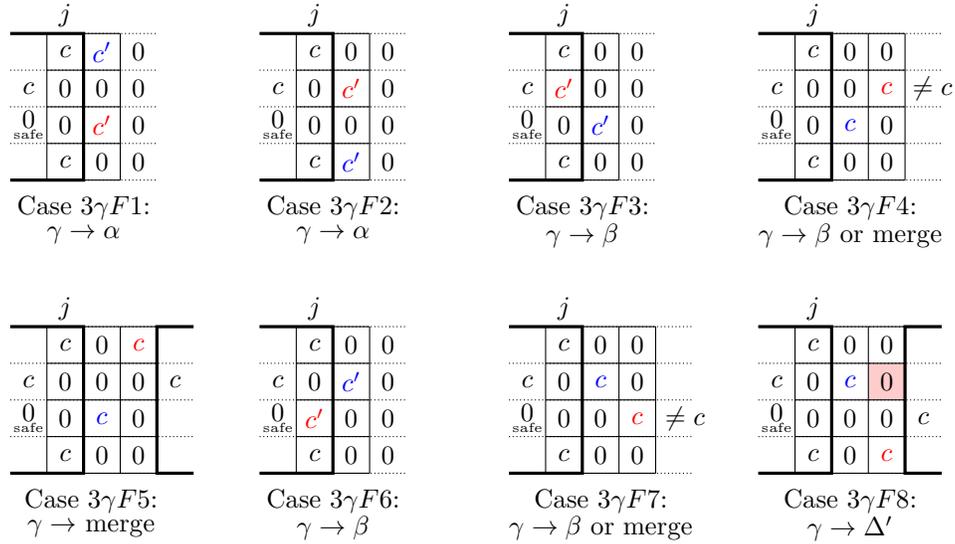

\begin{itemize}
\item {\bf Case $3\gamma F 1$:} If Bob colors $v_{4,j+1}$ with $c'$, then Alice  colors $v_{2,j+1}$ with $c'$ obtaining a border in configuration $\alpha$.
\item {\bf Case $3\gamma F 2$:} If Bob colors $v_{1,j+1}$ with $c'$, then Alice  colors $v_{3,j+1}$ with $c'$ obtaining a border in configuration $\alpha$.
\item {\bf Case $3\gamma F 3$:} If Bob colors $v_{2,j+1}$ with $c' \neq c$, then Alice  colors $v_{3,j}$ with $c'$ obtaining a border in configuration $\beta$.
\item {\bf Case $3\gamma F 4$:} If Bob colors $v_{2,j+1}$ with $c$ and $c(v_{3,j+3})\neq c$, then Alice  colors $v_{3,j}$ with $c$ obtaining a border in configuration $\beta$ or merging two blocks.
\item {\bf Case $3\gamma F 5$:} If Bob colors $v_{2,j+1}$ with $c$ and $c(v_{3,j+3})= c$ and $c(v_{1,j+3})\neq c$ then Alice  colors $v_{4,j+2}$ with $c$ making $v_{3,j+2}$ safe, $v_{2,j+2}$ sound (with doctors $v_{3,j+2}$ and $v_{1,j+2}$) and $v_{3,j+1}$ sound (with doctors $v_{3,j}$ and $v_{4,j+1}$).
\item {\bf Case $3\gamma F 6$:} If Bob colors $v_{3,j+1}$ with $c'\ne c$, then Alice  colors $v_{2,j}$ with $c'$ obtaining a border in configuration $\beta$.
\end{itemize}

So, assume that Bob colors $v_{3,j+1}$ with $c$.

\begin{itemize}
\item {\bf Case $3\gamma F 7$:} If $v_{2,j+3}$ is not colored $c$, then Alice colors $v_{2,j+2}$ with $c$ obtaining a border in configuration $\beta$.
\item {\bf Case $3\gamma F8$:} So, also assume that $v_{2,j+3}$ is colored with $c$. Then we are in configuration $\Delta'$, with all vertices safe or sound, except $v_{3,j+2}$, which is sick.
\end{itemize}

\end{itemize}

In the next, we assume that Bob just played in an empty column separating two borders (not in configuration $\delta$ nor $\pi$). First, let us assume that at least one of these borders is in configuration $\gamma$. 

\begin{itemize}
\item {\bf Case $3\gamma$.} Bob colors a vertex in an empty column that is separating two blocks with one of them in configuration $\gamma$ and the other in configuration $\alpha$, $\beta$ or $\gamma$. The cases are depicted in Figure~\ref{fig-case7c} and all of them merge the two blocks.

Let $j$ be the column of the border of the block in configuration $\gamma$. Up to symmetry, we consider the column $j$ is the right border of its block, that is, $j+1$ is the empty column before Bob's move.
If Bob colors $v_{2,j+1}$ (resp. $v_{3,j+1}$), Alice can color $v_{3,j+1}$ (resp. $v_{2,j+1}$) and we are done. So assume that Bob colored either $v_{1,j+1}$ or $v_{4,j+1}$.

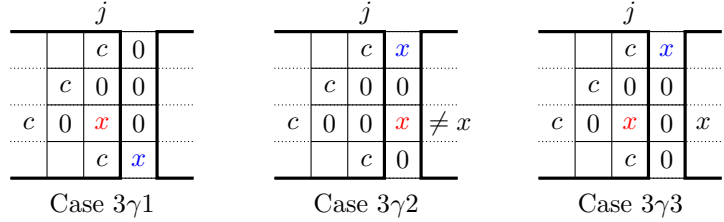
\begin{figure}[ht!]\centering
\begin{tikzpicture}[scale=0.195]
\tikzstyle{vertex}=[draw,circle,fill=black,inner sep=0pt, minimum size=1pt]

\begin{scope}[xshift=0cm, yshift=0cm, scale=2.5]
\foreach \i in {2,3,4,5}{
\path[draw, thin] (\i,0)--(\i,4);}
\foreach \j in {0,1,2,3,4}{
\path[draw, densely dotted] (1,\j)--(6,\j);} 
\foreach \j in {0,1,2,3,4}{ 
\path[draw, thin] (2,\j)--(5,\j);}
\path[draw, very thick]  (1,0)--(4,0)--(4,4)--(1,4);
\path[draw, very thick]  (6,0)--(5,0)--(5,4)--(6,4);
\draw (3.5,4.5) node {$j$};
\draw (1.5,1.5) node {$c$};
\draw (2.5,1.5) node {$0$};
\draw (2.5,2.5) node {$c$};
\draw (3.5,0.5) node {$c$};
\draw (3.5,1.5) node {$\red{x}$};
\draw (3.5,2.5) node {$0$};
\draw (3.5,3.5) node {$c$};
\draw (4.5,0.5) node {$\blue{x}$};
\draw (4.5,1.5) node {$0$};
\draw (4.5,2.5) node {$0$};
\draw (4.5,3.5) node {$0$};
\draw (3.5,-.75) node {Case $3\gamma 1$};
\end{scope}

\begin{scope}[xshift=18cm, yshift=0cm, scale=2.5]
\foreach \i in {2,3,4,5}{
\path[draw, thin] (\i,0)--(\i,4);}
\foreach \j in {0,1,2,3,4}{
\path[draw, densely dotted] (1,\j)--(6,\j);} 
\foreach \j in {0,1,2,3,4}{ 
\path[draw, thin] (2,\j)--(5,\j);}
\path[draw, very thick]  (1,0)--(4,0)--(4,4)--(1,4);
\path[draw, very thick]  (6,0)--(5,0)--(5,4)--(6,4);
\draw (3.5,4.5) node {$j$};
\draw (1.5,1.5) node {$c$};
\draw (2.5,1.5) node {$0$};
\draw (2.5,2.5) node {$c$};
\draw (3.5,0.5) node {$c$};
\draw (3.5,1.5) node {$0$};
\draw (3.5,2.5) node {$0$};
\draw (3.5,3.5) node {$c$};
\draw (4.5,0.5) node {$0$};
\draw (4.5,1.5) node {$\red{x}$};
\draw (4.5,2.5) node {$0$};
\draw (4.5,3.5) node {$\blue{x}$};
\draw (5.8,1.5) node {$\ne x$};
\draw (3.5,-.75) node {Case $3\gamma 2$}; 
\end{scope}

\begin{scope}[xshift=36cm, yshift=0cm, scale=2.5]
\foreach \i in {2,3,4,5}{
\path[draw, thin] (\i,0)--(\i,4);}
\foreach \j in {0,1,2,3,4}{
\path[draw, densely dotted] (1,\j)--(6,\j);} 
\foreach \j in {0,1,2,3,4}{ 
\path[draw, thin] (2,\j)--(5,\j);}
\path[draw, very thick]  (1,0)--(4,0)--(4,4)--(1,4);
\path[draw, very thick]  (6,0)--(5,0)--(5,4)--(6,4);
\draw (3.5,4.5) node {$j$};
\draw (1.5,1.5) node {$c$};
\draw (2.5,1.5) node {$0$};
\draw (2.5,2.5) node {$c$};
\draw (3.5,0.5) node {$c$};
\draw (3.5,1.5) node {$\red{x}$};
\draw (3.5,2.5) node {$0$};
\draw (3.5,3.5) node {$c$};
\draw (4.5,0.5) node {$0$};
\draw (4.5,1.5) node {$0$};
\draw (4.5,2.5) node {$0$};
\draw (4.5,3.5) node {$\blue{x}$};
\draw (5.5,1.5) node {$x$};
\draw (3.5,-.75) node {Case $3\gamma 3$};
\end{scope}

\end{tikzpicture}
\caption{Case $3\gamma$: Bob colors a vertex in an empty column that is separating two blocks with one
of them in configuration $\gamma$ and the other in configuration $\alpha$, $\beta$ or $\gamma$. In all cases, the two blocks are merged after Alice's move.}
\label{fig-case7c}
\end{figure}

\begin{itemize}
\item {\bf Case $3\gamma 1$:} If Bob colors $v_{1,j+1}$ with $x\ne c$, Alice colors $v_{2,j}$ with $x$, making $v_{2,j+1}$ safe and $v_{3,j+1}$ sound (with doctors $v_{3,j}$ and $v_{4,j+1}$).
\end{itemize}

So, assume Bob colors $v_{4,j+1}$ with $x\ne c$.

\begin{itemize}
\item {\bf Case $3\gamma 2$:} If $v_{2,j+2}$ is not colored $x$, Alice colors $v_{2,j+1}$ with $x$, making $v_{2,j+1}$ and $v_{3,j+1}$ safe.
\end{itemize}

So, also assume that $v_{2,j+2}$ is colored $x$.

\begin{itemize}
\item {\bf Case $3\gamma 3$:} Since $x\ne c$, Alice can color $v_{2,j}$ with $x$, making $v_{2,j+1}$ safe and $v_{3,j+1}$ sound (with doctors $v_{2,j+1}$ and $v_{3,j}$).
\end{itemize}

Next, we assume that Bob just played in an empty column separating two borders (not in configuration $\delta$ nor $\pi$ nor $\gamma$). We further assume that one of these borders is in configuration $\beta$.

\item {\bf Case $3\beta$.} Bob colors a vertex in an empty column that is separating two blocks with one of them in configuration $\beta$ and the other in configuration $\alpha$ or $\beta$. The cases are depicted in Figure~\ref{fig-case7b} and all of them merge the two blocks.

Let $j$ be the column of the border of the block in configuration $\beta$. Up to symmetry, we consider the column $j$ is the right border of its block, that is, $j+1$ is the empty column before Bob's move. If Bob colors $v_{2,j+1}$ (resp. $v_{3,j+1}$), Alice can color $v_{3,j+1}$ (resp. $v_{2,j+1}$) and we are done. So assume that Bob colored either $v_{1,j+1}$ or $v_{4,j+1}$.

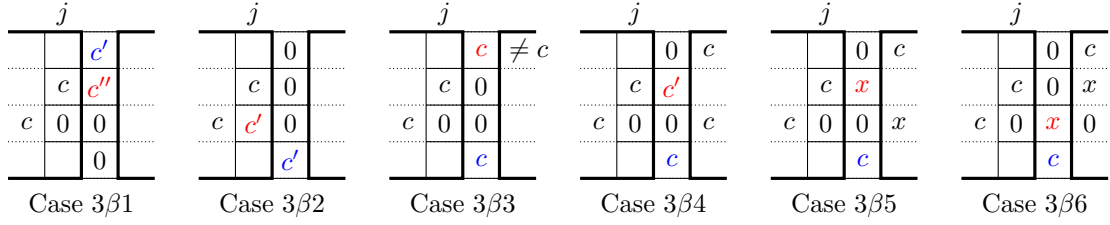
\begin{figure}[ht!]\centering
\begin{tikzpicture}[scale=0.195]
\tikzstyle{vertex}=[draw,circle,fill=black,inner sep=0pt, minimum size=1pt]

\begin{scope}[xshift=0cm, yshift=0cm, scale=2.5]
\foreach \i in {1,2,3}{
\path[draw, thin] (\i,0)--(\i,4);}
\foreach \j in {0,1,2,3,4}{
\path[draw, densely dotted] (0,\j)--(4,\j);} 
\foreach \j in {0,1,2,3,4}{ 
\path[draw, thin] (1,\j)--(3,\j);} 
\path[draw, very thick] (0,0)--(2,0)--(2,4)--(0,4);
\path[draw, very thick] (4,0)--(3,0)--(3,4)--(4,4);
\draw (1.5,1.5) node {$0$};
\draw (1.5,4.5) node {$j$};
\draw (0.5,1.5) node {$c$};
\draw (1.5,2.5) node {$c$};
\draw (2.5,3.5) node {$\blue{c'}$};
\draw (2.5,2.5) node {$\red{c''}$};
\draw (2.5,1.5) node {$0$};
\draw (2.5,0.5) node {$0$};
\draw (2,-0.75) node {Case $3\beta 1$};
\end{scope}

\begin{scope}[xshift=13cm, yshift=0cm, scale=2.5]
\foreach \i in {1,2,3}{
\path[draw, thin] (\i,0)--(\i,4);}
\foreach \j in {0,1,2,3,4}{
\path[draw, densely dotted] (0,\j)--(4,\j);} 
\foreach \j in {0,1,2,3,4}{ 
\path[draw, thin] (1,\j)--(3,\j);} 
\path[draw, very thick] (0,0)--(2,0)--(2,4)--(0,4);
\path[draw, very thick] (4,0)--(3,0)--(3,4)--(4,4);
\draw (1.5,4.5) node {$j$};
\draw (0.5,1.5) node {$c$};
\draw (1.5,2.5) node {$c$};
\draw (1.5,1.5) node {\red{$c'$}};
\draw (2.5,3.5) node {$0$};
\draw (2.5,2.5) node {$0$};
\draw (2.5,1.5) node {$0$};
\draw (2.5,0.5) node {$\blue{c'}$};
\draw (2,-0.75) node {Case $3\beta 2$};
\end{scope}

\begin{scope}[xshift=26cm, yshift=0cm, scale=2.5]
\foreach \i in {1,2,3}{
\path[draw, thin] (\i,0)--(\i,4);}
\foreach \j in {0,1,2,3,4}{
\path[draw, densely dotted] (0,\j)--(4,\j);} 
\foreach \j in {0,1,2,3,4}{ 
\path[draw, thin] (1,\j)--(3,\j);} 
\path[draw, very thick] (0,0)--(2,0)--(2,4)--(0,4);
\path[draw, very thick] (4,0)--(3,0)--(3,4)--(4,4);
\draw (1.5,4.5) node {$j$};
\draw (0.5,1.5) node {$c$};
\draw (1.5,2.5) node {$c$};
\draw (1.5,1.5) node {$0$};
\draw (2.5,3.5) node {$\red{c}$};
\draw (2.5,2.5) node {$0$};
\draw (2.5,1.5) node {$0$};
\draw (2.5,0.5) node {$\blue{c}$};
\draw (3.8,3.5) node {$\ne c$};
\draw (2,-0.75) node {Case $3\beta 3$};
\end{scope}

\begin{scope}[xshift=39cm, yshift=0cm, scale=2.5]
\foreach \i in {1,2,3}{
\path[draw, thin] (\i,0)--(\i,4);}
\foreach \j in {0,1,2,3,4}{
\path[draw, densely dotted] (0,\j)--(4,\j);} 
\foreach \j in {0,1,2,3,4}{ 
\path[draw, thin] (1,\j)--(3,\j);} 
\path[draw, very thick] (0,0)--(2,0)--(2,4)--(0,4);
\path[draw, very thick] (4,0)--(3,0)--(3,4)--(4,4);
\draw (1.5,4.5) node {$j$};
\draw (0.5,1.5) node {$c$};
\draw (1.5,2.5) node {$c$};
\draw (1.5,1.5) node {$0$};
\draw (2.5,3.5) node {$0$};
\draw (2.5,2.5) node {$\red{c'}$};
\draw (2.5,1.5) node {$0$};
\draw (2.5,0.5) node {$\blue{c}$};
\draw (3.5,3.5) node {$c$};
\draw (3.5,1.5) node {$c$};
\draw (2,-0.75) node {Case $3\beta 4$};
\end{scope}

\begin{scope}[xshift=52cm, yshift=0cm, scale=2.5]
\foreach \i in {1,2,3}{
\path[draw, thin] (\i,0)--(\i,4);}
\foreach \j in {0,1,2,3,4}{
\path[draw, densely dotted] (0,\j)--(4,\j);} 
\foreach \j in {0,1,2,3,4}{ 
\path[draw, thin] (1,\j)--(3,\j);} 
\path[draw, very thick] (0,0)--(2,0)--(2,4)--(0,4);
\path[draw, very thick] (4,0)--(3,0)--(3,4)--(4,4);
\draw (1.5,4.5) node {$j$};
\draw (0.5,1.5) node {$c$};
\draw (1.5,2.5) node {$c$};
\draw (1.5,1.5) node {$0$};
\draw (2.5,3.5) node {$0$};
\draw (2.5,2.5) node {$\red{x}$};
\draw (2.5,1.5) node {$0$};
\draw (2.5,0.5) node {$\blue{c}$};
\draw (3.5,3.5) node {$c$};
\draw (3.5,1.5) node {$x$};
\draw (2,-0.75) node {Case $3\beta 5$};
\end{scope}

\begin{scope}[xshift=65cm, yshift=0cm, scale=2.5]
\foreach \i in {1,2,3}{
\path[draw, thin] (\i,0)--(\i,4);}
\foreach \j in {0,1,2,3,4}{
\path[draw, densely dotted] (0,\j)--(4,\j);} 
\foreach \j in {0,1,2,3,4}{ 
\path[draw, thin] (1,\j)--(3,\j);} 
\path[draw, very thick] (0,0)--(2,0)--(2,4)--(0,4);
\path[draw, very thick] (4,0)--(3,0)--(3,4)--(4,4);
\draw (1.5,4.5) node {$j$};
\draw (0.5,1.5) node {$c$};
\draw (1.5,2.5) node {$c$};
\draw (1.5,1.5) node {$0$};
\draw (2.5,3.5) node {$0$};
\draw (2.5,2.5) node {$0$};
\draw (2.5,1.5) node {$\red{x}$};
\draw (2.5,0.5) node {$\blue{c}$};
\draw (3.5,3.5) node {$c$};
\draw (3.5,2.5) node {$x$};
\draw (3.5,1.5) node {$0$};
\draw (2,-0.75) node {Case $3\beta 6$};
\end{scope}

\end{tikzpicture}
\caption{Case $3\beta$: Bob colors a vertex in an empty column that is separating two blocks with one of them in configuration $\beta$ and the other in configuration $\alpha$ or $\beta$. In each case, Alice's move merges the two blocks.}
\label{fig-case7b}
\end{figure}

\begin{itemize}
\item {\bf Case $3\beta 1$:} If Bob colors $v_{4,j+1}$ with any color, Alice colors $v_{3,j+1}$ with any other color, making $v_{3,j+1}$ safe and $v_{2,j+1}$ sound (with doctors $v_{1,j+1}$ and $v_{2,j}$).
\end{itemize}

So, assume Bob colors $v_{1,j+1}$ with $c'$, which may be $c$.

\begin{itemize}
\item {\bf Case $3\beta 2$:} If $c'\neq c$, Alice colors $v_{2,j}$, making $v_{2,j+1}$ safe and $v_{3,j+1}$ sound (with doctors $v_{2,j+1}$ and $v_{4,j+1}$).
\end{itemize}

So, in fact, assume Bob colors $v_{1,j+1}$ with $c$.

\begin{itemize}
\item {\bf Case $3\beta 3$:} If $v_{4,j+2}$ is not colored $c$, Alice colors $v_{4,j+1}$ with $c$, making $v_{3,j+1}$ safe and $v_{2,j+1}$ sound (with doctors $v_{2,j}$ and $v_{3,j+1}$).
\end{itemize}

So, also assume $v_{4,j+2}$ is colored $c$.

\begin{itemize}
\item {\bf Case $3\beta 4$:} If $v_{2,j+2}$ is colored $c$ (then $v_{2,j+1}$ is already safe), Alice colors $v_{3,j+1}$ with any color, making  it safe.
\item {\bf Case $3\beta 5$:} If $v_{2,j+2}$ is colored $x\ne c$, Alice colors $v_{3,j+1}$ with $x$, making $v_{2,j+1}$ and $v_{3,j+1}$ safe.
\end{itemize}

So, also assume $v_{2,j+2}$ is not colored.

\begin{itemize}
\item {\bf Case $3\beta 6$:} If $v_{3,j+2}$ is colored $x$, Alice colors $v_{2,j+1}$ with $x$, making $v_{2,j+1}$ and $v_{3,j+1}$ safe.
\end{itemize}

So, also assume $v_{3,j+2}$ is not colored.
Then, by induction, column $j+2$ must be in configuration $\gamma$, a contradiction since this case considers an empty column between a border in configuration $\beta$ and a border in configuration $\alpha$ or $\beta$.

\end{itemize}

Finally, we assume that Bob just played in an empty column separating two borders both in configuration $\alpha$, except in configurations $\Lambda$ and $\Lambda'$, which are treated in cases $1\Lambda$ and $1\Lambda'$.

\begin{itemize}
\item {\bf Case $3\alpha$.} Bob colors a vertex in an empty column that is separating two blocks in configuration $\alpha$. The cases are depicted in Figure~\ref{fig-case7a} and all of them merge the two blocks.

Let $j$ be the column of the right border of the block in the left, that is, $j+1$ is the empty column before Bob's move.
If Bob colors $v_{2,j+1}$ (resp. $v_{3,j+1}$), Alice can color $v_{3,j+1}$ (resp. $v_{2,j+1}$) and we are done. So assume that Bob colored either $v_{1,j+1}$ or $v_{4,j+1}$.
Also notice that the border at column $j+2$ contains two non-colored vertices from Property $(5)$.

\begin{figure}[ht!]\centering
\begin{tikzpicture}[scale=0.195]
\tikzstyle{vertex}=[draw,circle,fill=black,inner sep=0pt, minimum size=1pt]

\begin{scope}[xshift=0cm, yshift=0cm, scale=2.5]
\fill[blue!20] (1,3) rectangle (2,4);
\fill[blue!20] (1,0) rectangle (2,1);
\foreach \j in {0,1,2,3,4}{ 
\path[draw, thin] (1,\j)--(2,\j) ;} 
\foreach \i in {1,2}{
\path[draw, thin] (\i,0)--(\i,4);}
\foreach \j in {0,1,2,3}{ 
\path[draw, densely dotted] (0,\j)--(3,\j);} 
\path[draw, very thick] (0,0)--(1,0)--(1,4)--(0,4) ;
\path[draw, very thick] (3,0)--(2,0)--(2,4)--(3,4);
\draw (0.5,4.5) node {$j$};
\draw (0.5,3.5) node {$c'$};
\draw (0.5,2.5) node {};
\draw (0.5,1.5) node {$c'$};
\draw (0.5,0.5) node {};
\draw (1.5,0.5) node {$0$};
\draw (1.5,1.5) node {$0$};
\draw (1.5,2.5) node {$\red{c'}$};
\draw (1.5,3.5) node {$0$};
\draw (2.5,0.5) node {$0$};
\draw (2.5,1.5) node {$c$};
\draw (2.5,2.5) node {$0$};
\draw (2.5,3.5) node {$c$};
\draw (1.5,-0.75) node {Case $3\alpha 1$}; 
\end{scope}

\begin{scope}[xshift=11cm, yshift=0cm, scale=2.5]
\fill[blue!20] (1,3) rectangle (2,4);
\fill[blue!20] (1,0) rectangle (2,1);
\foreach \i in {1,2}{
\path[draw, thin] (\i,0)--(\i,4);}
\foreach \j in {0,1,2,3,4}{
\path[draw, densely dotted] (0,\j)--(3,\j);} 
\foreach \j in {0,1,2,3,4}{ 
\path[draw, thin] (1,\j)--(2,\j) ;} 
\path[draw, very thick] (0,0)--(1,0)--(1,4)--(0,4) ;
\path[draw, very thick] (3,0)--(2,0)--(2,4)--(3,4);
\draw (0.5,4.5) node {$j$};
\draw (0.5,3.5) node {};
\draw (0.5,2.5) node {$c'$};
\draw (0.5,1.5) node {};
\draw (0.5,0.5) node {$c'$};
\draw (1.5,0.5) node {$0$};
\draw (1.5,1.5) node {$\red{c'}$};
\draw (1.5,2.5) node {$0$};
\draw (1.5,3.5) node {$0$};
\draw (2.5,0.5) node {$0$};
\draw (2.5,1.5) node {$c$};
\draw (2.5,2.5) node {$0$};
\draw (2.5,3.5) node {$c$};
\draw (1.5,-0.75) node {Case $3\alpha 2$}; 
\end{scope}

\begin{scope}[xshift=22cm, yshift=0cm, scale=2.5]
\foreach \i in {1,2}{
\path[draw, thin] (\i,0)--(\i,4);}
\foreach \j in {0,1,2,3,4}{ 
\path[draw, densely dotted] (0,\j)--(3,\j);} 
\foreach \j in {0,1,2,3,4}{ 
\path[draw, thin] (1,\j)--(2,\j) ;} 
\path[draw, very thick] (0,0)--(1,0)--(1,4)--(0,4) ;
\path[draw, very thick] (3,0)--(2,0)--(2,4)--(3,4);
\draw (0.5,4.5) node {$j$};
\draw (0.5,3.5) node {};
\draw (0.5,2.5) node {$c$};
\draw (0.5,1.5) node {};
\draw (0.5,0.5) node {$c$};
\draw (1.5,0.5) node {$\blue{x}$};
\draw (1.5,1.5) node {$0$};
\draw (1.5,2.5) node {$\red{x}$};
\draw (1.5,3.5) node {$0$};
\draw (2.5,0.5) node {$0$};
\draw (2.5,1.5) node {$c$};
\draw (2.5,2.5) node {$0$};
\draw (2.5,3.5) node {$c$};
\draw (1.5,-0.75) node {Case $3\alpha 3$}; 
\end{scope}

\begin{scope}[xshift=33cm, yshift=0cm, scale=2.5]
\foreach \i in {1,2}{
\path[draw, thin] (\i,0)--(\i,4);}
\foreach \j in {0,1,2,3,4}{ 
\path[draw, densely dotted] (0,\j)--(3,\j);} 
\foreach \j in {0,1,2,3,4}{ 
\path[draw, thin] (1,\j)--(2,\j) ;} 
\path[draw, very thick] (0,0)--(1,0)--(1,4)--(0,4) ;
\path[draw, very thick] (3,0)--(2,0)--(2,4)--(3,4);
\draw (0.5,4.5) node {$j$};
\draw (0.5,3.5) node {};
\draw (0.5,2.5) node {$c$};
\draw (0.2,1.5) node {$\ne x$};
\draw (0.5,0.5) node {$c$};
\draw (1.5,0.5) node {$0$};
\draw (1.5,1.5) node {$\red{x}$};
\draw (1.5,2.5) node {$0$};
\draw (1.5,3.5) node {$\blue{x}$};
\draw (2.5,0.5) node {$0$};
\draw (2.5,1.5) node {$c$};
\draw (2.5,2.5) node {$0$};
\draw (2.5,3.5) node {$c$};
\draw (1.5,-0.75) node {Case $3\alpha 4$}; 
\end{scope}

\begin{scope}[xshift=44cm, yshift=0cm, scale=2.5]
\foreach \i in {1,2,3}{
\path[draw, thin] (\i,0)--(\i,4);}
\foreach \j in {0,1,2,3,4}{
\path[draw, densely dotted] (0,\j)--(4,\j);} 
\foreach \j in {0,1,2,3,4}{ 
\path[draw, thin] (1,\j)--(3,\j) ;} 
\path[draw, thin] (3,4)--(4,4) ;
\path[draw, thin] (3,0)--(4,0) ;
\path[draw, very thick] (0,0)--(1,0)--(1,4)--(0,4) ;
\path[draw, very thick] (3,0)--(2,0)--(2,4)--(3,4) -- (3,0);
\draw (0.5,4.5) node {$j$};
\draw (0.5,3.5) node {};
\draw (0.5,2.5) node {$c$};
\draw (0.5,1.5) node {$x$};
\draw (0.5,0.5) node {$c$};
\draw (1.5,0.5) node {$0$};
\draw (1.5,1.5) node {$0$};
\draw (1.5,2.5) node {$0$};
\draw (1.5,3.5) node {$\blue{x}$};
\draw (2.5,0.5) node {$0$};
\draw (2.5,1.5) node {$c$};
\draw (2.5,2.5) node {$\red{x}$};
\draw (2.5,3.5) node {$c$};
\draw (3.5,3.5) node {$0$};
\draw (3.5,2.5) node {$0$};
\draw (3.5,1.5) node {$0$};
\draw (3.5,0.5) node {$0$};
\draw (1.5,-0.75) node {Case $3\alpha 5$}; 
\end{scope}

\begin{scope}[xshift=60cm, yshift=0cm, scale=2.5]
\foreach \i in {1,2,3}{
\path[draw, thin] (\i,0)--(\i,4);}
\foreach \j in {0,1,2,3,4}{ 
\path[draw, densely dotted] (0,\j)--(4,\j);} 
\foreach \j in {0,1,2,3,4}{ 
\path[draw, thin] (1,\j)--(3,\j) ;} 
\path[draw, very thick] (0,0)--(1,0)--(1,4)--(0,4) ;
\path[draw, very thick] (4,0)--(2,0)--(2,4)--(4,4);
\draw (0.5,4.5) node {$j$};
\draw (0.5,3.5) node {};
\draw (0.5,2.5) node {$c$};
\draw (0.5,1.5) node {$x$};
\draw (0.5,0.5) node {$c$};
\draw (1.5,0.5) node {$\red{x}$};
\draw (1.5,1.5) node {$0$};
\draw (1.5,2.5) node {$0$};
\draw (1.5,3.5) node {$\blue{x}$};
\draw (2.5,0.5) node {$0$};
\draw (2.5,1.5) node {$c$};
\draw (2.5,2.5) node {$0$};
\draw (2.5,3.5) node {$c$};
\draw (3.8,4.5) node {$\ne 0$};
\draw (1.5,-0.75) node {Case $3\alpha 6$};
\end{scope}

\end{tikzpicture}
\caption{Case $3\alpha$: Bob colors a vertex in an empty column that is separating two blocks in configuration $\alpha$.  In each case, Alice's move merges the two blocks.}
\label{fig-case7a}
\end{figure}
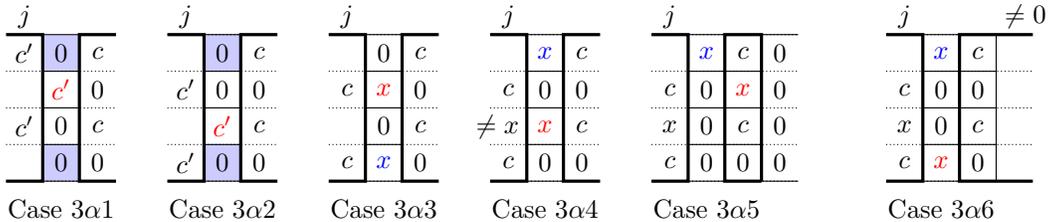

The subcases depend on the symmetry between the two borders in configuration $\alpha$ separated by the empty column.

\begin{itemize}
\item {\bf Case $3\alpha 1$:} If Bob colors $v_{1,j+1}$ or $v_{4,j+1}$, Alice colors $v_{3,j+1}$ with $c'$, making $v_{2,j+1}$ and $v_{3,j+1}$ safe.
\item {\bf Case $3\alpha 2$:} If $c'\ne c$, Alice colors $v_{2,j+1}$ with $c'$, making $v_{2,j+1}$ and $v_{3,j+1}$ safe, independently if Bob colored $v_{1,j+1}$ or $v_{4,j+1}$.
\end{itemize}

So, assume that $v_{1,j}$, $v_{3,j}$, $v_{2,j+2}$ and $v_{4,j+2}$ have the same color $c$.

\begin{itemize}
\item {\bf Case $3\alpha 3$:} If Bob colors $v_{1,j+1}$ with $x$, Alice colors $v_{3,j+1}$ with $x$, making $v_{2,j+1}$ and $v_{3,j+1}$ safe.
\end{itemize}

So, also assume that Bob colored $v_{4,j+1}$ with $x\ne c$.

\begin{itemize}
\item {\bf Case $3\alpha 4$:} If $v_{2,j}$ is not colored $x$, Alice colors $v_{2,j+1}$ with $x$, making $v_{2,j+1}$ and $v_{3,j+1}$ safe.
\end{itemize}

So, also assume that $v_{2,j}$ is colored $x$.

\begin{itemize}
\item {\bf Case $3\alpha 5$:} If column $j+3$ is empty, Alice colors $v_{3,j+2}$ with $x$, making $v_{3,j+1}$ safe and $v_{2,j+1}$ sound (with doctors $v_{1,j+1}$ and $v_{3,j+1}$).
\item {\bf Case $3\alpha 6$:} 
So, assume that column $j+3$ is not empty.
Then, by Property $(4)$, $v_{3,j+2}$ is not a doctor and Alice colors $v_{1,j+1}$ with $x$, making $v_{2,j+1}$ safe and $v_{3,j+1}$ sound (with doctors $v_{2,j+1}$ and $v_{3,j+2}$).
\end{itemize}

\end{itemize}

Notice that, in each of the subcases above of the Case $3$, Alice's answer is never in a column that was a border in configuration $\alpha$ that could be in configuration $\Delta$. So if the empty column played by Bob was the column $j+3$ of configuration $\Delta$ according to the orientation of Figure \ref{fig-delta}, then this configuration disappears and the sick vertex becomes sound, because its empty neighbor in the border in configuration $\alpha$ can now become its missing doctor. The same for configurations $\Lambda$ and $\Lambda'$ if the empty column played by Bob is the column $j-3$ of these configurations according to the orientation of Figure \ref{fig-delta} and if Alice's answer is not in the column $j-2$. If Alice's answer is in the column $j-2$, then the configurations $\Lambda$ and $\Lambda'$ become $\Lambda_2$ and $\Lambda_2'$, respectively.





In every subcase of Case $3$, after Alice's move, the properties of the induction hypothesis still hold.

\end{description}

This concludes the description of Alice's algorithm and its correctness and the proof of the theorem.
\end{proof}

\section{Further work}
The natural open problem is to decide whether $\chi_g(P_m\x P_n)\leq 4$ for every grid when $m,n\geq5$.
In our opinion, obtaining such results for $m,n\geq 5$ should require the use of an essentially different technique.
Indeed, there can be many more configurations already fitting in a $5\times 5$ square that may lead Alice to be inevitably defeated with only four colors; the question is to know if (and when) Bob can always force a fifth color by reaching such a configuration.

The question we addressed could also be extended to the cylinder $P_4\x C_n$ (and more generally to $P_m\x C_n$): does the absence of a first and last column have an impact on whether or not Alice has a winning strategy with four colors when she begins?

Finally, it would also be interesting to consider the coloring game in other graph classes.
In particular, the case of trees is not settled yet.

\normalem
\nocite{*}
\bibliographystyle{plain}

\begin{thebibliography}{10}

\bibitem{bodlaender91}
H.~L. Bodlaender.
\newblock On the complexity of some coloring games.
\newblock {\em Int. J. Found. Comput. Sci.}, 2(2):133--147, 1991.

\bibitem{costa20}
E.~R. Costa, V.~L. Pessoa, R.~Sampaio, and R.~Soares.
\newblock {PSPACE}-completeness of two graph coloring games.
\newblock {\em Theor. Comput. Sci.}, 824-825:36--45, 2020.

\bibitem{faigle1993game}
Ulrich Faigle, Walter Kern, H~Kierstead, and William~T Trotter.
\newblock On the game chromatic number of some classes of graphs.
\newblock {\em Ars combinatoria}, 35:143--150, 1993.

\bibitem{FRS23}
Ana Furtado, Miguel Del Rio~Palma, Simone Dantas, and Celina Figueiredo.
\newblock On the degree of trees with game chromatic number 4.
\newblock {\em RAIRO - Operations Research}, 57, 09 2023.

\bibitem{DBLP:journals/entcs/FurtadoDFG19}
Ana Lu{\'{\i}}sa~C. Furtado, Simone Dantas, Celina M.~H. de~Figueiredo, and
  Sylvain Gravier.
\newblock On caterpillars of game chromatic number 4.
\newblock In {\em Proceedings of the tenth Latin and American Algorithms,
  Graphs and Optimization Symposium, {LAGOS}}, volume 346 of {\em Electronic
  Notes in Theoretical Computer Science}, pages 461--472. Elsevier, 2019.

\bibitem{DBLP:journals/rairo/FurtadoPDF23}
Ana Lu{\'{\i}}sa~C. Furtado, Miguel A. D.~R. Palma, Simone Dantas, and Celina
  M.~H. de~Figueiredo.
\newblock On the degree of trees with game chromatic number 4.
\newblock {\em {RAIRO} Oper. Res.}, 57(5):2757--2767, 2023.

\bibitem{gardner81}
M.~Gardner.
\newblock Mathematical games.
\newblock {\em Scientific American}, 244(4):18--26, 1981.

\bibitem{kierstad94}
H.~A. Kierstead and W.~T. Trotter.
\newblock Planar graph coloring with an uncooperative partner.
\newblock {\em Journal of Graph Theory}, 18(6):569--584, 1994.

\bibitem{lima22}
C.~V. Lima, T.~Marcilon, N.~Martins, and R.~Sampaio.
\newblock {PSPACE}-hardness of variants of the graph coloring game.
\newblock {\em Theor. Comput. Sci.}, 909:87--96, 2022.

\bibitem{lima23}
C.V. Lima, T.~Marcilon, N.~Martins, and R.~Sampaio.
\newblock The connected greedy coloring game.
\newblock {\em Theor. Comput. Sci.}, 940:1--13, 2023.

\bibitem{nakprasit18}
Keaitsuda~Maneeruk Nakprasit and Kittikorn Nakprasit.
\newblock The game coloring number of planar graphs with a specific girth.
\newblock {\em Graphs and Combinatorics}, 34(2):349--354, 2018.

\bibitem{DBLP:journals/ipl/RaspaudW09}
Andr{\'{e}} Raspaud and Jiaojiao Wu.
\newblock Game chromatic number of toroidal grids.
\newblock {\em Inf. Process. Lett.}, 109(21-22):1183--1186, 2009.

\bibitem{sekiguchi14}
Yosuke Sekiguchi.
\newblock The game coloring number of planar graphs with a given girth.
\newblock {\em Discrete Mathematics}, 330:11--16, 2014.

\bibitem{game-cactus07}
Elzbieta Sidorowicz.
\newblock The game chromatic number and the game colouring number of cactuses.
\newblock {\em Information Processing Letters}, 102(4):147--151, 2007.

\bibitem{zhu00}
Xuding Zhu.
\newblock The game coloring number of pseudo partial k-trees.
\newblock {\em Discrete Mathematics}, 215(1):245--262, 2000.

\bibitem{zhu08}
Xuding Zhu.
\newblock Refined activation strategy for the marking game.
\newblock {\em Journal of Combinatorial Theory, Series B}, 98(1):1--18, 2008.

\end{thebibliography}

\newpage
\thispagestyle{empty}

\section*{Visual summary of the configurations to keep in mind}
\vspace{2cm}
\begin{center}
\begin{tikzpicture}[scale=0.2]
\begin{scope}
    \begin{scope}[xshift=30cm, yshift = 15cm]
        \draw (0,0) node {\large Border configurations};
    \end{scope}
    
    \begin{scope}[xshift=0cm, yshift=0cm, scale=2.5]
    \foreach \i in {2}{
    \path[draw, thin] (\i,0)--(\i,4);}
    \foreach \j in {0,1,2,3,4}{ 
    \path[draw, densely dotted] (1,\j)--(4,\j);}
    \foreach \j in {0,1,2,3,4}{
    \path[draw, thin] (2,\j)--(3,\j);} 
    \path[draw, very thick] (1,0)--(3,0)--(3,4)--(1,4);
    \draw (2.5,4.5) node {$j$};
    \draw (2.5,0.5) node {$c$};
    \draw (2.5,2.5) node {$c$};
    \draw (3.5,0.5) node {$0$};
    \draw (3.5,1.5) node {$0$};
    \draw (3.5,2.5) node {$0$};
    \draw (3.5,3.5) node {$0$};
    \draw (2.5,-0.75) node {Config. $\alpha$}; 
    \end{scope}

    \begin{scope}[xshift=12cm, yshift=0cm, scale=2.5]
    \foreach \i in {2,3}{
    \path[draw, thin] (\i,0)--(\i,4);}
    \foreach \j in {0,1,2,3,4}{ 
    \path[draw, densely dotted] (1,\j)--(4,\j);} 
    \foreach \j in {0,1,2,3,4}{ 
    \path[draw, thin] (2,\j)--(3,\j);} 
    \path[draw, very thick]  (1,0)--(3,0)--(3,4)--(1,4);
    \draw (2.5,4.5) node {$j$};
    \draw (1.5,1.5) node {$c$};
    \draw (2.5,0.5) node {};
    \draw (2.5,1.5) node {$0$};
    \draw (2.5,2.5) node {$c$};
    \draw (2.5,3.5) node {};
    \draw (3.5,0.5) node {$0$};
    \draw (3.5,1.5) node {$0$};
    \draw (3.5,2.5) node {$0$};
    \draw (3.5,3.5) node {$0$};
    \draw (2.5,-0.75) node {Config. $\beta$}; 
    \end{scope}

    \begin{scope}[xshift=22cm, yshift=0cm, scale=2.5]
    \foreach \i in {3,4}{
    \path[draw, thin] (\i,0)--(\i,4);}
    \foreach \j in {0,1,2,3,4}{ 
    \path[draw, densely dotted] (2,\j)--(3,\j);
    \path[draw, densely dotted] (4,\j)--(5,\j);} 
    \foreach \j in {0,1,2,3,4}{ 
    \path[draw, thin] (3,\j)--(4,\j);} 
    \path[draw, very thick]  (2,0)--(4,0)--(4,4)--(2,4) ;
    \draw (3.5,4.5) node {$j$};
    \draw (2.5,2.5) node {$c$};
    \draw (2.5,1.65) node {$0$};
    \draw (2.5,1.2) node {\tiny safe};
    \draw (3.5,0.5) node {$c$};
    \draw (3.5,1.5) node {$0$};
    \draw (3.5,2.5) node {$0$};
    \draw (3.5,3.5) node {$c$};
    \draw (4.5,0.5) node {$0$};
    \draw (4.5,1.5) node {$0$};
    \draw (4.5,2.5) node {$0$};
    \draw (4.5,3.5) node {$0$};
    \draw (3.5,-0.75) node {Config. $\gamma$}; 
    \end{scope}

    \begin{scope}[xshift=36cm, yshift=0cm, scale=2.5]
    \foreach \i in {2,3}{
    \path[draw, thin] (\i,0)--(\i,4) ;}
    \foreach \j in {0,1,2,3,4}{ 
    \path[draw, densely dotted] (1,\j)--(4,\j);} 
    \foreach \j in {0,1,2,3,4}{ 
    \path[draw, thin] (2,\j)--(3,\j);} 
    \path[draw, very thick]  (1,0)--(3,0)--(3,4)--(1,4) ;
    \draw (2.5,4.5) node {$j$};
    \draw (1.5,2.5) node {$c$};
    \draw (2.5,3.5) node {$c$};
    \draw (2.5,2.5) node {$0$};
    \draw (2.5,1.5) node {$c'$};
    \draw (3.5,0.5) node {$0$};
    \draw (3.5,1.5) node {$0$};
    \draw (3.5,2.5) node {$0$};
    \draw (3.5,3.5) node {$0$};
    \draw (2.5,-0.75) node {Config. $\delta$}; 
    \end{scope}
    
    \begin{scope}[xshift=50cm, yshift=0cm, scale=2.5]
    \foreach \i in {1,2}{
    \path[draw, thin] (\i,0)--(\i,4);}
    \foreach \j in {0,1,2,3}{ 
    \path[draw, densely dotted] (0,\j)--(3,\j);} 
    \foreach \j in {0,1,2,3,4}{ 
    \path[draw, thin] (1,\j)--(2,\j) ;} 
    \path[draw, very thick]  (0,0)--(1,0)--(1,4)--(0,4);
    \path[draw, very thick] (3,0)--(2,0)--(2,4)--(3,4);
    \draw (0.5,4.5) node {$j$};
    \draw (0.5,3.5) node {$c$};
    \draw (0.5,2.5) node {$c'$};
    \draw (0.5,1.5) node {$c''$};
    \draw (0.5,0.5) node {$c$};
    \draw (1.5,0.5) node {$0$};
    \draw (1.5,1.5) node {$0$};
    \draw (1.5,2.5) node {$0$};
    \draw (1.5,3.5) node {$0$};
    \draw (2.8,0.5) node {};
    \draw (2.5,1.5) node {$c'$};
    \draw (2.5,3.5) node {};
    \draw (1.5,-0.75) node {Config. $\pi$}; 
    \end{scope}
\end{scope}

\begin{scope}[xshift=5cm, yshift = -30cm]
    \begin{scope}[xshift = 25cm, yshift=15cm]
        \draw (0,0) node {\large Special configurations (sick vertices)};
    \end{scope}
    
    \begin{scope}[xshift=0cm, yshift=0cm, scale=2.5]
    \fill[red!20] (1,2) rectangle (2,3);
    \foreach \i in {1,2,3}{
    \path[draw, thin] (\i,0)--(\i,4);}
    \foreach \j in {0,1,2,3,4}{ 
    \path[draw, densely dotted] (0,\j)--(4,\j) ;} 
    \foreach \j in {0,1,2,3,4}{ 
    \path[draw, thin] (1,\j)--(3,\j) ;} 
    \path[draw, very thick]  (0,0)--(3,0)--(3,4)--(0,4);
    \draw (0.5,4.5) node {$j$};
    \draw (0.5,3.5) node {};
    \draw (0.5,2.5) node {$c$};
    \draw (0.5,1.5) node {};
    \draw (0.5,0.5) node {$c$};
    \draw (1.5,0.5) node {$0$};
    \draw (1.5,1.5) node {$c'$};
    \draw (1.5,2.5) node {$0$};
    \draw (1.5,3.5) node {$0$};
    \draw (2.5,3.5) node {$c$};
    \draw (2.5,2.5) node {$0$};
    \draw (2.5,1.5) node {$c$};
    \draw (2.5,0.5) node {$0$};
    \draw (3.5,3.5) node {$0$};
    \draw (3.5,2.5) node {$0$};
    \draw (3.5,1.5) node {$0$};
    \draw (3.5,0.5) node {$0$};
    \draw(1.4,-0.7) node {Configuration $\Delta$};
    \end{scope}
    
    \begin{scope}[xshift=16cm, yshift=0cm, scale=2.5]
    \fill[red!20] (4,2) rectangle (5,3);
    \foreach \i in {1,2,3,4,5}{
    \path[draw, thin] (\i,0)--(\i,4);}
    \foreach \j in {0,1,2,3,4}{
    \path[draw, densely dotted] (0,\j)--(6,\j);} 
    \foreach \j in {0,1,2,3,4}{ 
    \path[draw, thin] (1,\j)--(5,\j);} 
    \path[draw, very thick]  (0,0)--(6,0);
    \path[draw, very thick]  (0,4)--(6,4);
    \draw (2.5,4.5) node {$j$};
    \draw (1.5,2.5) node {$c$};
    \draw (1.5,1.65) node {$0$};
    \draw (1.5,1.25) node {\tiny{safe}};
    \draw (2.5,3.5) node {$c$};
    \draw (2.5,2.5) node {$0$};
    \draw (2.5,1.5) node {$0$};
    \draw (2.5,0.5) node {$c$};
    \draw (3.5,3.5) node {$0$};
    \draw (3.5,2.5) node {$c$};
    \draw (3.5,1.5) node {$0$};
    \draw (3.5,0.5) node {$0$};
    \draw (4.5,3.5) node {$0$};
    \draw (4.5,2.5) node {$0$};
    \draw (4.5,1.5) node {$0$};
    \draw (4.5,0.5) node {$c$};
    \draw (5.5,1.5) node {$c$};
    \draw (3.5,-0.7) node {Configuration $\Delta'$};
    \end{scope}
    
    \begin{scope}[xshift=37cm, yshift=0cm, scale=2.5]
    \fill[red!20] (3,1) rectangle (4,2);
    \foreach \i in {1,2,3,4}{
    \path[draw, thin] (\i,0)--(\i,4);}
    \foreach \j in {0,1,2,3,4}{
    \path[draw, densely dotted] (0,\j)--(5,\j);} 
    \foreach \j in {0,1,2,3,4}{ 
    \path[draw, thin] (1,\j)--(4,\j);} 
    \path[draw, very thick]  (0,0)--(5,0);
    \path[draw, very thick]  (0,4)--(5,4);
    \draw (2.5,4.5) node {$j$};
    \draw (1.5,2.5) node {$c$};
    \draw (1.5,1.65) node {$0$};
    \draw (1.5,1.25) node {\tiny{safe}};
    \draw (2.5,3.5) node {$c$};
    \draw (2.5,2.5) node {$0$};
    \draw (2.5,1.5) node {$0$};
    \draw (2.5,0.5) node {$c$};
    \draw (3.5,3.5) node {$0$};
    \draw (3.5,2.5) node {$c$};
    \draw (3.5,1.5) node {$0$};
    \draw (3.5,0.5) node {$0$};
    \draw (4.5,1.5) node {$c'$};
    \draw (4.5,0.5) node {$c$};
    \draw (3.1,-0.7) node {Configuration $\Delta'_2$};
    \end{scope}
    
    \begin{scope}[xshift=5cm, yshift=-17cm, scale=2.5]
    \fill[red!20] (2,2) rectangle (3,3);
    \foreach \i in {1,2,3,4}{
    \path[draw, thin] (\i,0)--(\i,4);}
    \foreach \j in {0,1,2,3,4}{ 
    \path[draw, densely dotted] (0,\j)--(5,\j);} 
    \foreach \j in {0,1,2,3,4}{ 
    \path[draw, thin] (1,\j)--(4,\j);}
    \path[draw, very thick] (1,0)--(2,0)--(2,4)--(1,4);
    \path[draw, very thick] (5,0)--(3,0)--(3,4)--(5,4);
    \draw (3.5,4.5) node {$j$};
    \draw (4.5,1.5) node {$c''$};
    \draw (3.5,0.5) node {$c$};
    \draw (3.5,1.5) node {$0$};
    \draw (3.5,2.5) node {$c$};
    \draw (3.5,3.5) node {$c'$};
    \draw (2.5,0.5) node {$0$};
    \draw (2.5,1.5) node {$0$};
    \draw (2.5,2.5) node {$0$};
    \draw (2.5,3.5) node {$0$};
    \draw (1.5,0.5) node {$0$};
    \draw (1.5,1.5) node {$c$};
    \draw (1.5,2.5) node {$0$};
    \draw (1.5,3.5) node {$c$};
    \draw (3,-0.75) node {Configuration $\Lambda$};
    \end{scope}
    
    \begin{scope}[xshift=30cm, yshift=-17cm, scale=2.5]
    \fill[red!20] (2,2) rectangle (3,3);
    \foreach \i in {1,2,3,4,5}{
    \path[draw, thin] (\i,0)--(\i,4);}
    \foreach \j in {0,1,2,3,4}{ 
    \path[draw, densely dotted] (0,\j)--(6,\j);} 
    \foreach \j in {0,1,2,3,4}{ 
    \path[draw, thin] (1,\j)--(4,\j);}
    \path[draw, very thick] (1,0)--(2,0)--(2,4)--(1,4);
    \path[draw, very thick] (6,0)--(3,0)--(3,4)--(6,4);
    \draw (3.5,4.5) node {$j$};
    \draw (4.5,0.5) node {$0$};
    \draw (4.5,1.5) node {$0$};
    \draw (4.5,2.5) node {$0$};
    \draw (4.5,3.5) node {$c$};
    \draw (5.6,2.5) node {\tiny{safe}};
    \draw (3.5,0.5) node {$c$};
    \draw (3.5,1.5) node {$0$};
    \draw (3.5,2.5) node {$c$};
    \draw (3.5,3.5) node {$c'$};
    \draw (2.5,0.5) node {$0$};
    \draw (2.5,1.5) node {$0$};
    \draw (2.5,2.5) node {$0$};
    \draw (2.5,3.5) node {$0$};
    \draw (1.5,0.5) node {$0$};
    \draw (1.5,1.5) node {$c$};
    \draw (1.5,2.5) node {$0$};
    \draw (1.5,3.5) node {$c$};
    \draw (3.5,-0.75) node {Configuration $\Lambda'$};
    \end{scope}

\end{scope}
\end{tikzpicture}

\end{center}
\vspace{1cm}

Configurations $\Lambda_2$ and $\Lambda_2'$ are obtained from Configurations $\Lambda$ and $\Lambda'$, respectively, by replacing exactly one 0 (zero) of column $j-2$ by a color different from $c$, with the additional constraint that column $j-3$ cannot be empty.

All configurations have their symmetrical counterpart according to the horizontal axis.
However, configurations $\Lambda, \Lambda', \Lambda_2$ and $\Lambda'_2$ are only defined for column $j$ being their left border, and, in configuration $\Delta$, column $j+2$ will always be its right border.

\end{document}